\documentclass[11pt]{article}
\usepackage[margin=1in]{geometry}
\usepackage{authblk}
\usepackage{footmisc} 
\usepackage[table, dvipsnames]{xcolor}
\definecolor{NCSUred}{RGB}{204,0,0}
\definecolor{NCSUyellow}{RGB}{250,200,0}
\definecolor{NCSUorange}{RGB}{209,73,5}
\definecolor{NCSUdarkred}{RGB}{153,0,0}
\definecolor{NCSUaqua}{RGB}{0,132,115}
\definecolor{NCSUgreen}{RGB}{111,125,28}
\definecolor{NCSUblue}{RGB}{65,86,161}
\definecolor{NCSUlightblue}{RGB}{66,126,147} 
\usepackage{graphicx, url, caption, subcaption, cite, booktabs, enumitem}
\usepackage[colorlinks=true, linkcolor=NCSUaqua, urlcolor=NCSUorange, citecolor=NCSUblue]{hyperref}
\usepackage{amsmath, amsfonts, amssymb, amsthm}
\allowdisplaybreaks
\theoremstyle{plain}
\newtheorem{problem}{Problem} 
 
\newtheorem{definition}{Definition}[section]
\newtheorem{assumption}[definition]{Assumption} 
\newtheorem{theorem}[definition]{Theorem}
\newtheorem{lemma}[definition]{Lemma}
\newtheorem{proposition}[definition]{Proposition}
\newtheorem{corollary}[definition]{Corollary}

\newtheorem{remark}[definition]{Remark} 

\RequirePackage[T1]{fontenc}
\usepackage{libertinus, libertinust1math}
\usepackage[cal=rsfso, calscaled=0.95]{mathalpha}
\newcommand{\mc}[1]{\mathcal{#1}}
\newcommand{\ms}[1]{\mathsf{#1}}
\newcommand{\mr}[1]{\mathrm{#1}}

\newcommand{\mbb}[1]{\mathbb{#1}}
\newcommand{\rg}[1]{\mathring{#1}}
\newcommand{\mR}{\mathbb{R}}
\newcommand{\mC}{\mathbb{C}}
\newcommand{\mN}{\mathbb{N}}
\newcommand{\xD}[1]{\mr{d} #1}
\newcommand{\xDD}[2]{\frac{\xD{#1}}{\xD{#2}}}

\newcommand{\bra}[1]{\left( #1 \right)}
\newcommand{\Bra}[1]{\left[ #1 \right]}
\newcommand{\BRA}[1]{\left\{ #1 \right\}}
\newcommand{\norm}[1]{\left\| #1 \right\|}
\newcommand{\abs}[1]{\left\vert  #1 \right\vert}
\newcommand{\ip}[2]{\left\langle #1, \, #2 \right\rangle}
\newcommand{\ipp}[3]{\left\langle #1 \right\vert #2 \left\vert #3 \right\rangle}

\newcommand{\ve}{\varepsilon}

\renewcommand{\Re}{\mr{Re}\,}

\DeclareMathOperator{\tr}{tr}
\DeclareMathOperator{\diag}{diag}
\DeclareMathOperator{\Diag}{Diag}

\title{Koopman--Lyapunov and Koopman--Riccati Controller Synthesis}
\author{Wentao Tang \thanks{Corresponding author: \href{mailto:wtang23@ncsu.edu}{\tt wtang23@ncsu.edu} }}
\affil{Department of Chemical and Biomolecular Engineering, North Carolina State University}
\date{\today}

\begin{document}
\maketitle
\begin{abstract} 
    While the Koopman operator-based lifted representations of nonlinear autonomous and controlled systems have been well covered in the literature, the direct controller synthesis problem largely remains open. 
    As the controlled system is generally a bilinear (instead of fully linear) Koopman model, controller synthesis appears to lose the linear--quadratic structures of the classical Lyapunov and Riccati methods for finite-dimensional linear systems. 
    However, in \emph{infinite dimensions}, bilinearity in the feature map should be merely artifactual and hence reducible to linearity in principle. 
    In this work, by considering the operator semigroups in the reproducing kernel Hilbert space with a linear--Sobolev kernel and utilizing the Yosida approximation of their generators, the synthesis problems of a stabilizing controller and an input-unconstrained optimal controller can be formulated as an \emph{operator Lyapunov inequality} and an \emph{operator Riccati equation}, respectively. 
    These formulations recover the desirable linear--quadratic structures, and hence, with sample data, they are solvable efficiently in analogous ways to usual linear systems by semidefinite programming solvers and Kleinman iterations, respectively. 
    In addition to proving control performance bounds dependent on the sampling denseness and Yosida regularization, through numerical comparison to linearized LQR, bilinear EDMD-based synthesis, nonlinear MPC, and Koopman MPC methods on numerical models and process control examples, this work demonstrates the proposed Koopman--Riccati synthesis as a computationally efficient approach to obtain a near-optimal state-feedback control law. 
\end{abstract}

\newpage 

\section{Introduction} 
The operator-theoretic framework, comprising the Koopman operator, the Perron--Frobenius operator, and akin concepts, has become increasingly well recognized as a standard pathway to the modeling and analysis of nonlinear systems \cite{mauroy2020koopman, mezic2021koopman, otto2021koopman, brunton2022modern}. This framework is built on two cornerstones. 
\begin{itemize}
    \item \emph{Operator theory.} The composition operator on a state-dependent function space\footnote{A state-dependent function in the chosen function space is called an ``observable'' in the literature. The author tends not to use this terminology to avoid confusion with observability and observer synthesis topics in control theory.} 
    or the transition operator on a probability measure space, if well-defined, can be used to ``lift'' a nonlinear system in finite dimensions into a linear system in an infinite-dimensional space. The richness of nonlinear system behaviors (convergence, ergodicity, etc.) is therefore embodied in the spectrum of the operator defined on the chosen function space \cite{lan2013linearization, korda2020data, das2020koopman, colbrook2024rigorous}. 
    For continuous-time systems and especially stochastic systems, it is more appropriate to use operator semigroups (parameterized by time) and consider their generators \cite{vcrnjaric2020koopman, susuki2021koopman, meng2026resolvent}. 
    \item \emph{Operator learning.} The learning of Koopman operator, Perron--Frobenius operator, or corresponding generators in continuous time from data can be formulated as least squares problems or regularized ones. The intuitive idea originated from the extended dynamic mode decomposition (EDMD) method, aiming to find the transition matrix between predecessor--successor data, transformed by nonlinear features \cite{bagheri2013koopman, williams2015data}. 
    The use of kernel-based learning instead of pre-assigned nonlinear features (known as kernel EDMD) suggests the definition of a reproducing kernel Hilbert space (RKHS) \cite{kevrekidis2016kernel}. 
    The analysis of kernel-based learning error was conducted with the help of statistical learning theory \cite{klus2020eigendecompositions, kostic2022learning, nuske2023finite}; for example, based on spectral analysis that assesses the limitation of finite data-based operator estimation after compactification into a space of weaker topologies, the standard analysis in Kostic \cite{kostic2023sharp} provided a sharp probabilistic error bound. 
\end{itemize} 

\paragraph{Koopman operator on RKHS and Sobolev-type RKHS.} 
The importance of choosing a space that guarantees the invariance of the Koopman operator (semigroup) is self-evident --- without invariance, the modeling error tends to be hard to quantify and its effect on dynamics is hard to predict \cite{haseli2026modeling}.
Although it may be beneficial to know some particular system structures that allow a finite-dimensional lifting \cite{shang2026existence}, systems in general cannot be fully linearized on a finite dictionary basis. Hence, RKHSs have been commonly used in the literature \cite{hou2023sparse, bevanda2024nonparametric, hou2026behavioral}, at least as a means to reconcile theoretical rigor and ease of learning. 
In K{\"{o}}hne et al. \cite{kohne2025error}, it was proved that a \emph{Sobolev-type RKHS} (associated with a Sobolev kernel) is equivalent to a Sobolev--Hilbert space $H^s(\mbb{X})$, and thus accommodates a well-defined, bounded linear Koopman operator, as long as the dynamics have a correspondingly high smoothness index; the fact that the Sobolev kernel gives an RKHS on $\mR^d$ and any regular bounded region $\mbb{X}\subset \mR^d$ is well known in kernel-based learning \cite{wendland2004scattered}. 
As also proved in \cite{kohne2025error}, if the Koopman operator is learned from the Sobolev-type RKHS, its error, when applied to predicting any RKHS function, is uniformly bounded. The error can be improved to be proportional to $|x|$ if the origin is forced to be sampled \cite{bold2025kernel}. 

\paragraph{Linear--Sobolev kernel and linear--Sobolev RKHS.} 
Realizing that a radial kernel, such as the Sobolev kernel $\kappa_{\mr{Sob}}(x, x')\propto \rho(|x-x'|)$, is not suitable for modeling a system with an equilibrium (at the origin)\footnote{More specifically, with any nontrivial radial kernel, the corresponding canonical feature map $\phi$, which ``lifts'' any $x$ into the RKHS, satisfies $\|\phi(x)\|= \rho(0)^{1/2}=\mr{const.}>0$. Thus, there is no longer an ``origin point'' after lifting, which is inconvenient for stability analysis and controller synthesis.}, 
the author \cite{tang-ye2025koopman} proposed to use a ``linear--radial product kernel'' instead. This kernel is the product kernel of the linear kernel $\kappa_{\mr{lin}}(x,x')=x\cdot x'$ and the Sobolev kernel $\kappa_{\mr{Sob}}$. All the functions in the resulting RKHS are ``locally at least linear'' and ``globally regular in the Sobolev sense''. (The precise meaning can be found in Section \ref{sec:preliminaries}.) 
This new RKHS definition brings the following beneficial advances. 
\begin{enumerate}[nolistsep]
    \item The Koopman spectrum--stability relation is found to be analogous to linear systems, i.e., the asymptotic stability of the origin in discrete time is reflected by a Koopman spectrum strictly confined in the unit disk of the complex plane \cite{tang-ye2025koopman}. 
    \item Using kernel quadratic and kernel sum-of-squares (SOS) forms, Lyapunov functions and their decay rates \cite{tang-ye2026koopman}, as well as the storage function and a corresponding supply rate function in the dissipation inequality \cite{tang-ye2026dissipativity}, can be represented, thus formulating Lyapunov stability and dissipativity certificates as operator linear inequalities.
    \item The data-driven estimation of a Zubov function, which solves Zubov's partial differential equation in a viscosity sense and thereby characterizes the maximal domain of attraction of an (autonomous) nonlinear system, can be reduced to a least-squares problem, where the Zubov function (by a suitable nonlinear transformation) is written as $1$ plus a function in this RKHS \cite{tang2026data}. 
    \item The synthesis of state observers for nonlinear systems that can be measured only in the output ($\dot{x} = f(x)$, $y=h(x)$) was formulated as the solution of an \emph{operator Lyapunov inequality} in a data-based empirical version \cite{tang2026koopman}. Concepts of Koopman observability and detectability were used to characterize the existence of such a Koopman--Luenberger observer. 
\end{enumerate} 

\paragraph{Problem of interest: Koopman-based control.}
The popularity of Koopman operator among control theorists was largely motivated by the hope of using linear controller synthesis approaches for nonlinear systems. In the paper of Korda and Mezi{\'{c}} written a decade ago \cite{korda2018linear}, it was claimed that ``\textit{these linear predictors can be \emph{readily} used to design controllers for the nonlinear dynamical system using linear controller design methodologies.}''\footnote{The linear predictor here refers to a lifted but finite-dimensional discrete-time model of the form $x^+=Ax+Bu$, proposed to be learnable via least squares in Proctor et al. \cite{proctor2018generalizing}. The non-existence of such a linear predictor for general nonlinear systems, however, makes such a natural idea appear to be oversimplified and become suboptimal.} 
However, it was clear that with input variables, the operator model can at simplest be \emph{bilinear} --- when differentiating or differencing any function of state, the input-dependent term generally still involves the state. 
This fact was pointed out in \cite{goswami2021bilinearization, iacob2024koopman} based on assuming closure conditions on Koopman eigenfunctions, and further discussed in Bevanda et al. for its nonparametric learning in RKHSs \cite{bevanda2024nonparametric}. Recently in \cite{morris-ye-tang2026}, it was proved that with the RKHS formulation, when the discrete-time dynamics satisfy certain regularity conditions, the Koopman model is \emph{exactly bilinear}. For continuous-time input-affine systems, it is not hard to realize that the bilinear model does not require any nonlinear features of the input variable \cite{williams2016extending, peitz2020data}. 

Now that the Koopman control model appears to be not fully linear, is it still possible to control in an analogous way to linear systems as originally conceived in \cite{korda2018linear}? More concretely, can we still expect to have a Lyapunov equation/inequality for stabilization and a Riccati equation for linear--quadratic regulator (LQR) synthesis?

\subsection{Relevant Works}
Since \cite{korda2018linear}, there have been a large number of works that explored the use of EDMD or Koopman operators in control, e.g., of robotics \cite{bruder2020data, shi2022deep}, vehicles \cite{xiao2022deep}, and chemical processes \cite{narasingam2019koopman, han2025deep, shi2026data}. 
A common limitation is that due to the obscurity of the underlying true Koopman model, the error in an arbitrarily chosen approximator, such as a neural network \cite{wang2022koopman, zhao2024deep}, is not analyzable in terms of its effect on Koopman modeling error and consequently the control performance degradation. 
Hence, many existing Koopman-based control schemes are not considered as rigorously formulated. 
In this section, relevant works are reviewed; however, with the voluminous literature under the theme of Koopman operators, this review can only be non-exhaustive. 

\paragraph{Koopman-based controller synthesis for bilinear Koopman models.} 
The representative works of Str{\"a}sser et al. \cite{strasser2024koopman, strasser2025kernel} considered the controller analysis for bilinear systems $\phi(x^+) = A\phi(x)+B_0u+ \sum_{j=1}^{d_u} u_jB_j\phi(x) + r(x,u)$ with a feature map $\phi$, in which the modeling error $r(x,u)$ is bounded by $|r(x,u)|\leq c_x|\phi(x)| + c_u|u|$.\footnote{
This \emph{sectorial} bound was established by a Taylor polynomial argument when $\phi(x)$ is chosen to be a finite-dimensional feature that preserves the origin and has sufficient smoothness, and when $\phi$ is the infinite-dimensional canonical feature of a Sobolev kernel, the bound comes from forcing the origin into the sample \cite{bold2025kernel}. The bound remains valid when taking into account the stochasticity of data samples \cite{nuske2023finite}. 
With the linear--Sobolev kernel, the sectorial bound can be naturally established in an alternative way, considering the projection onto a finite-dimensional lifted space and the sample stochasticity \cite{ye2026}. } 
Upon the imposition of a feedback law $u=K\phi(x)$, the closed-loop system contains a bilinear term in the feature map $\phi(x)$. 
To handle this bilinearity, in the SafEDMD framework of Str{\"a}sser et al.  \cite{strasser2024koopman, strasser2025kernel}, an additional ellipsoidal constraint on the vector-valued $\phi(x)$ is introduced, thus regarding the bilinear term as an uncertain linear term that is amenable to robust controller synthesis with guarantees. 
However, this treatment is at least unnatural. Because the bilinearity in the Koopman model is not a genuine source of uncertainty, the introduction of a range constraint not only restricts the applicability to a (somewhat obscurely) \textit{a priori} assigned region and increases the conservativeness of the controller, but also rejects the possibility of obtaining optimal closed-loop properties (e.g., global stabilization would not be achievable). 
Readers are recommended to consult the recent review by Str{\"a}sser et al. \cite{strasser2026overview} and a tutorial by Mezi{\'c} et al. \cite{mezic2026koopman}. 

\paragraph{Koopman model predictive control (MPC).} 
When directly synthesizing a control law is difficult for a nonlinear system, or if the system has input/state constraints, MPC is a commonly used approach to determine control inputs recursively in discrete time, based on optimizing the future trajectory in a horizon. Hence it becomes attractive to use a Koopman model for MPC from a practical perspective \cite{son2021application, de2024koopman, zhong2025data, yan2026data}. 
When using a bilinear \cite{bold2025kernel, schimperna2026data} or even a linear Koopman model as the surrogate \cite{shang26exponential} in MPC, with appropriate bounds on the modeling errors, recursive feasibility and closed-loop stability have been studied.  
However, when the Koopman model used in the MPC is bilinear \cite{schimperna2026data}, it remains doubtful whether identifying a high-order bilinear Koopman model provides advantage over a low-order fully nonlinear first-principles or empirical (e.g., NARX or Volterra) model, while a linear Koopman model in principle contains a mismatch. 
From a theoretical point of view, the closed-loop guarantees of MPC hinge on a synthesis problem --- the terminal cost, terminal constraint, and input constraint must admit a hypothetical terminal feedback law that satisfies a Lyapunov inequality. While in practice, MPC is often posed by ignoring the nonlinear system's synthesis problem or degrading it to an assumption, the synthesis problem as the prerequisite issue ought not to be overlooked. 

\paragraph{Koopman--Nemytskii operator.} 
In the author's prior work \cite{tang2025koopman}, the Koopman operator for autonomous systems was extended to an operator for closed-loop systems, called Koopman--Nemytskii operator, that maps the state's feature and the control policy's feature (in an RKHS) to the successor state's feature. For any given policy (feedback law), its performance (i.e., the accumulated cost) can thus be evaluated by predicting future states with the learned Koopman--Nemytskii operator. 
This brings the controller design problem to a kernel-based optimization problem, namely that of finding the policy with the optimal feature, for which an efficient algorithm was given in Rudi et al. \cite{rudi2025finding}. 
The advantages over Koopman MPC, both in computational time and in achieved optimality, were demonstrated in \cite{tang2025koopman}. However, the approach relies on defining the control policy's feature and hence requires a parametric family of policies; this typically applies to low-dimensional systems with partial prior knowledge. 
In this work, seeking a general \emph{synthesis} approach, we instead consider nonparametric policies. 

\paragraph{Koopman operators in optimal control and stochastic optimal control.} 
For open-loop optimal control problems, Villanueva et al. defined a Pontryagin--Koopman operator to lift the dynamics of state and co-state in Pontryagin's minimum principle, and characterized optimal feedback through the symplectic spectral structure \cite{villanueva2021towards}. In a similar spirit, Vaidya \cite{vaidya2025koopman} related Koopman eigenfunctions to the Hamilton--Jacobi--Bellman (HJB) equation. 
For stochastic systems, Houska \cite{houska2025convex} considered the Fokker--Planck--Kolmogorov dynamics of the state/input density under control and hence formulated the optimal control problem as a convex one in the density. Bevanda et al. made this framework data-driven by learning the generator of the controlled diffusion in an RKHS \cite{bevanda2025kernel}. For deterministic systems, since the density is not bounded away from zero, the density-based formulation has to exclude a neighborhood of the origin and hence only a practical stability guarantee is attainable \cite{huang2022convex}. 
A third line of work learns the Koopman generator to recover the vector fields and then solves the HJB equation by classical means, where the Koopman generator is used during the iterations for HJB solution \cite{zeng2025datadriven, chen2026datadriven}. 
We also note the works of Breiten and H{\"{o}}veler \cite{breiten2023approximability, breiten2025unifying}, where by considering any asymptotically stable system with a known weighting function $w$ that characterizes its convergence rate, Koopman operator-based Lyapunov conditions and Riccati equations have been formulated in weighted $L^\infty$ spaces based on the investigation of quadratic forms; on the other hand, a direct synthesis formulation was missing. 

\bigskip
In addition, the optimal control theory for infinite-dimensional systems in a fully linear form ($\dot{x} = Ax + Bu$) has been well studied in the literature \cite{li1995optimal, bensoussan2007representation, curtain2020introduction, tucsnak2009observation}. The problem of interest in this paper is with regard to a bilinear infinite-dimensional system. 
For finite-dimensional linear systems, data-driven control taking an ``informativity'' approach posed the interesting problem of whether a well-performing controller can be found without enough data for identifying a full model \cite{van2020data}. This work shares a similar philosophy; the controller synthesis approach in this work is Koopman-based, but without explicitly identifying a Koopman model.

\subsection{Contributions of This Work}
In this work, the direct controller synthesis problem is considered for a continuous-time nonlinear system in the input-affine form: 
\begin{equation}\label{eq:system}
    \dot{x}(t)=f_0(x(t)) + \sum_{j=1}^{d_u} u_j(t) f_j(x(t)). 
\end{equation}
Assuming that the origin $x=0$ is an equilibrium point under zero input and assuming sufficient smoothness on the vector fields $f_0$, $f_1, \cdots, f_{d_u}$, when an equilibrium-consistent feedback policy closes the loop, the Koopman semigroup is well-defined on the RKHS with a linear--Sobolev kernel ($\mc H_{\rg\kappa}(\mbb X)$) and its generator is expressed in terms of the Koopman semigroups associated with $f_0$, $f_1, \cdots, f_{d_u}$ and a multiplication operator with the policy components. 
On this premise, we formulate the stabilizing controller and optimal controller synthesis problems. 
\begin{enumerate}[nolistsep]
    \item For stabilization, a control Lyapunov function in a kernel sum-of-squares form, therefore a ``locally at least quadratic'' and ``globally Sobolev-regular'' function of state, is sought to satisfy a decaying condition, with an assigned decay rate function being in a kernel quadratic form. This Lyapunov condition is first expressed as a pointwise bilinear \emph{operator Lyapunov inequality}. 
    \item For optimal control, the cost function under any given policy is expressed as a kernel sum-of-squares, and therefore is optimizable over the set of stabilizing policies that ensure exponential stability of the closed-loop semigroup. The optimization problem results in an optimality condition, which, upon the finite-time Yosida approximation of Meng et al. \cite{meng2026resolvent}, gives a formula for the optimal policy and an \emph{operator Riccati equation} for the optimal cost. 
    \item With suitable transformations, the bilinear terms in the feature maps are absorbed in the infinite dimensions, thus eliminating the bilinearity or multilinearity in the data-approximated synthesis problems. With data, the operator Lyapunov inequality and Riccati equation are approximated in finite dimensions, resulting in a convex semidefinite programming (SDP) problem and a quadratic matrix equation that can be solved by classical Kleinman iterations, respectively. The idea contained in the approach to resolve the bilinearity is simple --- \emph{an infinite-dimensional RKHS is an algebra by itself.}
    \item Building on existing ingredients of learning theory and methods, including the techniques used in scattered data approximation theory of Wendland et al. \cite{wendland2004scattered} and Yosida approximation \cite{meng2026resolvent}, the degradation of control performance due to the use of sample data is analyzed. The error bound depends on the growth rate of the closed-loop semigroup, Yosida approximation constants, and sampling density. 
    \item Comparisons to existing approaches, including na{\"i}ve linear Koopman model-based control, the robust control approach of Str{\"a}sser et al. \cite{strasser2024koopman}, Sobolev kernel-based Koopman MPC \cite{schimperna2026data} and the linear--Sobolev kernel counterpart of Koopman MPC (examined in \cite{tang2025koopman}), as well as the LQR using linearized model, and nonlinear MPC using the true model, are conducted on several numerical and process control examples. 
\end{enumerate}
In total, the paper contributes to a unifying approach of Koopman-based controller synthesis for nonlinear systems (in continuous-time input-affine forms). The approach can be considered as theoretically rigorous with provable probabilistic performance bounds, and in the meantime, due to the high similarity of its formulations to finite-dimensional linear systems (with order equal to the sample size), it is naturally efficient in computation. 
It should be emphasized that the approach is completely different from any ``explicit MPC'' or analogous surrogate of an MPC solver. While the approach is conceptually related to reinforcement learning (RL), the use of kernel method to convert the problem to linear formulations makes the solution free from using neural function approximators within the RL routine.

\subsection{Notation}
We use lower-case Latin and Greek letters for scalar and vector-valued objects, and upper-case letters for matrices and operators. Matrix approximations of operators will use \textsf{Sans Serif} fonts (e.g., $\ms{P}$ for $P$). $\mN=\{0, 1, 2, \cdots\}$, $\mN_+=\mN \setminus\{0\}$, $\mR_+=[0, \infty)$, $\mR^d$, and regions or intervals thereon use blackboard fonts. The component functions $x\mapsto x_k$ are denoted as $e_k$, while the standard basis vectors of $\mR^d$ are denoted as $\ms e_k$ ($k=1,\cdots,d$) in \textsf{Sans Serif} for the sake of distinction. 
\par Function spaces are denoted by normal ($H^s(\mbb{X})$ -- Sobolev space, namely $W^{s,2}(\mbb X)$) or calligraphic upper-case letters (e.g., $\mc{H}_\kappa$: an RKHS with kernel $\kappa$). Operator spaces use calligraphic letters (e.g., $\mc{B}$ -- space of bounded operators, $\mc{B}_2$ -- space of Hilbert--Schmidt operators on a Hilbert space). The domain of an unbounded operator $A$ is written as $\mc{D}(A)$. 
Norms are denoted by $\|\cdot\|$ or $|\cdot|$ (in $\mR^d$), inner products are $\ip{\cdot}{\cdot}$, and outer product between two vectors on a Hilbert space (which forms a rank-$1$ operator) is denoted by $\times$, i.e., $(z_1\times z_2)z_3 = \ip{z_2}{z_3}z_1$. This avoids the confusion with $\otimes$, the Kronecker product between two vectors defined by the relation $\ip{z_1\otimes z_2}{z_3\otimes z_4} = \ip{z_1}{z_3}\ip{z_2}{z_4}$. 
The spectrum of an operator or a matrix $A$ is denoted as $\sigma(A)$ and the spectral radius is written as $r(A)$. 

\par A positive (semi-)definite matrix or positive (nonnegative) operator is indicated by $P\succ(\succeq)0$, and hence $P\succeq Q$ indicates $P-Q\succeq 0$ (and similarly for $\succ$, $\preceq$, and $\prec$). A multi-index $\beta$ refers to a tuple $(\beta_1, \cdots, \beta_d) \in \mN^d$, the partial derivative $\partial^\beta h = \partial^{|\beta|} h /\partial x_1^{\beta_1}\cdots \partial x_d^{\beta_d}$, where the length of the multi-index $|\beta|=\beta_1+\cdots+\beta_d$. 
For $a, b\geq0$, we write $a\lesssim (\propto) b$ to mean that $a\leq(=) \mr{const}\cdot b$.

\section{Preliminaries}\label{sec:preliminaries}
Consider the continuous-time input-affine system \eqref{eq:system}, with a $d_x$-dimensional state and a $d_u$-dimensional input. Focusing on the controller synthesis problem, without loss of generality, we let the origin be an equilibrium point of interest to be regulated at, i.e., $f_0(0)=0$. 
We assume that the states always evolve on a \emph{compact} region $\mbb{X} \subset \mR^{d_x}$.\footnote{While the invariance of $\mbb X$ appears to be restrictive, it is achievable by multiplying all the vector fields $f_0$, $f_1, \cdots, f_{d_u}$ with a smooth bump function $\chi$, which satisfies $\chi(x)\equiv 1$ for all $x$ in a closed ball $\overline{\mbb B}_1$ (centered at the origin) and $\chi(x)\equiv 0$ outside a larger ball $\overline{\mbb B}_2$. Thus, the modified dynamics is invariant on $\overline{\mbb B}_2$, and there is no modification on $\overline{\mbb B}_1$. As the difference between these two balls can be made arbitrarily small, there will be no substantial effect on the learning procedures.}
More regularity assumptions will be made as we navigate through the fundamentals of Koopman operator theory for system \eqref{eq:system}. 

\subsection{Koopman and Perron--Frobenius Semigroups}
\par The vector field $f_0$ on $\mbb{X}$ gives an operator $L_0: g\mapsto \mr{D}g\cdot f_0$, which is the infinitesimal generator of the following Koopman semigroup, if such a semigroup is strongly continuous on the function space on which it is defined:
\begin{equation}\label{eq:Koopman.semigroup}
    K_0^t: \mc{G}\to \mc{G}, \enspace g\mapsto g\circ S_{f_0}^t, 
\end{equation}
namely $K_0^tg(x)=g(S_{f_0}^t(x))$. Here $S_{f_0}^t$ is the flow map of the autonomous system $\dot{x} = f_0(x)$, mapping any initial state $x(0)$ to the state after a time duration of $t$. 
As long as $f_0$ defines a global flow on $\mbb X$ (i.e., the solution is well-defined at all $t\geq0$ starting from any $x\in \mbb{X}$), then $\{S_{f_0}^t\}_{t\geq 0}$ is a semigroup of mappings under the composition operation (namely $S_{f_0}^0 = \mr{id}$ and $S_{f_0}^{t_2}\circ S_{f_0}^{t_1} = S_{f_0}^{t_1+t_2}$ for all $t_1,t_2\geq0$), and hence $\{K_0^t\}_{t\geq 0}$ is a semigroup of linear operators. 
For the Koopman semigroup to be a well-defined strongly continuous semigroup, we should assume certain regularity on the vector field $f_0$ and choose the defining function space $\mc{G}$ in concert. 

\begin{definition}[Spaces of Sobolev functions and smooth functions]
    For $s\in \mbb{N}$, the $s$-th Sobolev--Hilbert space $H^s(\mbb{X})$ is the collection of generalized functions\footnote{Here we only focus on real-valued function spaces. Whenever their spectral properties are mentioned, which possibly involve complex numbers, we tacitly refer to their complexified versions. This will not cause any confusion.}
    (namely tempered distributions) such that all generalized partial derivatives up to order $s$ exist in $L^2(\mbb{X})$. That is, 
    $$H^s(\mbb{X}) = \BRA{ g:\mbb{X}\to\mR \ : \ \partial^\beta g \in L^2(\mbb{X}), |\beta|\leq s}. $$
    The inner product and its induced norm are defined as:
    $$ \ip{g}{h}_{H^s} = \sum_{|\beta|\leq s} \int_{\mbb{X}} \partial^\beta g(x) \partial^\beta h(x) \xD{x}, \quad 
    \norm{g}_{H^s} = \Bra{ \sum_{|\beta|\leq s} \int_{\mbb{X}} |\partial^\beta g(x)|^2 \xD{x}}^{1/2} $$
    Similarly, $C^s(\mbb{X})$ is the collection of functions with continuous partial derivatives up to order $s$:
    $$C^s(\mbb{X}) = \BRA{ g:\mbb{X}\to\mR \ : \ \partial^\beta g \in C(\mbb{X}), |\beta|\leq s}. $$
    It is a Banach space with norm:
    $$ \norm{g}_{C^s} = \max_{|\beta|\leq s} \max_{x\in \mbb{X}} |\partial^\beta g(x)|. $$
    For fractional $s>0$, writing $s = (1-\theta)\lfloor s \rfloor + \theta \lceil s \rceil = \lfloor s \rfloor+\theta$ for some $\theta\in(0,1)$, $H^s(\mbb{X})$ is the $\theta$-th interpolation space between $H^{\lfloor s \rfloor}(\mbb{X})$ and $H^{\lceil s \rceil}(\mbb{X})$, namely the Slobodecky--Sobolev spaces, and $C^s(\mbb{X})$ is the $\theta$-th interpolation space between $C^{\lfloor s \rfloor}(\mbb{X})$ and $C^{\lceil s \rceil}(\mbb{X})$, namely the H{\"{o}}lder--Zygmund spaces. 
\end{definition}
In the sequel, we only consider the case of $s>d_x/2$ and assume that $\mbb{X}$ has a Lipschitz boundary. When so, according to the Sobolev embedding theorem and extension theorem \cite{adams2003sobolev}, $H^s(\mbb{X})$ is continuously embedded into $C(\mbb{X})$, and therefore any member of $H^s(\mbb{X})$ is genuinely a function (instead of only a tempered distribution).
In view of the equilibrium point condition: $f_0(0)=0$, the following variants of the Sobolev space $H^s(\mbb X)$ and smooth function space $C^s(\mbb X)$ are considered. 
\begin{definition}[Spaces of ``linear--Sobolev'' and ``linear--smooth'' functions]
    Writing $e_k: \mbb{X}\to\mR$, $x\mapsto x_k$ as the component projections $(k=1,\cdots,d_x)$, we denote 
    $$\rg H^s(\mbb{X}) = \BRA{\sum_{k=1}^{d_x} e_kg_k \ : \ g_k\in H^s(\mbb{X}), \forall k=1,\cdots,d_x}, $$
    which is a Hilbert space endowed with an inner product, and hence induced norm
    $$\ip{\sum_{k=1}^{d_x} e_kg_k}{\sum_{k=1}^{d_x} e_kh_k}_{{\rg H}^s} = \sum_{k=1}^{d_x} \ip{g_k}{h_k}_{H^s},
    \quad \norm{\sum_{k=1}^{d_x} e_kg_k}_{{\rg H}^s(\mbb{X})} = \Bra{ \sum_{k=1}^{d_x} \norm{g_k}_{H^s(\mbb{X})}^2 }^{1/2}. $$ 
    Also, we write
    $$\rg C^s(\mbb{X}) = \BRA{\sum_{k=1}^{d_x} e_kg_k : g_k\in C^s(\mbb{X}), \forall k=1,\cdots,d_x}, $$
    endowed with the following norm\footnote{A given function in $\rg H^s(\mbb{X})$ or $\rg C^s(\mbb{X})$ may have non-unique representations as a linear combination of component functions multiplied by $H^s$- or $C^s$-functions. For example, for $g(x)=x_1x_2$, $g\in \rg C^s(\mbb{X})$. If viewed as $e_1g_1$ with $g_1 = e_2$, then $\|g\|_{\rg C^s} = \sup_{x\in \mbb{X}}|x_2|$, while if viewed as $e_2g_2$ with $g_2=e_1$, then $\|g\|_{\rg C^s} = \sup_{x\in \mbb{X}}|x_1|$. To be precise, the norm should be defined as the maximum among all these representations. In this example, $\|g\|_{\rg C^s} = \max \BRA{\sup_{x\in \mbb{X}} |x_1|, \sup_{x\in \mbb{X}} |x_2|}$. }, which makes ${\rg C}^s(\mbb{X})$ a Banach space:
    $$\norm{\sum_{k=1}^{d_x} e_kg_k}_{{\rg C}^s} = \max_{k=1,\cdots,d_x} \|g_k\|_{C^s}.$$ 
\end{definition}
We interpret any member of $\rg H^s(\mbb X)$ as a function that is ``locally at least linear'' (since it contains component functions as factors of its constituent terms) and ``globally Sobolev-regular'' (since component functions clearly belong to $H^s(\mbb{X})$ and a $\rg H^s(\mbb X)$-function is automatically an $H^s(\mbb X)$-function). 
 
The following conclusion, which relates the smoothness of any vector field to the strong continuity of the Koopman semigroup that it induces, is natural. A proof can be found, e.g., in \cite{ye2026}. 
\begin{lemma}[Strong continuity of the Koopman semigroup]
    Suppose that $\mbb{X} \subset \mR^{d_x}$ is compact. 
    If $f\in [C^s(\mbb{X})]^{d_x}$, the semigroup $\{K^t\}_{t\geq 0}$ with infinitesimal generator $L: g\mapsto \mr{D}g\cdot f$ is a strongly continuous semigroup on $\mc{G} = H^s(\mbb{X})$. 
    If $f\in [\rg C^s(\mbb{X})]^{d_x}$, then the semigroup is a strongly continuous semigroup on $\rg{\mc{G}} = \rg H^s(\mbb{X})$. 
\end{lemma}
Hence, to utilize the conclusion of the lemma, we make the following assumption. 
\begin{assumption}[Regularity of the system]
\label{assum:regularity}
    The system \eqref{eq:system} is invariant on a compact region $\mbb{X} \subset \mR^{d_x}$ with a Lipschitz boundary. The vector fields are sufficiently smooth: $f_0\in [\rg C^s(\mbb{X})]^{d_x}$ and $f_j\in [C^s(\mbb{X})]^{d_x}$ for every $j=1,\cdots,d_u$, with some $s>d_x/2$. 
\end{assumption}
\begin{corollary}[Koopman generators]
\label{cor:strong.continuity} 
    Under Assumption \ref{assum:regularity}, 
    \begin{enumerate}[nolistsep]
        \item the infinitesimal generator associated with $f_0$, $L_0: g\mapsto \mr{D}g\cdot f_0$ generates a strongly continuous Koopman semigroup $\{K_0^t\}_{t\geq 0}$ on $\rg{\mc{G}} = \rg H^s(\mbb X)$;  
        \item when the inputs are present, for every $j=1,\cdots,d_u$, $L_j: g\mapsto \mr{D}g\cdot f_j$ generates a strongly continuous Koopman semigroup $\{K_j^t\}_{t\geq 0}$ on $\mc{G} = H^s(\mbb X)$; 
        \item under any given $\rg C^s$-class policy $\pi = (\pi_1, \cdots, \pi_{d_u}) \in [\rg C^s(\mbb X)]^{d_u}$, the closed-loop system:
        \begin{equation}\label{eq:system.closed-loop}
            \dot{x}(t) = f_0(x(t)) + \sum_{j=1}^{d_u} \pi_j(x(t))f_j(x(t)) =: f_\pi(x(t)), 
        \end{equation}
        in which $f_\pi = f_0 + \sum_{j=1}^{d_u} \pi_jf_j \in [\rg C^s(\mbb{X})]^{d_x}$, has a corresponding strongly continuous Koopman semigroup $\{K_\pi^t\}_{t\geq 0}$ on $\rg{\mc{G}} = \rg H^s(\mbb X)$, with infinitesimal generator 
        $$ L_\pi: g\mapsto \mr{D}g\cdot f_\pi = \mr{D}g\cdot \bra{f_0 + \sum_{j=1}^{d_u} \pi_jf_j}. $$
    \end{enumerate}
\end{corollary}
We note that due to the input affinity of system \eqref{eq:system}, the closed-loop Koopman generator $L_\pi$ has an affine dependence on the policy $\pi$. To be more precise, we introduce the multiplication operators.
\begin{definition}[Multiplication operator]
    Let $\varpi$ be a function that belongs to $\rg C^s(\mbb{X})$, where $\mbb{X}\subset \mR^{d_x}$ is a compact region. 
    The multiplication operator with factor $\varpi$ refers to $M_\varpi: g\mapsto \varpi g$, which is a bounded linear operator from $H^s(\mbb{X})$ to $\rg H^s(\mbb{X})$. 
\end{definition}
To migrate between the two Hilbert spaces introduced in the foregoing discussions, it is convenient to use the following notions. 
\begin{definition}[Inclusion map from linear--Sobolev space to Sobolev space]
    Denote $\rg \iota: \rg H^s(\mbb{X})\to H^s(\mbb{X})$, $g\mapsto g$ as the inclusion operator that maps any function to itself.
\end{definition}
Thus, we can formally write 
\begin{equation}\label{eq:Koopman.generator.closed-loop} 
    L_\pi = L_0 + \sum_{j=1}^{d_u} M_{\pi_j}L_j \rg\iota \ : \ \mc{D}(L_\pi) := \mc{D}(L_0) \cap \bra{\bigcap_{j=1}^{d_u} \rg\iota^{-1}\mc{D}(L_j)} \subset \rg H^s(\mbb{X}) \to \rg H^s(\mbb{X})
\end{equation} 
and due to the fact that the domains of all $L_j$ have an intersection that contains at least all functions that are $C^1$ (and therefore dense), the above relation is meaningful.

\subsection{Sobolev Kernel and Linear--Sobolev Kernel}
A useful perspective to investigate system \eqref{eq:system} is to view any state $x\in \mbb{X}$ through a feature map $\phi$ into an infinite-dimensional space. In the machine learning literature, the reproducing kernel Hilbert space (RKHS) theory is used commonly \cite{smola1998learning, steinwart2008support}. Specifically, by assigning a kernel and considering the closed span of kernel functions equipped with an inner product structure, the resulting function space, namely the RKHS, has a canonical feature map implied by the kernel. 
\begin{definition}[RKHS]
    A \emph{Mercer kernel} on a nonempty set $\mbb{X}$ is a bivariate continuous function $\kappa: \mbb{X}\times\mbb{X} \to\mR$ such that for any finite number of points $x_1, \cdots,x_n\in \mbb{X}$, the resulting matrix $\ms{G} \in \mR^{n\times n}$ whose $(i,j)$-th entry is $\kappa(x_i,x_j)$ ($i,j=1,\cdots,n$), is positive semidefinite. For any two points $x,x'\in \mbb{X}$, define the inner product $\ip{\kappa(x,\cdot)}{\kappa(x',\cdot)} = \kappa(x,x')$. Then the span of all $\kappa(x,\cdot)$ ($x\in \mbb{X}$) forms an inner product space, which can be completed to result in a Hilbert space:
    $$ \mc{H}_\kappa(\mbb{X}) = \overline{\operatorname{span}} \BRA{\kappa(x,\cdot) \ : \ x\in \mbb{X}}.$$
    This space is referred to as the \emph{reproducing kernel Hilbert space} with kernel $\kappa$ over $\mbb{X}$. 
\end{definition} 

In fact, with a given $\mbb{X}$, the RKHS and its kernel have a one-to-one correspondence \cite{steinwart2008support}. When the kernel is given, we refer to $\phi: \mbb{X}\to\mc{H}_\kappa(\mbb{X})$, $\phi(x)=\kappa(x,\cdot)$ as the \emph{canonical feature map}; hence $\kappa(x,x') = \ip{\phi(x)}{\phi(x')}$, $\forall x,x'\in \mbb{X}$. 
The most important property of an RKHS is the \emph{reproducing property}, which states that the inner product on the RKHS satisfies 
\begin{equation}
    \ip{g}{\phi(x)} = g(x), \ \forall g\in \mc{H}_\kappa(\mbb{X}), \ \forall x\in \mbb{X}.
\end{equation}
As we are considering the closed-loop system \eqref{eq:system.closed-loop} with a smoothness index $s$, the following discussions will be useful, as the Sobolev space $H^s(\mbb{X})$ and linear--Sobolev space $\rg H^s(\mbb{X})$ are known to be RKHSs. 

\begin{lemma}[Sobolev kernel for Sobolev space]
    If $\mbb{X}\subset \mR^{d_x}$ is a compact region with Lipschitz boundary, and $s>d_x/2$, then $H^s(\mbb{X})$ is an RKHS, whose reproducing kernel is 
    $$\kappa(x,x') = \rho(|x-x'|)$$
    where $\rho: \mR_+\to\mR_+$ is any function such that its Fourier transform (e.g., by even-extending $\rho$ to $\mR\to\mR_+$), $\hat\rho: \mC\to\mC$, satisfies $c_1(1+|\xi|^2)^{-s} \leq \hat\rho(\xi) \leq c_2(1+|\xi|^2)^{-s}$ with some constants $c_2\geq c_1> 0$, e.g., the so-called $\nu$-th Mat{\'{e}}rn kernel:
    $$\rho(|x-x'|)=\int_{\mR^{d_x}} \frac{\mr{e}^{\mr{i}\xi \cdot (x-x')}}{(1+|\xi|^2)^s}\xD{\xi} \propto \bra{ \sqrt{2\nu}|x-x'| }^{\nu} \mr{K}_\nu\bra{\sqrt{2\nu}|x-x'|}, \enspace \nu=s-\frac{d_x}{2}, $$
    $\mr{K}_\nu$ being the modified Bessel function of the second kind.\footnote{The kernel can be scaled by a length scale parameter $\ell>0$. That is, by replacing the function $\rho(\cdot)$ with $\rho(\cdot/\ell)$, the RKHS defined remains equivalent.}  
    Such a kernel function, due to its dependence on $|x-x'|$ only, is called a translation-invariant kernel or a radial kernel. When $s>d_x/2$, $\rho\in C^\alpha(\mR_+)$ for all $\alpha<s-d_x/2$. 
    We denote the canonical feature of $x$ in the Sobolev RKHS $\mc{H}_\kappa(\mbb{X})$ as $\phi(x)$. 
\end{lemma}
\begin{remark}[Wendland kernel]
    An alternative construction of the Sobolev kernel is by Wendland \cite{wendland2004scattered}. The Wendland kernel is a compactly supported kernel, i.e., its $\rho$ function has a compact support of $[0, \ell]$ (where $\ell$ is a tunable length scale hyperparameter). 
    Its construction is based on recursively integrating the function $r\mapsto \max\{1-r/\ell, 0\}$. We omit the details here, but only point out that the resulting $s-d_x/2$ is equal to half odd integers ($1/2$, $3/2$, etc.). 
\end{remark}

\begin{lemma}[Linear--Sobolev kernel for linear--Sobolev space]
\label{lem:LR.kernel}
    If $\mbb{X}\subset \mR^{d_x}$ is a compact region with Lipschitz boundary, and $s>d_x/2$, then $\rg H^s(\mbb{X})$ is an RKHS, whose reproducing kernel is
    $$\rg \kappa(x,x') = (x\cdot x')\kappa(x,x') = (x\cdot x')\rho(|x-x'|), $$
    where $\kappa$ is the Sobolev kernel and $\rho$ is its defining radial function as in the previous lemma. We will refer to this kernel as a ``linear--Sobolev'' kernel. We denote the canonical feature of $x$ in the linear--Sobolev RKHS $\mc{H}_{\rg\kappa}(\mbb{X})$ as $\rg\phi(x)$. 
\end{lemma}

\begin{proposition}[Geometry of two liftings]
    For any $x\in\mbb{X}$, we have 
    $$\|\phi(x)\|_{\mc{H}_\kappa(\mbb{X})} = \rho(0)^{1/2}=\mr{const.}>0, \enspace 
    \|\rg\phi(x)\|_{\mc{H}_{\rg\kappa}(\mbb{X})} = \rho(0)^{1/2} |x|\propto |x|.$$
\end{proposition}
\begin{proof}
    Clearly, $\|\phi(x)\|_{\mc{H}_\kappa(\mbb{X})}^2 = \ip{\kappa(x,\cdot)}{\kappa(x,\cdot)} = \kappa(x,x) = \rho(0)$ for all $x\in\mbb{X}$. Similarly, we see that  $\|\rg\phi(x)\|_{\mc{H}_{\rg\kappa}(\mbb{X})} ^2 = \rg\kappa(x,x) = (x\cdot x)\kappa(x,x) = |x|^2\rho(0)$. 
\end{proof}
The above simple conclusion reveals a key difference between the kernels of the two RKHSs and their corresponding canonical feature maps. As a result, in $H^s(\mbb{X})$, any dynamics is mapped onto a dynamics on the unit sphere in the infinite dimensions, while in $\rg H^s(\mbb{X})$, dynamics around the origin as an equilibrium point is mapped onto a dynamics around the origin in the infinite dimensions as an equilibrium point. This conceptual distinction is illustrated by Figure \ref{fig:geometry}. 
In the next subsection, we shall see the usefulness of the linear--Sobolev feature in defining the kernel quadratic forms, which serves as a key instrument in controller synthesis. 
\begin{figure}[!t]
    \centering
    \begin{subfigure}[b]{0.45\textwidth}
         \centering
         \includegraphics[width=\textwidth]{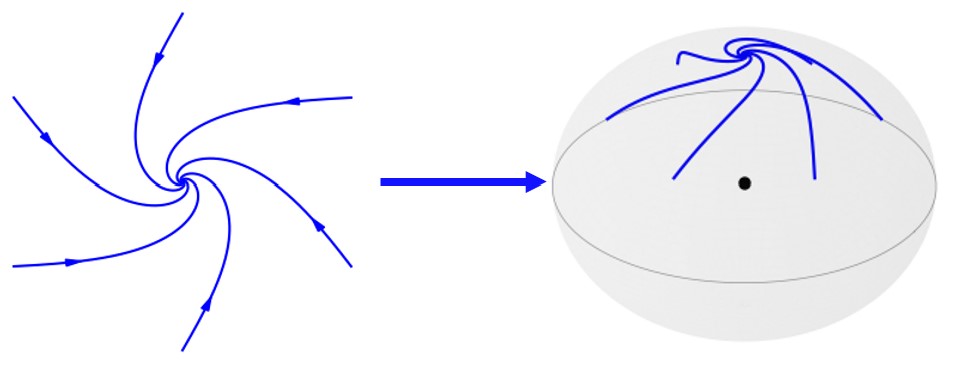}
         \caption{Lifting by $\phi$} 
     \end{subfigure}
     \hfill
     \begin{subfigure}[b]{0.45\textwidth}
         \centering
         \includegraphics[width=\textwidth]{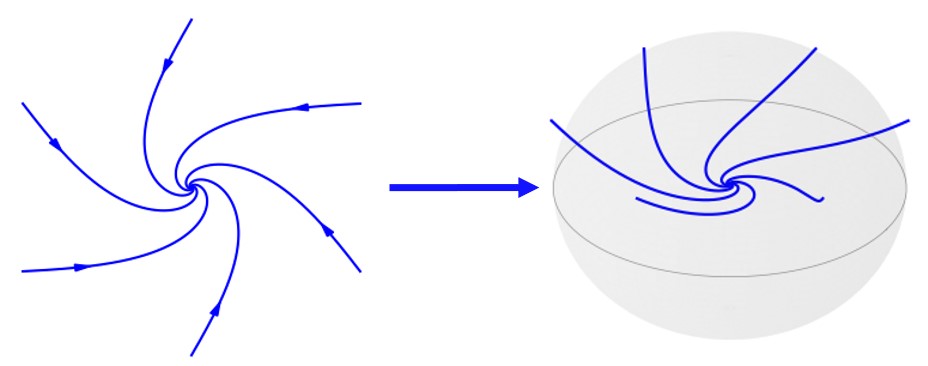}
         \caption{Lifting by $\rg\phi$}
     \end{subfigure}
    \caption{Different geometry of the canonical features in Sobolev RKHS and linear--Sobolev RKHS.}
    \label{fig:geometry}
\end{figure}

In an RKHS, it is more convenient to describe the evolution of the states by their canonical features, for which we need Perron--Frobenius semigroups, which are the adjoint semigroups of Koopman semigroups.\footnote{In the literature, Perron--Frobenius semigroups or operators are usually defined on the space of essentially bounded functions $L^\infty(\mbb{X})$ or the space of signed Borel measures $\mc{M}(\mbb{X})$. Here we instead consider defining the Perron--Frobenius semigroup on an RKHS for its computational usefulness.} 
\begin{definition}[Perron--Frobenius semigroups and generators]
    For a strongly continuous semigroup $\{K^t\}_{t\geq 0}$ on a Hilbert space $\mc{G}$ with inner product $\ip{\cdot}{\cdot}$, its adjoint semigroup $\{K^{t*}\}_{t\geq 0}$ is defined by $\ip{K^{t*}g}{h} = \ip{g}{K^t h}$ ($\forall g, h\in \mc{G}$). 
    Under Assumption \ref{assum:regularity}, we define 
    \begin{enumerate}[nolistsep]
        \item the adjoint of the drift Koopman semigroup $\{K_0^t\}_{t\geq 0}$ on $\rg H^s(\mbb{X})\equiv \mc{H}_{\rg\kappa}(\mbb{X})$ as the drift Perron--Frobenius semigroup $\{T_0^t\}_{t\geq 0} = \{(K_0^t)^*\}_{t\geq 0}$, 
        \item the adjoint of the actuation Koopman semigroups $\{K_j^t\}_{t\geq 0}$ on $H^s(\mbb{X})\equiv \mc{H}_{\kappa}(\mbb{X})$ as the actuation Perron--Frobenius semigroup $\{T_j^t\}_{t\geq 0} = \{(K_j^t)^*\}_{t\geq 0}$, for every $j=1,\cdots,d_u$, and 
        \item the adjoint of the closed-loop Koopman semigroup $\{K_\pi^t\}_{t\geq 0}$ (for a given $\pi\in [\rg C^s(\mbb{X})]^{d_u}$) on $\rg H^s(\mbb{X})\equiv \mc{H}_{\rg\kappa}(\mbb{X})$ as the closed-loop Perron--Frobenius semigroup $\{T_\pi^t\}_{t\geq 0} = \{(K_\pi^t)^*\}_{t\geq 0}$. 
    \end{enumerate}
    We denote their infinitesimal generators as $A_0$, $A_j$ ($j=1,\cdots,d_u$), and $A_\pi$, respectively. 
\end{definition}
Let us denote the flow of the autonomous system $\dot{x}(t)=f_0(x(t))$ by $\{S_{f_0}^t\}_{t\geq0}$, i.e., starting from any $x\in \mbb{X}$, after a time duration of $t$, the autonomous system ends up at the state $S_{f_0}^t(x)$. Similarly, denote the flow of each virtual autonomous system $\dot{x}(t) = f_j(x(t))$ as $\{S_{f_j}^t\}_{t\geq 0}$, and the flow under any feedback policy $\pi\in [\rg C^s(\mbb{X})]^{d_u}$ as $\{S_\pi^t\}_{t\geq 0}$. Then we have the following characterization. 
\begin{proposition}[Characterization of Perron--Frobenius semigroups] 
    Under Assumption \ref{assum:regularity}, for any $x\in\mbb{X}$ and $t\geq 0$, we have
    $$T_0^t\rg\phi(x) = \rg\phi(S_{f_0}^t(x)), \ T_j^t\phi(x) = \phi(S_{f_j}^t(x)), \ T_\pi^t\rg\phi(x) = \rg\phi(S_\pi^t(x)). $$ 
\end{proposition}
\begin{proof}
    Without loss of generality it suffices to prove the first conclusion. For any $g\in \rg H^s(\mbb{X})$, 
    $$\ip{T_0^t\rg\phi(x)}{g}_{\mc{H}_{\rg\kappa}(\mbb{X})} = \ip{\rg\phi(x)}{K_0^tg}_{\mc{H}_{\rg\kappa}(\mbb{X})} = (K_0^t g)(x) = g(S_{f_0}^t(x)) = \ip{\rg\phi(S_{f_0}^t(x))}{g}_{\mc{H}_{\rg\kappa}(\mbb{X})}. $$ 
    Therefore, due to the arbitrariness of $g$, we have $T_0^t\rg\phi(x) = \rg\phi(S_{f_0}^t(x))$. 
\end{proof}

\begin{remark}[Short-time average of canonical features along the flow]
\label{rem:short-time-average}
    Although the conclusion of the above proposition holds for the Perron--Frobenius semigroup, we must be careful not to apply it na\"{i}vely to its infinitesimal generator $A_0$. Actually, $\rg\phi(x)$ ($\phi(x)$) is not guaranteed to reside in the generator's domain $\mc{D}(A_0)$ ($\mc{D}(A_j)$, respectively), which is an unbounded operator whose domain is only a dense subset of the RKHS. Hence, $A_0\rg\phi(x)$, $A_j\phi(x)$, or $A_\pi \rg\phi(x)$ is not meaningful. 
    Nevertheless, if $\rg\phi(x)$ is replaced by a short-time average with a given $\epsilon>0$: $\rg\phi(x)_{0,\epsilon} = \frac{1}{\epsilon}\int_0^\epsilon \rg\phi(S_{f_0}^t(x)) \xD{t}$, then $\rg\phi(x)_{0, \epsilon} \in \mc{D}(A_0)$ holds indeed:
    $$\begin{aligned}
        A_0\rg\phi(x)_{0, \epsilon} &= \lim_{t\downarrow 0} \frac{1}{t} \bra{T_0^t \rg\phi(x)_{0,\epsilon} - \rg\phi(x)_{0,\epsilon}} \\
        &= \lim_{t\downarrow 0} \frac{1}{t\epsilon } \bra{ \int_\epsilon^{\epsilon+t} \rg\phi(S_{f_0}^\tau(x))\xD{\tau} - \int_0^t \rg\phi(S_{f_0}^\tau(x))\xD{\tau}} = \frac{1}{\epsilon} \bra{\rg\phi(S_{f_0}^\epsilon(x)) - \rg\phi(x)}.
    \end{aligned}$$
    Analogously we can define $\phi(x)_{j, \epsilon}\in \mc{D}(A_j)$ for every $j=1,\cdots,d_u$ and $\rg\phi(x)_{\pi, \epsilon}\in \mc{D}(A_\pi)$.  
\end{remark}

By the affine relation \eqref{eq:Koopman.generator.closed-loop} among the Koopman generators, we have by duality the affine relation among the Perron--Frobenius generators:
\begin{equation}\label{eq:PF.generator.closed-loop} 
    A_\pi = A_0 + \sum_{j=1}^{d_u} \rg\iota^* A_j M_{\pi_j}^* \ : \ \mc{D}(A_\pi):= \mc{D}(A_0)\cap \bra{\bigcap_{j=1}^{d_u} (M_{\pi_j}^*)^{-1}\mc{D}(A_j)} \subset \rg H^s(\mbb{X}) \to \rg H^s(\mbb{X}). 
\end{equation} 
It is not difficult to check the following characterization of the adjoint of multiplication operator. 
\begin{proposition}[Adjoint of multiplication operator]
    Suppose that $\varpi \in \rg C^s(\mbb{X})$, where $\mbb{X}\subset \mR^{d_x}$ is a compact region. Then 
    $$M_\varpi^* : \rg H^s(\mbb{X})\to H^s(\mbb{X}), \ \rg\phi(x) \mapsto \varpi(x) \phi(x), \ \forall x\in \mbb{X}.$$
\end{proposition}
\begin{proof}
    Given any $g\in H^s(\mbb{X})$, we have
    $$\ip{M_\varpi^* \rg\phi(x)}{g} = \ip{\rg\phi(x)}{M_\varpi g} = \ip{\rg\phi(x)}{\varpi g} = \varpi(x)g(x) = \varpi(x)\ip{\phi(x)}{g} = \ip{\varpi(x)\phi(x)}{g}. $$
    The arbitrariness of $g$ then implies the conclusion. 
\end{proof}
For the adjoint operator of the inclusion map, $\rg\iota^*: H^s(\mbb{X})\to \rg H^s(\mbb{X})$, we have the following characterization. 
\begin{proposition}[Adjoint of inclusion map]
\label{lem:inclusion}
    The inclusion $\rg\iota$ and its adjoint $\rg\iota^*$ are bounded linear operators, and 
    $$\rg\iota^*\phi(x)=\rg\phi(x), \ \forall x\in \mbb{X}. $$ 
\end{proposition}
\begin{proof}
    For any $g\in \rg H^s(\mbb{X})$, it is clear that $\ip{\rg\iota^*\phi(x)}{g} = \ip{\phi(x)}{\rg\iota g} = g(x)$ holds for all $x\in \mbb{X}$. 
\end{proof}

\subsection{Koopman Spectrum and Kernel Quadratic Form}\label{subsec:prelim-spectrum}
Our purpose to study the stabilization of a nonlinear system in the form of \eqref{eq:system} will inevitably invoke the concept of a \emph{control Lyapunov function}, which should be a nonnegative-valued function of the state $x$ with some regularity. Such a control Lyapunov function must satisfy some linear--quadratic form of inequalities in closed loop. 
To facilitate such discussions, we define the kernel quadratic form in this subsection. 

\begin{definition}[Trace operator and Hilbert--Schmidt operator]
    Let $\mc{G}$ be a separable Hilbert space with orthonormal basis $\{w_k\}_{k=1}^\infty$. By a trace operator $A \in \mc{B}_1(\mc{G})$, we refer to a bounded linear operator $A \in \mc{B}(\mc{G})$ such that its trace $\tr(A) = \sum_{k=1}^\infty \ipp{w_k}{A}{w_k}$ exists (i.e., the sequence $\BRA{ \ipp{w_k}{A}{w_k} }_{k=1}^\infty$ is absolutely summable, and if so, independent of the choice of orthonormal basis). 
    By a Hilbert--Schmidt operator $P\in \mc{B}_2(\mc{G})$, we refer to a bounded linear operator $P \in \mc{B}(\mc{G})$ such that $P^*P\in \mc{B}_1(\mc{G})$ (or equivalently $PP^*\in \mc{B}_1(\mc{G})$, due to the equality $\tr AB = \tr BA$). 
\end{definition}
In this work, we only consider self-adjoint operators $P=P^* \in \mc{B}_1(\mc{G})$ or $\mc{B}_2(\mc{G})$. For such operators, there is always a spectral decomposition $P = \sum_{k=1}^\infty p_kw_k\times w_k$, where $\{w_k\}_{k=1}^\infty$ forms an orthonormal basis of the Hilbert space and the eigenvalues $\{p_k\}_{k=1}^\infty$ are summable or square-summable (i.e., in $\ell^1(\mN)$ or $\ell^2(\mN)$, respectively). Moreover, the $\ell^1$-norm (or $\ell^2$-norm) of the eigenvalue sequence is a norm of $P\in \mc{B}_1(\mc{G})$ (or $\mc{B}_2(\mc{G})$, respectively), called its trace norm (or Hilbert--Schmidt norm, respectively): 
$$ \|P\|_{\mc B_1} = \sum_{k=1}^\infty |p_k|, \quad \|P\|_{\mc B_2} = \Bra{\sum_{k=1}^\infty |p_k|^2}^{1/2}.$$

\begin{definition}[Kernel quadratic form and kernel SOS form]
    A kernel quadratic form on $\mc{H}_\kappa(\mbb{X})$ refers to a function $v_P: x\mapsto \ipp{\phi(x)}{P}{\phi(x)}$ with a self-adjoint Hilbert--Schmidt operator $P\in \mc{B}_2(\mc{H}_\kappa(\mbb{X}))$. Similarly, a kernel quadratic form on $\mc{H}_{\rg\kappa}(\mbb{X})$ refers to a function $v_P: x\mapsto \ipp{\rg\phi(x)}{P}{\rg\phi(x)}$ with a self-adjoint Hilbert--Schmidt operator $P\in \mc{B}_2(\mc{H}_{\rg\kappa}(\mbb{X}))$. 
    By spectral decomposition of the operator $P = \sum_{k=1}^\infty p_kw_k\times w_k$, where $w_k\in \mc{H}_\kappa(\mbb{X})$ (or $\mc{H}_{\rg\kappa}(\mbb{X})$), we have $v_P(x)=\ip{\phi(x)}{P\phi(x)} = \sum_{k=1}^\infty p_k|w_k(x)|^2$. 
    In particular, when all $p_k\geq 0$, we call $v_P$ a kernel sum-of-squares (SOS) form. 
\end{definition} 
\begin{remark}[Typical kernel quadratic and SOS forms]
\label{rem:quadratic-form}
    Because of the conclusion of Proposition \ref{lem:inclusion}, any kernel quadratic form on $\rg H^s(\mbb{X})$: $v(x)=\ipp{\rg\phi(x)}{P}{\rg\phi(x)}$ can be equivalently rewritten as a kernel quadratic form on $H^s(\mbb{X})$ by $v(x)=\ip{\phi(x)}{\rg\iota P \rg\iota^* \phi(x)}$. 
    Kernel quadratic forms on $\mc{H}_\kappa(\mbb{X})$ for the Sobolev kernel $\kappa$ encode functions that can be expressed as sums of squares of $\mc{H}_\kappa(\mbb{X})$-class functions, namely a sum of squares with Sobolev regularity. Hence, kernel quadratic forms on $\mc{H}_{\rg\kappa}(\mbb{X})$ are sums of squared linear--Sobolev functions and can be interpreted as ``locally at least quadratic, globally Sobolev-regular'' functions. 
    For example, consider $v(x)=|x|^2=\sum_{k=1}^{d_x} e_k(x)^2$. Since $1\in H^s(\mbb{X})$, $e_k\in \rg{H}^s(\mbb{X})=\mc{H}_{\rg\kappa}(\mbb{X})$. We thus have 
    $$v(x) = \sum_{k=1}^{d_x} \abs{\ip{e_k}{\rg\phi(x)}}^2 = \ipp{\rg\phi(x)}{P}{\rg\phi(x)},\, P:=\sum_{k=1}^{d_x} e_k\times e_k. $$
    Similarly, if we let $v(x)=x^\top \ms{P}x$ where the matrix $\ms{P} \succeq 0$, then by finding the eigenvectors $\{\ms{w}_k\}_{k=1}^{d_x}$ and eigenvalues $\{p_k\}_{k=1}^{d_x}$ of $\ms{P}$, we have $P=\sum_{k=1}^{d_x} p_k w_k\times w_k$, where $w_k \in \rg H^s(\mbb{X}): x\mapsto \ms{w}_k^\top x$. 
\end{remark}

\par A natural question is whether, for the closed-loop system \eqref{eq:system.closed-loop} under any policy $\pi\in [\rg C^s(\mbb{X})]^{d_u}$, there exists a Lyapunov function that can be written as a kernel SOS form in $\rg H^s(\mbb{X})$. The answer is negative, as should be expected. 
Nevertheless, in the following case, where the policy $\pi$ is stabilizing and guarantees exponential stability locally around the origin, a kernel SOS form of control Lyapunov function can be found. 
\begin{definition}[Exponential stabilization in the homeomorphic sense]
    The policy $\pi\in [\rg C^s(\mbb{X})]^{d_u}$ is said to be \emph{exponentially stabilizing in the homeomorphic sense}, if for the closed-loop flow $\{S_{f_\pi}^t\}_{t\geq 0}$ on $\mbb{X}$, there is a bijective mapping $\psi: \mbb{X}\to \psi(\mbb{X}) \subset \mR^{d_x}$ such that $\psi \in [\rg C^s(\mbb{X})]^{d_x}$, $\psi^{-1} \in [\rg C^s(\psi(\mbb{X}))]^{d_x}$ (and hence $\psi$ can be called a $\rg C^s$-diffeomorphism), and $$\psi\circ S_{f_\pi}^t(x) = \exp(tJ_\pi)\psi(x),$$ 
    where $J_\pi = \mr{D}f_{\pi}(0)$ is the closed-loop Jacobian at the origin, and all the $d_x$ eigenvalues of $J_\pi$ belong to the open left half of the complex plane $\{\lambda\in \mC: \Re\lambda<0\}$. 
\end{definition}

\begin{theorem}[Koopman spectral radius bound, an extension of the discrete-time version in Tang \& Ye \cite{tang-ye2025koopman}]
    Under Assumption \ref{assum:regularity}, if $\pi\in [\rg C^s(\mbb{X})]^{d_u}$ is an exponentially stabilizing policy in the homeomorphic sense, then at any given $t\geq0$, the closed-loop Koopman operator $K_\pi^t$ and closed-loop Perron--Frobenius operator $T_\pi^t$ have spectral radii $$r(K_\pi^t) = r(T_\pi^t) \leq \exp(\lambda_1 t).$$ 
    Here $\lambda_1 = \max\{\Re\lambda: \lambda\in\sigma(J_\pi)\}$. 
\end{theorem}
\begin{proof}
    To prove $r(K_\pi^t) \leq \exp(\lambda_1 t)$, by the Gel'fand formula for spectral radius \cite{lax2002functional}, it suffices to prove that there exists a constant $c>0$ such that $\|(K_\pi^t)^k\|\leq c\exp(\lambda_1 kt)$ for all $k\in \mN$. Note that $(K_\pi^t)^k = K_\pi^{kt}$, which maps any $g\in \rg H^s(\mbb{X})$ to $g\circ S_{f_\pi}^{kt}$. It thus remains to prove that $\norm{g\circ S_{f_\pi}^{kt}}_{\rg H^s} \lesssim \exp(\lambda_1 kt)\norm{g}_{\rg H^s}$. 
    Due to the homeomorphism hypothesis, we have
    $$ g\circ S_{f_\pi}^{kt}(\cdot) = g\bra{\psi^{-1}\bra{\exp(ktJ_\pi)\psi(\cdot)}}. $$
    We hence investigate the $\rg H^s$-norm of this function. 
    \par First, the following assertion is easy to prove. If $\psi: \psi^{-1}(\mbb{X})\to \mbb{X}$ is an injection belonging to the $C^s$-class, and $\inf_{x\in \mbb{X}} |\det \mr{D}\psi|\geq \epsilon$ for some constant $\epsilon>0$, then  
    $$\|g\circ \psi\|_{H^s} \lesssim \|g\|_{H^s}, \ \forall g\in H^s(\mbb{X}).$$ 
    To show this, we use Fa{\`a} di Bruno's formula: for any multi-index $\alpha$ with $|\alpha|<s$, we have 
    $$\partial^\alpha(g\circ\psi) = \sum_{1\leq|\beta|\leq|\alpha|} (\partial^\beta g)\circ\psi \cdot p_{\alpha,\beta}\big(\{\partial^\gamma\psi\}_{1\leq|\gamma|\leq|\alpha|-|\beta|+1}\big),$$
    where each $p_{\alpha, \beta}$ is a polynomial function, in which, because $\inf_{x\in \mbb{X}} |\det \mr{D}\psi|\geq \epsilon$, we have 
    $$\|(\partial^\beta g)\circ\psi\|_{L^2}\lesssim \|\partial^\beta g\|_{L^2}.$$
    The assertion thus follows. 
    Then the assertion can be extended by substituting $H^s$ and $C^s$ with $\rg H^s$ and $\rg C^s$, respectively, i.e., if $\psi$ is an injection belonging to the $\rg C^s$-class, and $\inf_{x\in \mbb{X}} |\det \mr{D}\psi|\geq \epsilon$ for some constant $\epsilon>0$, then  
    $$\|g\circ \psi\|_{\rg H^s} \lesssim \|g\|_{\rg H^s}, \ \forall g\in \rg H^s(\mbb{X}).$$ 
    
    \par Second, we have the following assertion. Given that $J_\pi\in \mR^{d_x\times d_x}$ is a Hurwitz matrix (i.e., all its eigenvalues reside in the open left half plane) and $t \geq 0$ is fixed, then the function $x\mapsto g(\exp(ktJ_\pi)x)$ is in $\rg H^s(\mbb{X})$, and 
    $$ \norm{g\bra{\exp(ktJ_\pi)(\cdot)}}_{\rg H^s} \lesssim \exp(\lambda_1kt)\|g\|_{\rg H^s}.$$
    Indeed, if we write $A = \exp(ktJ_\pi)$, then the eigenvalues of $A$ all have moduli smaller than or equal to $\exp(\lambda_1kt)$. Hence, for any $g\in H^s(\mbb{X})$
    $$ \partial^\alpha (g(Ax)) = \sum_{|\beta|=|\alpha|} a^{\alpha\beta} (\partial^\beta g)(Ax), $$
    in which the coefficient $a^{\alpha\beta}$ is a product of entries in $A$, all of which decay to $0$ as $k\to\infty$ at the rate of $\exp(\lambda_1kt)$. Therefore, $\|g(A(\cdot))\|_{H^s}\lesssim \|g\|_{H^s}$. 
    Now if we have $g = \sum_{\ell=1}^{d_x} e_\ell g_\ell(\cdot)$ with $g_1, \cdots, g_{d_x}\in H^s(\mbb{X})$, $g(A(\cdot)) = \sum_{\ell=1}^{d_x} (A\circ e)_\ell g_\ell(A(\cdot)) = \sum_{\ell, \ell'=1}^{d_x} e_{\ell'} a_{\ell\ell'}g_\ell(A(\cdot))$. Since $|a_{\ell\ell'}| \lesssim \exp(\lambda_1kt)$, clearly, $\|g(A(\cdot))\|_{\rg H^s}\lesssim \exp(\lambda_1kt) \|g\|_{\rg H^s}$. 

    \par Applying the above two assertions to the $\rg C^s$-homeomorphism $\psi$ that brings the flow of $f_\pi\in [\rg C^s(\mbb{X})]^{d_x}$ to its Jacobian exponent, we obtain
    $$\norm{g\bra{\psi^{-1}\bra{\exp(ktJ_\pi)\psi(\cdot)}} }_{\rg H^s} \lesssim \norm{g\bra{\psi^{-1}\bra{\exp(ktJ_\pi)(\cdot)}} }_{\rg H^s} \lesssim \exp(\lambda_1 kt)\norm{g\bra{\psi^{-1}(\cdot)}}_{\rg H^s} \lesssim \exp(\lambda_1 kt)\norm{g}_{\rg H^s},$$
    which is what needs to be proved. 
\end{proof}

\begin{corollary}[Control Lyapunov function as a kernel SOS form, an extension of the discrete-time version in Tang \& Ye \cite{tang-ye2026koopman}]
    Under Assumption \ref{assum:regularity}, if $\pi\in [\rg C^s(\mbb{X})]^{d_u}$ is an exponentially stabilizing policy in the homeomorphic sense, then for any $q(x) = \ipp{\rg\phi(x)}{Q}{\rg\phi(x)}$ where $0\preceq Q \in \mc{B}_2(\rg H^s(\mbb{X}))$ (or $\mc{B}_1(\rg H^s(\mbb{X}))$), there exists a $v(x) = \ipp{\rg\phi(x)}{P}{\rg\phi(x)}$ where $0\preceq P \in \mc{B}_2(\rg H^s(\mbb{X}))$ (respectively, $\mc{B}_1(\rg H^s(\mbb{X}))$), such that 
    $$v\bra{S_{f_\pi}^t(x)} - v(x) = -\int_0^t q\bra{S_{f_\pi}^\tau(x)} \xD{\tau}.$$
\end{corollary}
\begin{proof}
    Let us define 
    $$v(x) = \int_0^\infty q\bra{S_{f_\pi}^t(x)} \xD{t} 
    = \int_0^\infty \ipp{\rg\phi \bra{S_{f_\pi}^t(x)}}{Q}{\rg\phi \bra{S_{f_\pi}^t(x)}} \xD{t} 
    = \int_0^\infty \ipp{\rg\phi (x)}{T_\pi^{t*}QT_\pi^t}{\rg\phi(x)} \xD{t}.$$
    It thus remains to prove that 
    $$ P = \int_0^\infty T_\pi^{t*}QT_\pi^t \xD{t} \in \mc{B}_2(\rg H^s(\mbb{X}))$$
    is well-defined. By the hypothesis of exponential stabilization in the homeomorphic sense, we have $r(T_\pi^t)\leq \exp(\lambda_1 t)$, and due to the fact that $\{T_\pi^t\}_{t\geq 0}$ is a strongly continuous semigroup, $\|T_\pi^t\|\lesssim \exp(\lambda_1 t)$. 
    With $\lambda_1<0$, this implies that $\norm{T_\pi^{t*}QT_\pi^t}_{\mc B_2} \lesssim  \mr{e}^{2\lambda_1 t}\norm{Q}_{\mc B_2}$, and hence $$\|P\|_{\mc{B}_2} \leq \int_0^\infty \mr{e}^{2\lambda_1 t}\norm{Q}_{\mc B_2} \xD{t} \lesssim \norm{Q}_{\mc B_2} <\infty.$$
    The proof for the $\mc{B}_1$ case is similar. 
\end{proof}

\subsection{Yosida Approximation and Its Error}
The Perron--Frobenius generators $A_0$, $A_1, \cdots, A_{d_u}$ are difficult to use directly in a data-driven context, since these infinitesimal generators are unbounded operators and each of their domains is only a dense subspace of the RKHS rather than the entire space. 
As already pointed out in Remark \ref{rem:short-time-average}, even the canonical features $\rg\phi(x)$ and $\phi(x)$ are not in their domains. To resolve this difficulty, the concept of Yosida approximation of generators is naturally recalled. 
\begin{definition}[Yosida approximation]
    For any infinitesimal generator $A$ of a strongly continuous semigroup $\{T^t\}_{t\geq 0}$ defined on a Banach space $\mc{G}$, its Yosida approximation refers to the following bounded operator
    $$ A_\lambda = \lambda A(\lambda I - A)^{-1} = \lambda^2(\lambda I-A)^{-1} - \lambda I =: \lambda^2 R_\lambda - \lambda I.$$
    Here $R_\lambda$ is the resolvent of $A$. The Yosida approximation is well-defined for any sufficiently large $\lambda>0$, where it can be expressed by the Laplace transform:
    $$ A_\lambda = \lambda^2\int_0^\infty \mr{e}^{-\lambda t}T^t \xD{t} - \lambda I.$$
\end{definition}
\par In this paper, under Assumption \ref{assum:regularity}, Yosida approximations can be defined for unbounded operators $A_0: \mc{D}(A_0)\subset \rg H^s(\mbb{X})\to \rg H^s(\mbb{X})$ and $A_j: \mc{D}(A_j)\subset H^s(\mbb{X})\to H^s(\mbb{X})$. We write $A_{0, \lambda}$ and $A_{j, \lambda}$ for them. 
Define the semigroups $\{T_{0,\lambda}^t\}_{t\geq 0}$ and $\{T_{j,\lambda}^t\}_{t\geq 0}$ respectively by:
$$T_{0,\lambda}^t = \exp(tA_{0,\lambda}), \enspace T_{j,\lambda}^t = \exp(tA_{j,\lambda}), $$
where for a bounded operator $A$, $\exp(A) = I + A + \frac{1}{2!}A^2 + \frac{1}{3!}A^3 + \cdots= \sum_{k=0}^\infty \frac{1}{k!}A^k$. 
Clearly, for any $x\in \mbb{X}$, the Yosida-approximated generators can act on its canonical feature $\phi(x)$ by:
$$A_{j,\lambda}\phi(x) = \lambda^2 \int_0^\infty \mr{e}^{-\lambda t} \phi\bra{S_{f_j}^t(x)} \xD{t} - \lambda\phi(x).$$

\par However, the infinite-time integral hampers computability. We hence introduce a further approximation, referred to as the \emph{finite-time Yosida approximation}, introduced in the literature by Meng et al. \cite{meng2024koopman, meng2026resolvent}. 
\begin{definition}[Finite-time Yosida approximation]
    For any infinitesimal generator $A$ of a strongly continuous semigroup $\{T^t\}_{t\geq 0}$ defined on a Banach space $\mc{G}$, define for sufficiently large $\lambda>0$ and large $\theta>0$: 
    $$ A_{\lambda, \theta} = \lambda^2\int_0^\theta \mr{e}^{-\lambda t}T^t \xD{t} - \lambda I =: \lambda^2 R_{\lambda,\theta} - \lambda I.$$
    It is a bounded operator. 
\end{definition}
We hence correspondingly introduce $A_{0,\lambda,\theta}$, $A_{j, \lambda,\theta}$, and the exponential semigroups that they generate: $\{T_{0,\lambda,\theta}^t\}_{t\geq 0}$ and $\{T_{j,\lambda,\theta}^t\}_{t\geq 0}$ ($j=1,\cdots,d_u$). Similarly, we have for any $x\in \mbb{X}$:
$$A_{j,\lambda,\theta}\phi(x) = \lambda^2 \int_0^\theta \mr{e}^{-\lambda t} \phi\bra{S_{f_j}^t(x)} \xD{t} - \lambda\phi(x).$$

\begin{theorem}[Error estimate of finite-time Yosida approximation]
\label{th:Yosida.DD}
    Suppose that $A$ is the infinitesimal generator of a strongly continuous semigroup $\{T^t\}_{t\geq 0}$ defined on a Banach space $\mc{G}$. Let $\omega\in\mR$ and $c>0$ be such that $\|T^t\|\leq c\mr{e}^{\omega t}$.\footnote{The fact that any strongly continuous semigroup $\{T_t\}_{t\geq0}$ on a Banach space $\mc{G}$ must satisfy $\|T^t\|\leq c\mr{e}^{\omega t}$ for some $c\geq 0$ and $\omega\in \mR$ is well known in functional analysis \cite{lax2002functional}. }
    If $\lambda>\omega$, then 
    $$ \|A_{\lambda,\theta}z - A_{\lambda} z\|_{\mc{G}} \lesssim \frac{\lambda^2}{\lambda-\omega} \mr{e}^{-(\lambda-\omega)\theta}\|z\|_{\mc{G}_1}, \enspace \forall z\in \mc{G}_1 := \mc{D}(A), $$
    where $\|z\|_{\mc{G}_1} = \|z\| + \|Az\|$.\footnote{This was shown in Meng et al. \cite{meng2026resolvent} in the context of $\mc{G} = C(\mbb{X})$. Transferring the conclusion to $\mc{G} = \rg H^s(\mbb{X})$ does not pose any difference.} 
    Moreover, for $z\in \mc{G}_2:= \mc{D}(A^2)$, it holds that
    \begin{equation}
        \|T_{\lambda,\theta}^t z - T^t z\| \lesssim  \frac{1}{\lambda-\omega} t\mr{e}^{\max(\omega, \omega_\lambda) t} \|z\|_{\mc{G}_2}, 
    \end{equation} 
    where $\|z\|_{\mc{G}_2} = \|z\|+\|A^2z\|$, $\omega_\lambda = \lambda\omega/(\lambda-\omega)$. 
    Hence, for any $z\in \mc{G}$ and $t \geq 0$, we have $\lim_{\lambda\to\infty} T_{\lambda,\theta}^t z = T^t z$. 
\end{theorem}
\begin{proof}
    First, we have $A_{\lambda,\theta} = A_\lambda - D$, with $D:=\lambda^2 \int_\theta^\infty \mr{e}^{-\lambda t}T^t \xD{t}$. We can check that $\|D\|\lesssim \frac{\lambda^2}{\lambda-\omega} \mr{e}^{-(\lambda-\omega)\theta}$, and hence the first conclusion follows. 
    \par For the finite-time resolvent $R_{\lambda,\theta} = \int_0^\theta \mr{e}^{-\lambda t}T^t\xD{t}$, we have $R_{\lambda,\theta}^n=\int_{[0,\theta]^n} \mr{e}^{-\lambda(s_1+\cdots+s_n)}\,T^{s_1+\cdots+s_n} \xD{s_1}\cdots \xD{s_n}$ and thus $\|R_{\lambda,\theta}^n\| \lesssim \bra{\lambda-\omega}^{-n}$. Thus
    $$\norm{ \mr{e}^{tA_{\lambda,\theta}} } = \norm{ \mr{e}^{-\lambda t} \sum_{n\geq0} \frac{(t\lambda^2)^n}{n!}R_{\lambda,\theta}^n } \lesssim \mr{e}^{(\frac{\lambda^2}{\lambda-\omega}-\lambda)t} = \mr{e}^{\omega_\lambda t}. $$ 
    Now fix $z\in \mc{D}(A)$ and set $\varphi(s) = \mr{e}^{(t-s)A_{\lambda,\theta}}T^s z$ for $s\in[0,t]$. Then 
    $$\varphi'(s) = \mr{e}^{(t-s)A_{\lambda,\theta}} \bra{T^s Az - A_{\lambda,\theta}T^s z} = \mr{e}^{(t-s)A_{\lambda,\theta}}T^s(A-A_{\lambda,\theta})z .$$
    Integrating from $0$ to $t$, 
    $$\norm{T^t z - \mr{e}^{tA_{\lambda,\theta}}z } \lesssim \int_0^t \mr{e}^{\omega_\lambda(t-s)}\, \mr{e}^{\omega s}\xD{s} \cdot \norm{(A-A_{\lambda,\theta}) z} \leq t\,\mr{e}^{\max(\omega,\omega_\lambda)t} \norm{(A-A_{\lambda,\theta}) z}.$$
    For $z\in \mc{D}(A^2)$, $A_\lambda z-Az = \lambda R_\lambda Az - Az = R_\lambda A^2 z$, and hence $\|A_\lambda z-Az\|\lesssim \frac{1}{\lambda-\omega}\|A^2z\|$. With $(A-A_{\lambda,\theta})z = (A-A_\lambda)z+Dz$, 
    $$\|(A-A_{\lambda,\theta})z\| \lesssim \frac{1}{\lambda-\omega}\bra{ \|A^2z\| + \lambda^2 \mr{e}^{-(\lambda-\omega)\theta}\|z\| }.$$
    Therefore
    $$ \norm{T^t z - \mr{e}^{tA_{\lambda,\theta}}z } \lesssim \frac{1}{\lambda-\omega} t\,\mr{e}^{\max(\omega,\omega_\lambda)t} \bra{ \|A^2z\| + \lambda^2 \mr{e}^{-(\lambda-\omega)\theta}\|z\| } \lesssim \frac{1}{\lambda-\omega} t\,\mr{e}^{\max(\omega,\omega_\lambda)t} \|z\|_{\mc{G}_2}. $$
    \par At any given $t$, upon $\lambda\to\infty$, the above conclusion implies $\norm{T^t z - \mr{e}^{tA_{\lambda,\theta}}z } \to 0$ for all $z\in \mc{G}_2$. Because $\mc{G}_2\subset \mc{G}$ densely, $T^t z - \mr{e}^{tA_{\lambda,\theta}}z \to 0$ for all $z\in \mc{G}$ when $\lambda\to\infty$. 
\end{proof}

\begin{remark}[Non-uniform convergence of Yosida-approximated semigroup]
    The Yosida-approximated semigroup only warrants \emph{strong convergence} to the original semigroup, i.e., when acting on an arbitrarily \emph{given} element $z\in \mc{G}$. The convergence speed is not uniform over the space and cannot be expressed only in terms of $\lambda, \theta$, and $\|z\|$. 
    In fact, if $z$ belongs to an interpolation space between $\mc{G}$ and $\mc{D}(A)$ (as a Banach space endowed with the graph norm $\|z\|_{\mc{G}_1} = \|z\|_{\mc{G}}+ \|Az\|_{\mc{G}}$), then it is possible to characterize the convergence rate according to the interpolation index. 
    However, here we are mainly concerned with members of the form $z=\phi(x) \in \mc{G}$, which does not appear to have a higher regularity than the regularity of the kernel function used to define $\mc{G}$ itself; hence the interpolation argument is not applicable. 
\end{remark}

For the closed-loop system \eqref{eq:system.closed-loop}, we know that the open-loop Perron--Frobenius generators $A_0, A_1, \cdots, A_{d_u}$ and the closed-loop Perron--Frobenius generator satisfy the affine equality \eqref{eq:PF.generator.closed-loop}. 
In a data-driven method, it is reasonable to conceive that the open-loop generators will be approximated via the finite-time Yosida formulation (which is related to open-loop simulation for data collection), while the closed-loop Yosida approximation will be built from the affine equality: 
\begin{equation}\label{eq:PF.generator.closed-loop.Yosida}
    A_{\pi,\lambda,\theta} = A_{0,\lambda,\theta} + \sum_{j=1}^{d_u} \rg\iota^* A_{j,\lambda,\theta} M_{\pi_j}^*. 
\end{equation}
In other words, $A_{\pi,\lambda,\theta}$ will not be the Yosida approximation of $A_{\pi}$. Naturally, one would be interested in knowing the difference of the flows generated by $A_{\pi,\lambda,\theta}$, namely $\{T_{\pi,\lambda,\theta}^t\}_{t\geq 0} = \{\exp(tA_{\pi,\lambda,\theta})\}_{t\geq 0}$, and the actual closed-loop flow $\{T_\pi^t\}_{t\geq 0}$. 
\begin{corollary}[Error estimate of Yosida-approximated closed-loop prediction] 
\label{cor:Yosida.DD}
    Let $\pi \in [\rg C^s(\mbb{X})]^{d_u}$ be given, and suppose that Assumption \ref{assum:regularity} holds. Denote 
    $$\mc{G}_1 := \mc{D}(A_0)\cap \bra{\bigcap_{j=1}^{d_u} (M_{\pi_j}^*)^{-1} \mc{D}(A_j)}, \ \mc{G}_2 := \mc{D}(A_0^2)\cap \bra{\bigcap_{j=1}^{d_u} (M_{\pi_j}^*)^{-1} \mc{D}(A_j^2)}, $$
    with norms $\|z\|_{\mc{G}_1} := \|z\| + \|A_0z\| + \sum_{j=1}^{d_u} \|A_jM_{\pi_j}^*z\|$, and $\|z\|_{\mc{G}_2} := \|z\| + \|A_0^2z\| + \sum_{j=1}^{d_u} \|A_j^2M_{\pi_j}^*z\|$, 
    and suppose that for any $t\geq 0$, $T_\pi^t$ boundedly maps $\mc{G}_1$ into $\mc{G}_1$ and $\mc{G}_2$ into $\mc{G}_2$ with an operator norm 
    $$\|T_\pi^t\|_{\mc{G}_1\to \mc{G}_1}, \|T_\pi^t\|_{\mc{G}_2\to \mc{G}_2} \lesssim \mr{e}^{\omega_\pi t}.$$
    Denote $\omega_0$ the growth rate of $\{T_0^t\}_{t\geq 0}$, $\omega_j$ the growth rate of $\{T_j^t\}_{t\geq 0}$ ($j=1,\cdots,d_u$), $\omega = \max\{\omega_0, \omega_1, \cdots, \omega_{d_u}\}$, and $\omega_{\pi,\lambda} = \lambda\omega_\pi/(\lambda-\omega_\pi)$. 
    Then, for any $z\in \mc{G}_2$, 
    \begin{equation}\label{eq:Yosida.error.bound}
        \norm{ T_\pi^tz - \mr{e}^{tA_{\pi, \lambda,\theta}}z } \lesssim \frac{1}{\lambda-\omega}t\mr{e}^{\max(\omega_\pi, \omega_{\pi,\lambda}) t} \|z\|_{\mc{G}_2}. 
    \end{equation} 
\end{corollary}
\begin{proof}
    Consider function $\varphi(s) = \mr{e}^{(t-s)A_{\pi,\lambda,\theta}} T_\pi^sz$ defined for $s\in [0, t]$ when $t$ is fixed and $z\in \mc{G}_2$. By differentiation, we find that 
    $$\varphi'(s) = \mr{e}^{(t-s)A_{\pi,\lambda,\theta}} \bra{A_{\pi,\lambda,\theta} - A_\pi} T_\pi^sz = \mr{e}^{(t-s)A_{\pi,\lambda,\theta}} \Bra{ \bra{A_{0,\lambda,\theta} - A_0} + \sum_{j=1}^{d_u} \rg\iota^* \bra{A_{j,\lambda,\theta} - A_j} M_{\pi_j}^* }T_\pi^sz. $$
    By the hypothesis on $T_\pi^s$ restricted from $\mc{G}_2$ to $\mc{G}_2$, we see that $T_\pi^sz \in \mc{G}_2$ and $\|T_\pi^sz\|_{\mc{G}_2}\lesssim \mr{e}^{\omega_\pi s}\|z\|_{\mc{G}_2}$. 
    As Theorem \ref{th:Yosida.DD} reveals, 
    $$\|\varphi'(s)\|\lesssim \mr{e}^{\omega_{\pi,\lambda}(t-s)} \Bra{\frac{1}{\lambda-\omega_0} + \sum_{j=1}^{d_u} \frac{1}{\lambda-\omega_j} \norm{\rg\iota^*}\norm{M_{\pi_j}^*} } \norm{T_\pi^sz}_{\mc{G}_2} \lesssim \mr{e}^{\omega_{\pi,\lambda}(t-s)} \mr{e}^{\omega_\pi s} \frac{1}{\lambda-\omega}\|z\|_{\mc{G}_2}. $$
    Integrating over the interval $s\in [0, t]$, the conclusion immediately follows. 
\end{proof}

A sufficient condition for $\|T_\pi^t\|_{\mc{G}_1\to \mc{G}_1}$ and $\|T_\pi^t\|_{\mc{G}_2\to \mc{G}_2}$ to be both bounded by a constant multiple of $\mr{e}^{\omega_\pi t}$ is sufficient smoothness of the system: $s\geq2$. 
\begin{proposition}[Growth bound on the generator domain]
    If $s\geq 2$ and $\|T_\pi^t\|\lesssim \mr{e}^{\omega_\pi t}$, then under Assumption \ref{assum:regularity}, any given $\pi \in [\rg C^s(\mbb{X})]^{d_u}$ satisfies $\|T_\pi^t\|_{\mc{G}_1\to \mc{G}_1}, \|T_\pi^t\|_{\mc{G}_2\to \mc{G}_2} \lesssim \mr{e}^{\omega_\pi t}$, and hence implies the conclusion of Corollary \ref{cor:Yosida.DD}. 
\end{proposition}
\begin{proof}
    We define 
    $$\mc{G}_{-1} = \mc{D}(L_0) \cap \bra{\bigcap_{j=1}^{d_u} \rg\iota^{-1} \mc{D}(L_j)}, \ \mc{G}_{-2} = \mc{D}(L_0^2) \cap \bra{\bigcap_{j=1}^{d_u} \rg\iota^{-1} \mc{D}(L_j^2)} $$
    which we endow with the norms $\|g\|_{\mc{G}_{-1}} = \|L_0g\| + \sum_{j=1}^{d_u} \|L_j \rg\iota g\|$ and $\|g\|_{\mc{G}_{-2}} = \|L_0^2g\| + \sum_{j=1}^{d_u} \|L_j^2 \rg\iota g\|$. 
    When $s\geq 2$, the $\mc{G}_{-1}$ space indeed contains all $g\in \rg H^s(\mbb{X}) \cap C^1(\mbb{X})$ and $\mc{G}_{-2}$ contains all $g\in \rg H^s(\mbb{X}) \cap C^2(\mbb{X})$. In fact, $\mc{G}_{-1}$ and $\mc{G}_{-2}$ are the dual spaces of $\mc{G}_1$ and $\mc{G}_2$ with respect to the pivot space $\mc{G} = \rg H^s(\mbb{X})$, respectively \cite{tucsnak2009observation}. Clearly, $\mc{G}_{-1}$ and $\mc{G}_{-2}$ are dense subspaces of $\mc{G}$. Moreover, $\|g\|_{\mc{G}_{-2}}$ defined here is equivalent to the norm by duality: $\|g\|_{\mc{G}_2^*} = \sup_{z\in \mc{G}_2} |\ip{g}{z}_{\mc{G}}|$. 
    \par Now we have, for any given $z\in \mc{G}_2$ and fixed $t\geq 0$, 
    \begin{equation*}
        \begin{aligned}
            \norm{T_\pi^t z}_{\mc{G}_2} &= \norm{T_\pi^tz} 
            + \sup_{g_0\in \mc{G}_{-2}, \|g_0\|_{\mc{G}_{-2}}=1} \ip{g_0}{A_0^2T_\pi^tz} 
            + \sum_{j=1}^{d_u} \sup_{g_j\in \mc{G}_{-2}, \|g_j\|_{\mc{G}_{-2}}=1} \ip{g_j}{A_j^2M_{\pi_j}^*T_\pi^tz} \\
            &= \norm{T_\pi^tz} + \sup_{g_0\in \mc{G}_{-2}, \|g_0\|_{\mc{G}_{-2}}=1} \ip{K_\pi^t L_0^2g_0}{z} + \sum_{j=1}^{d_u} \sup_{g_j\in \mc{G}_{-2}, \|g_j\|_{\mc{G}_{-2}}=1} \ip{K_\pi^t M_{\pi_j}L_j^2g_j}{z} \\
            &\lesssim \norm{T_\pi^t}\|z\| + \norm{K_\pi^t} \|z\| \lesssim \mr{e}^{\omega_\pi t}\|z\|_{\mc{G}_2}. 
        \end{aligned}
    \end{equation*}
    The proof is completed. 
\end{proof}

\section{Synthesis of Stabilizing Controllers: A Koopman--Lyapunov Method}
After all the necessary mathematical preparations in the previous section, now let us consider the first problem of interest --- the synthesis of a stabilizing control law of $\rg C^s$ type. As we saw in Subsection \ref{subsec:prelim-spectrum}, when there is a $\rg C^s$-class control law $\pi$ that is exponentially stabilizing in the homeomorphic sense, the search for a control Lyapunov function in the kernel SOS form should be feasible.  
\begin{problem}[Stabilizing control]\label{prob:stabilizing}
    Given a kernel quadratic form $q(x) = \ipp{\rg\phi(x)}{Q}{\rg\phi(x)}$, referred to as the \emph{descent rate}, where $Q \in \mc{B}_1(\mc{H}_{\rg\kappa}(\mbb{X}))$ (or $\mc{B}_2(\mc{H}_{\rg\kappa}(\mbb{X}))$) is positive semidefinite and guarantees that $q(x)=0$ only when $x=0$, determine a positive semidefinite kernel quadratic form $v(x) = \ipp{\rg\phi(x)}{P}{\rg\phi(x)}$, called the \emph{(control) Lyapunov function}, along with a \emph{feedback control policy} $\pi\in [\rg C^s(\mbb{X})]^{d_u}$, such that for any $x\in \mbb{X}$ and $t\geq 0$:
    \begin{equation}\label{eq:stabilizing.0}
        v(S_{f_\pi}^t(x)) - v(x) \leq -\int_0^t q(S_{f_\pi}^\tau(x)) \, \xD{\tau}. 
    \end{equation}
\end{problem}

\subsection{Lyapunov Descent Condition and Its Differentiation}
We first express the Lyapunov descent condition \eqref{eq:stabilizing.0} of Problem \ref{prob:stabilizing} in a kernel quadratic form of the canonical feature of states. 
\begin{proposition}[Lyapunov descent in kernel quadratic form]
    Policy $\pi\in [\rg C^s(\mbb{X})]^{d_u}$ along with the positive semidefinite Hilbert--Schmidt operator $P\in \mc{B}_2(\rg H^s(\mbb{X}))$ solves Problem \ref{prob:stabilizing} if and only if
    \begin{equation}\label{eq:stabilizing.1}
        \ipp{\rg\phi(x)}{ T_\pi^{t*}PT_\pi^t - P + \int_0^t T_\pi^{\tau *} QT_\pi^\tau \xD{\tau} }{\rg\phi(x)} \leq 0, \quad \forall x\in \mbb{X},\ \forall t\geq 0,
    \end{equation}
    where $\{T_\pi^t\}_{t\geq 0}$ is the closed-loop Perron--Frobenius semigroup defined on $\rg H^s(\mbb{X}) \equiv \mc{H}_{\rg\kappa}(\mbb{X})$. 
\end{proposition} 

The proof is straightforward. The condition \eqref{eq:stabilizing.1}, however, is difficult to compute because it involves all $t\geq0$. Such an integral-time condition must be converted into a differential counterpart, which is complicated by the fact that $\rg\phi(x)$ is not in the domain of the infinitesimal generator $A_\pi$. 
Now let us denote 
\begin{equation}\label{eq:violation.matrix.W}
    W_\pi^t = T_\pi^{t*}PT_\pi^t - P + \int_0^t T_\pi^{\tau *} QT_\pi^\tau \xD{\tau}. 
\end{equation} 
The condition \eqref{eq:stabilizing.1} is now written as $\ipp{\rg\phi(x)}{W_\pi^t}{\rg\phi(x)}\leq 0$ for all $x\in \mbb{X}$ and $t\geq 0$. To perform differentiation over $t$ on $W_\pi^t$, we have to approximate $\rg\phi(x)$ using a member of $\mc{D}(A_\pi)$. 
To this end, we let 
$$\rg\phi(x)_{\pi,\epsilon} = \frac{1}{\epsilon}\int_0^\epsilon T_\pi^t \rg\phi(x) \xD{t} = \frac{1}{\epsilon}\int_0^\epsilon \rg\phi\bra{S_{f_\pi}^t(x)} \xD{t}$$ 
for a sufficiently small $\epsilon>0$. 
In $\rg H^s(\mbb{X})$, it is easy to verify that as $\epsilon\to 0$, $\rg\phi(x)_{\pi,\epsilon} \to \rg\phi(x)$. 
Thus, we approximate condition \eqref{eq:stabilizing.1} by
\begin{equation}\label{eq:stabilizing.2}
    \ipp{\rg\phi(x)_{\pi,\epsilon}}{W_\pi^t}{\rg\phi(x)_{\pi,\epsilon}} \leq 0, 
    \enspace\forall x\in \mbb{X}, \, t\geq0, \, \epsilon>0,  
\end{equation}
at the expense of an error as follows. 
\begin{lemma}[Error on Lyapunov descent due to short-time average]
    Let $\pi\in [\rg C^s(\mbb{X})]^{d_u}$, $P, Q \in \mc{B}_2(\rg H^s(\mbb{X}))$, $P, Q \succeq 0$, $W_\pi^t$ be as defined in \eqref{eq:violation.matrix.W}, and Assumption \ref{assum:regularity} hold.  
    There exists a function $\eta: \mR_+\to\mR_+$ with $\eta(0)=0$, dependent on the choice of $Q$ and $\pi$,  such that $$\abs{ \ipp{\rg\phi(x)_{\pi,\epsilon}}{W_\pi^t}{\rg\phi(x)_{\pi,\epsilon}} - \ipp{\rg\phi(x)}{W_\pi^t}{\rg\phi(x)} } \lesssim (1+t)\eta(\epsilon).$$ 
    In particular, if \eqref{eq:stabilizing.2} is satisfied, then \eqref{eq:stabilizing.1} holds with a violation bounded by a multiple of $(1+t)\eta(\epsilon)$ on its right-hand side. 
\end{lemma}
\begin{proof}
    It is not difficult to see that $T_\pi^t\rg\phi(x)_{\pi,\epsilon} = \rg\phi(S_{f_\pi}^t(x))_{\pi,\epsilon}$. Hence, 
    $$
    \begin{aligned}
    & \left| \ipp{\rg\phi(x)_{\pi,\epsilon}}{W_\pi^t}{\rg\phi(x)_{\pi,\epsilon}} - \ipp{\rg\phi(x)}{W_\pi^t}{\rg\phi(x)} \right| \\ 
    \leq & 
    \left| \ipp{\rg\phi(S_{f_\pi}^t(x))_{\pi,\epsilon}}{P}{\rg\phi(S_{f_\pi}^t(x))_{\pi,\epsilon}} - \ipp{\rg\phi(S_{f_\pi}^t(x))}{P}{\rg\phi(S_{f_\pi}^t(x))} \right| \\
    & + \left| \ipp{\rg\phi(x)_{\pi,\epsilon}}{P}{\rg\phi(x)_{\pi, \epsilon}} - \ipp{\rg\phi(x)}{P}{\rg\phi(x)} \right| \\
    & + \int_0^t \left| \ipp{\rg\phi(S_{f_\pi}^\tau(x))_{\pi,\epsilon}}{Q}{\rg\phi(S_{f_\pi}^\tau(x))_{\pi,\epsilon}} - \ipp{\rg\phi(S_{f_\pi}^\tau(x))}{Q}{\rg\phi(S_{f_\pi}^\tau(x))} \right| \, \xD{\tau}.
    \end{aligned}$$
    Since $P$ is a bounded operator, the second term can be bounded by a constant multiple of the distance $\|\rg\phi(x)_{\pi,\epsilon} - \rg\phi(x)\|$ and the squared distance $\|\rg\phi(x)_{\pi,\epsilon} - \rg\phi(x)\|^2$ (which, at small $\epsilon$ is dominated by the linear distance). 
    Due to the compactness of $\mbb{X}$, the collection of canonical features $\{\rg\phi(x): x\in \mbb{X}\}$ is also a compact set. Hence, the strong convergence of $\rg\phi(x)_{\pi,\epsilon}$ pointwise to $\rg\phi(x)$ as $\epsilon\to 0$ results in equicontinuity. Thus, there exists a modulus function of $\epsilon$, say $\eta_P(\epsilon)$, such that 
    $$\left| \ipp{\rg\phi(x)_{\pi,\epsilon}}{P}{\rg\phi(x)_{\pi, \epsilon}} - \ipp{\rg\phi(x)}{P}{\rg\phi(x)} \right| \lesssim \eta_P(\epsilon)$$
    and the same holds for the first term, and each integrand in the third term with another modulus $\eta_Q(\epsilon)$. Therefore, 
    $$\left| \ipp{\rg\phi(x)_{\pi,\epsilon}}{W_\pi^t}{\rg\phi(x)_{\pi,\epsilon}} - \ipp{\rg\phi(x)}{W_\pi^t}{\rg\phi(x)} \right| \lesssim 2\eta_P(\epsilon) + t\eta_Q(\epsilon). $$
    The conclusion is thus proved. 
\end{proof}

Now we can convert \eqref{eq:stabilizing.2} into a differential condition. The left-hand side is now differentiable with respect to $t$, and equal to $0$ at $t=0$. 
\begin{lemma}[Differential Lyapunov descent condition]
    Condition \eqref{eq:stabilizing.2} holds if and only if 
    \begin{equation}\label{eq:stabilizing.3}
    \ipp{\rg\phi(x)_{\pi,\epsilon}}{ P A_\pi+ A_\pi^* P + Q}{\rg\phi(x)_{\pi,\epsilon}} \leq 0, \enspace \forall x\in \mbb{X}, \,\epsilon>0.
    \end{equation}
\end{lemma}
\begin{proof}
    Since $\rg\phi(x)_{\pi,\epsilon} \in \mc{D}(A_\pi)$, we have $\frac{\xD{}}{\xD{t}} T_\pi^t \rg\phi(x)_{\pi,\epsilon} = T_\pi^t A_\pi \rg\phi(x)_{\pi,\epsilon} = A_\pi T_\pi^t \rg\phi(x)_{\pi,\epsilon}$. Hence, 
    $$ \begin{aligned}
        \xDD{}{t} \ipp{\rg\phi(x)_{\pi,\epsilon}}{W_\pi^t}{\rg\phi(x)_{\pi,\epsilon}} &= \ipp{\rg\phi(x)_{\pi,\epsilon}}{(A_\pi T_\pi^t)^* P T_\pi^t + T_\pi^{t*} PA_\pi T_\pi^t + T_\pi^{t*}Q T_\pi^t}{\rg\phi(x)_{\pi,\epsilon}} \\
        &= \ipp{\rg\phi(x)_{\pi,\epsilon}}{T_\pi^{t*} \bra{A_\pi^*P+PA_\pi+Q} T_\pi^t}{\rg\phi(x)_{\pi,\epsilon}} \\
        &= \ipp{T_\pi^t\rg\phi(x)_{\pi,\epsilon}}{A_\pi^*P+PA_\pi+Q}{T_\pi^t\rg\phi(x)_{\pi,\epsilon}}
    \end{aligned}$$
    Note that $T_\pi^t\rg\phi(x)_{\pi,\epsilon} = \rg\phi\bra{S_{f_\pi}^t(x)}_{\pi, \epsilon}$, where $S_{f_\pi}^t(x) \in \mbb{X}$. If \eqref{eq:stabilizing.2} holds for $x\in \mbb{X}$ at all $t\geq 0$ and $\epsilon>0$, then 
    $$\xDD{}{t} \ipp{\rg\phi(x)_{\pi,\epsilon}}{W_\pi^t}{\rg\phi(x)_{\pi,\epsilon}} \bigg|_{t=0} = \ipp{\rg\phi(x)_{\pi,\epsilon}}{ P A_\pi+ A_\pi^* P + Q}{\rg\phi(x)_{\pi,\epsilon}} \leq 0, $$
    namely \eqref{eq:stabilizing.3} holds. Vice versa, if \eqref{eq:stabilizing.3} holds, then $\xDD{}{t} \ipp{\rg\phi(x)_{\pi,\epsilon}}{W_\pi^t}{\rg\phi(x)_{\pi,\epsilon}}$ must be non-positive at all times, which implies \eqref{eq:stabilizing.2} by integration over time on $[0,t]$. 
\end{proof}

In \eqref{eq:stabilizing.3}, $A_\pi$ is still an unbounded operator, which is inconvenient for data-based approximation, where the hypothetical short-time average $\rg\phi(x)_{\pi,\epsilon}$ would be undesirably involved in the computation. 
The Yosida approximation provides an instrument to avoid such a limitation, by replacing the semigroup $\{T_{\pi}^t\}_{t\geq0}$ with $\{T_{\pi,\lambda,\theta}^t\}_{t\geq0}$.

\subsection{Yosida Approximation and Data-Based Constraints}
With the Yosida approximation, we rather directly approximate the $W_\pi^t$ matrix by its Yosida-like surrogate $W_{\pi,\lambda,\theta}^t$, in which we define 
\begin{equation}\label{eq:violation.matrix.W.Yosida}
    W_{\pi,\lambda,\theta}^t = T_{\pi,\lambda,\theta}^{t*} P T_{\pi,\lambda,\theta}^t - P + \int_0^t T_{\pi,\lambda,\theta}^{\tau *} Q T_{\pi,\lambda,\theta}^\tau \xD{\tau}, 
\end{equation} 
where the semigroup $\{T_{\pi,\lambda,\theta}^t\}_{t\geq0}$ is generated by the bounded operator $A_{\pi,\lambda,\theta}$ constructed by the finite-time Yosida approximations of $A_0, A_1, \cdots, A_{d_u}$ according to the affine relation \eqref{eq:PF.generator.closed-loop.Yosida}. 
Hence we approximate the condition \eqref{eq:stabilizing.1} as 
\begin{equation}\label{eq:stabilizing.4}
    \ipp{\rg\phi(x)}{W_{\pi,\lambda,\theta}^t}{\rg\phi(x)} \leq 0, 
    \enspace\forall x\in \mbb{X}, \, t\geq0,  
\end{equation}
at the expense of an error to be analyzed later. Intuitively, this error is associated with the difference between $W_\pi^t$ and $W_{\pi,\lambda,\theta}^t$. 

Similarly, we derive a differential condition to replace the integral condition:
\begin{equation}\label{eq:stabilizing.5}
    \ipp{\rg\phi(x)}{PA_{\pi,\lambda,\theta} + A_{\pi,\lambda,\theta}^*P + Q}{\rg\phi(x)} \leq 0, 
    \enspace\forall x\in \mbb{X}. 
\end{equation}
Since $A_{\pi,\lambda,\theta}$ is a bounded linear operator, its domain is the entire $\rg H^s(\mbb{X})$ and hence there is no need to explicitly invoke a short-time average $\rg\phi(x)_{\pi,\epsilon}$. This is convenient for numerical computation. 
On the other hand, there is no longer bidirectional implication between \eqref{eq:stabilizing.4} and \eqref{eq:stabilizing.5}. Clearly \eqref{eq:stabilizing.4} implies \eqref{eq:stabilizing.5} by taking the differentiation at $t=0$; but \eqref{eq:stabilizing.5} does not give \eqref{eq:stabilizing.4}. In fact, \eqref{eq:stabilizing.5} implies that
\begin{equation*}
    \ipp{T_\pi^t\rg\phi(x)}{PA_{\pi,\lambda,\theta} + A_{\pi,\lambda,\theta}^*P + Q}{T_\pi^t \rg\phi(x)} \leq 0, 
    \enspace\forall x\in \mbb{X}, \, t\geq0, 
\end{equation*}
as $T_\pi^t \rg\phi(x)$ is the $\rg\phi$ of another state in $\mbb{X}$. However, to restore \eqref{eq:stabilizing.4} we need the inequality to hold for  
\begin{equation*}
    \ipp{T_{\pi,\lambda,\theta}^t \rg\phi(x)}{PA_{\pi,\lambda,\theta} + A_{\pi,\lambda,\theta}^*P + Q}{T_{\pi,\lambda,\theta}^t \rg\phi(x)} \leq 0, 
    \enspace\forall x\in \mbb{X}, \, t\geq0. 
\end{equation*}
In other words, since the Yosida-approximated flow is not a flow in the original state space, but a virtual flow in the RKHS with bounded generators, the differential condition cannot be truly integrated on the state-space flow. 
Therefore, if we impose \eqref{eq:stabilizing.5} as the constraint on the controller, there is an error associated with the difference between $T_{\pi,\lambda,\theta}^t$ and $T_\pi^t$. 

\par Finally, the condition \eqref{eq:stabilizing.5} cannot be made to hold on all $x\in \mbb{X}$ computationally. Instead, for a data-driven solution, assume that we have a sample $\mbb{S} = \{x_1,\cdots,x_n\}\subset \mbb{X}$, obtained by random sampling or on a lattice grid. We impose the differential condition on the sample points only:
\begin{equation}\label{eq:stabilizing.6}
    \ipp{\rg\phi(x)}{PA_{\pi,\lambda,\theta} + A_{\pi,\lambda,\theta}^*P + Q}{\rg\phi(x)} \leq 0, 
    \enspace\forall x\in \mbb{S}=\{x_i\}_{i=1}^n. 
\end{equation}
For this, we need to characterize the possible violation of \eqref{eq:stabilizing.5} over the state space when the inequalities are satisfied only on $\mbb{S}$ as in \eqref{eq:stabilizing.6}. 
This bound depends on the fill distance 
$$r_\mbb{S} = \max_{x\in \mbb{X}} \min_{i=1,\cdots,n} |x-x_i|.$$
The fill distance under uniform random sampling is related to the sample size $n$. 
\begin{lemma}[Fill distance of uniform random sample]
    Suppose that $\mbb{X}\subset \mR^{d_x}$ is compact and Lebesgue measurable, and $\mbb{S} =\{x_i\}_{i=1}^n$ are independent and identically distributed sample points drawn from the uniform distribution on $\mbb{X}$. Then with probability $1-\delta$ over the draw, we have $$r_{\mbb{S}} \lesssim \bra{\frac{1}{n}\log\frac{n}{\delta}}^{1/d_x}. $$
\end{lemma}

However, since we are concerned with a stabilization problem where convergence to the origin is needed, instead of relying on a small fill distance on $\mbb{X}$, we require that the sample points be ``proportionally close'' to any point in the state space. Specifically, we shall assume the following sampling pattern. 
\begin{assumption}[Proportionally dense sampling]\label{assum:sampling.strategy}
    For any $x\in \mbb{X}$, there exists a sampled point $x'\in \mbb{S}$ such that $|x-x'|\leq r_{\mbb{S}}|x|$ where $r_{\mbb{S}}$ is a positive constant. 
\end{assumption}
\begin{remark}[Condition for sampling denseness]
\label{rem:sampling.strategy}
    The above assumption on sampling denseness can be satisfied with $r_{\mbb{S}}=1$ if $0\in \mbb{S}$, i.e., if the origin point is contained in the sample. In the sample-based formulation of stabilizing control problem \eqref{eq:stabilizing.6} and optimal control problems in the later section, it can be verified that this sample point at the origin always leads to a vacuously correct ``$0\leq0$'' or ``$0=0$'' constraint and has no effect on the problem. 
    \par If we need a small enough $r_{\mbb{S}} \ll 1$, then including $0\in \mbb{S}$ is not enough. This sampling pattern can be achieved by first randomly sampling the unit ball $\mbb{B}(0,1)\subset \mR^{d_x}$ to create a sample $\mbb{S}^\circ$. Suppose that $a_{\mbb{X}} = \max\{|x|:x\in \mbb{X}\}$. Then we let $\mbb{S}_0 = a_{\mbb{X}}\mbb{S}^\circ \cap \mbb{X}$ (by multiplying each sample point in the unit ball by the region size), $\mbb{S}_1 = 2^{-1}a_{\mbb{X}}\mbb{S}^\circ \cap \mbb{X}$, $\mbb{S}_2 = 2^{-2}a_{\mbb{X}}\mbb{S}^\circ \cap \mbb{X}$, etc., and put $\mbb{S} = \cup_{k=0}^\infty \mbb{S}_k$. The $r_{\mbb{S}}$ in Assumption \ref{assum:sampling.strategy} is then twice the fill distance in $\mbb{S}^\circ$. 
    Theoretically, this would require an infinite number of points, and hence the union over $k\in \mbb{N}$ would have to be truncated to a finite index in practical use. At large $k$, the sample $\mbb{S}_k$ would be extremely close to $0$, and hence the constraint \eqref{eq:stabilizing.6} on its points is below numerical precision. We can therefore ignore the effect of discarding these sample points on practical computation. 
    \par Therefore, a more precise characterization of the sampling denseness can be stated as: $\forall x\in \mbb{X}$, $\exists x'\in \mbb{S}$, such that $|x-x'|\leq r_{\mbb{S}}(|x|)|x|$, where $r_{\mbb{S}}(|x|)$ depends on $|x|$, bounded above by $1$ when $|x|$ is small, and by a smaller constant when $|x|$ is large. All later error analyses can be revised by this characterization. For computation, there is no need to actually use any different strategy than uniform random sampling. 
\end{remark} 

\par Let us now postpone the computational solution of the \emph{operator Lyapunov inequality} \eqref{eq:stabilizing.6} to the next subsection, before which we shall analyze the errors arising from the approximations \eqref{eq:stabilizing.1}--\eqref{eq:stabilizing.4}--\eqref{eq:stabilizing.5}--\eqref{eq:stabilizing.6} and thus understand what control performance the condition \eqref{eq:stabilizing.6} can guarantee. 
In particular, we will establish a bound on the violation of the Lyapunov condition \eqref{eq:stabilizing.1} in three steps, by backtracking the errors in the three-step approximations.

\subsection{Error Analysis and Control Performance Bound}
Below let us denote $\Delta_\pi = \Delta_\pi(P,Q) =  PA_{\pi} + A_\pi^* P + Q$ and $\Delta_{\pi, \lambda,\theta} = \Delta_{\pi, \lambda,\theta}(P,Q) = PA_{\pi, \lambda,\theta} + A_{\pi, \lambda,\theta}^* P + Q$. The difference between \eqref{eq:stabilizing.5} and \eqref{eq:stabilizing.6} lies in that the former requires that $\ipp{\rg\phi(x)}{\Delta_{\pi, \lambda,\theta}}{\rg\phi(x)}\leq 0$ hold for all $x\in \mbb{X}$, while the latter requires only for $x\in \mbb{S}$. 
The matter is then to characterize how the function $x \mapsto \ipp{\rg\phi(x)}{\Delta_{\pi, \lambda,\theta}}{\rg\phi(x)}$ can vary between the sampled and non-sampled points. 
\begin{proposition}[Error on Yosida-approximated differential Lyapunov descent due to sampling]
\label{prop:error-sample}
    Suppose that $0 \preceq P, Q\in \mc{B}_1(\rg H^s(\mbb{X}))$. Let the sampling strategy satisfy Assumption \ref{assum:sampling.strategy}. There exists a constant $c_{\mbb{S}}$ such that \eqref{eq:stabilizing.6} implies the following relaxed version of \eqref{eq:stabilizing.5}, with a violation bounded by a quadratic function:
    \begin{equation}
        \ipp{\rg\phi(x)}{\Delta_{\pi, \lambda,\theta}}{\rg\phi(x)}\leq c_{\mbb{S}}r_{\mbb{S}}^{\min(1,\alpha)} |x|^2, \ \forall x\in \mbb{X} .
    \end{equation}
    Here $\alpha$ is any positive number in $(0, s-d_x/2)$. 
    Hence if $Q$ is such that $\ipp{\rg\phi(x)}{Q}{\rg\phi(x)} \geq c_Q|x|^2$ for large enough $c_Q>0$ (see Remark \ref{rem:quadratic-form}), then there is a small $\varepsilon_{Q}$, such that 
    \begin{equation}
        \ipp{\rg\phi(x)}{PA_{\pi, \lambda,\theta} + A_{\pi, \lambda,\theta}^* P + (1-\varepsilon_Q)Q}{\rg\phi(x)}\leq 0, \ \forall x\in \mbb{X}.    
    \end{equation}
\end{proposition}
\begin{proof}
    For the symmetric operator $\Delta_{\pi, \lambda,\theta}$, a spectral decomposition results in 
    $$ \Delta_{\pi, \lambda,\theta} = \sum_{i=1}^\infty \mu_i w_i\times w_i, \enspace \Rightarrow \enspace 
    \ipp{\rg\phi(x)}{\Delta_{\pi, \lambda,\theta}}{\rg\phi(x)} = \sum_{i=1}^\infty \mu_i |w_i(x)|^2 $$
    where $\{w_i\}_{i=1}^\infty$ is an orthonormal basis of $\rg H^s(\mbb{X})$ and thus can be further written as $w_i(x) = x\cdot \bar w_i(x)$ where $\bar w_i\in [H^s(\mbb{X})]^{d_x}$. 
    By the sampling strategy that satisfies Assumption \ref{assum:sampling.strategy}, for all $x\in \mbb{X}$, there must be an $x'\in \mbb{S}$ that has a distance not exceeding $r_{\mbb{S}}|x|$ from $x$. Consider the difference:
    $$\begin{aligned}
        & \abs{ \ipp{\rg\phi(x)}{\Delta_{\pi, \lambda,\theta}}{\rg\phi(x)} - \ipp{\rg\phi(x')}{\Delta_{\pi, \lambda,\theta}}{\rg\phi(x')}} \\ 
        & \quad \leq \sum_{i=1}^\infty |\mu_i| \abs{|w_i(x)|^2 - |w_i(x')|^2} 
        \lesssim \sum_{i=1}^\infty |\mu_i| \abs{w_i(x)-w_i(x')} \bra{|w_i(x)|+|w_i(x')|}. 
    \end{aligned}$$
    The difference between $w_i(x)$ and $w_i(x')$ is thus bounded by:
    $$ \abs{w_i(x)-w_i(x')} \leq \abs{(x-x')\cdot \bar w_i(x)} + \abs{x'\cdot \bra{\bar w_i(x)- \bar w_i(x')}} . $$
    Due to the continuous embedding of Sobolev space $H^s(\mbb{X})$ into $C^{\alpha}(\mbb{X})$ and hence further into $C^0(\mbb{X})$, we have
    $$ \abs{w_i(x)-w_i(x')} \lesssim r_{\mbb{S}}|x| \|\bar{w}_i\|_{C^0} + |x'| \|\bar{w}_i\|_{C^\alpha} |x-x'|^\alpha \lesssim r_{\mbb{S}}|x| + (1+r_{\mbb{S}}) r_{\mbb{S}}^\alpha|x|^{\alpha+1}.$$
    We used the fact that $|x'|\leq |x| + r_{\mbb{S}}|x| = (1+r_{\mbb{S}}) |x|$. We also then have $|w_i(x)| \lesssim \|w_i\|_{\rg H^s} |x|= |x|$ and similarly $|w_i(x')|\lesssim |x'| \lesssim (1+r_{\mbb{S}})|x|$. Hence, 
    $$ \begin{aligned}
        & \abs{ \ipp{\rg\phi(x)}{\Delta_{\pi, \lambda,\theta}}{\rg\phi(x)} - \ipp{\rg\phi(x')}{\Delta_{\pi, \lambda,\theta}}{\rg\phi(x')}}  \\
        & \quad \lesssim \sum_{i=1}^\infty |\mu_i| (2+r_{\mbb{S}}) \bra{ r_{\mbb{S}}|x|^2 + (1+r_{\mbb{S}})r_{\mbb{S}}^\alpha |x|^{\alpha+2}}
        \lesssim \norm{\Delta_{\pi,\lambda,\theta}}_{\mc{B}_1}\bra{ r_{\mbb{S}} + r_{\mbb{S}}^\alpha }|x|^2 \lesssim r_{\mbb{S}}^{\min(1,\alpha)} |x|^2.
    \end{aligned}$$
    Therefore, when $c_Q$ is large enough vis-\`{a}-vis $r_{\mbb{S}}$, the error above is bounded by a small proportion of $\ipp{\rg\phi(x)}{Q}{\rg\phi(x)}$, and hence the conclusion to be proved holds true. 
\end{proof}

The above proposition shows that under the dense sampling strategy that satisfies Assumption \ref{assum:sampling.strategy}, there is no essential loss of the differential Lyapunov descent condition, except that $\Delta_{\pi,\lambda,\theta}(P,Q)$ should be replaced by $\Delta_{\pi,\lambda,\theta}(P,(1-\varepsilon_Q) Q)$, or equivalently, by replacing $P$ with $(1-\varepsilon_Q)^{-1}P$. Denoting $\Delta_{\pi,\lambda,\theta}(P,(1-\varepsilon_Q) Q) := \tilde\Delta_{\pi,\lambda,\theta}$, we now have 
$$ \ipp{T_{\pi}^t\rg\phi(x)}{\tilde\Delta_{\pi,\lambda,\theta}}{T_{\pi}^t\rg\phi(x)}\leq 0, \ \forall x\in \mbb{X}, t\geq 0.$$

To recover an integral Lyapunov descent condition, we should consider the error of replacing $T_{\pi}^t$ with $T_{\pi,\lambda,\theta}^t$. This is unfortunately impossible directly with $\rg\phi(x)\notin \mc{D}(A_\pi^2)$. Following the thoughts in Remark \ref{rem:short-time-average}, we define a ``twice short-time average'' by 
$$\rg\phi(x)_{\pi,\epsilon,\epsilon} = \frac{1}{\epsilon}\int_0^\epsilon \rg\phi \bra{S_{f_\pi}^t(x)}_{\pi,\epsilon} \xD{t}, \text{ where } \rg\phi(x)_{\pi,\epsilon} = \frac{1}{\epsilon}\int_0^\epsilon \rg\phi \bra{S_{f_\pi}^t(x)} \xD{t}.$$
It is then straightforward to verify the following properties. 
\begin{lemma}[Error in twice short-time average]
    Under Assumption \ref{assum:regularity}, for $\pi\in [\rg C^s(\mbb{X})]^{d_u}$, we have $\rg\phi(x)_{\pi,\epsilon,\epsilon}\in \mc{D}(A_\pi^2)$. In fact, $A_\pi^2 \rg\phi(x)_{\pi,\epsilon,\epsilon} = \frac{1}{\epsilon^2} \bra{T_\pi^{2\epsilon} - 2T_\pi^\epsilon + I} \rg\phi(x)$. Moreover, $\norm{\rg\phi(x)_{\pi,\epsilon,\epsilon} - \rg\phi(x)}\lesssim \epsilon|x|$.
\end{lemma}
\begin{proof}
    With the definitions, we have
    $$A_\pi \rg\phi(x)_{\pi,\epsilon,\epsilon} = \frac{1}{\epsilon}\int_0^\epsilon A_\pi \rg\phi \bra{S_{f_\pi}^t(x)}_{\pi,\epsilon} \xD{t}.$$
    As shown in Remark \ref{rem:short-time-average}, 
    $$A_\pi\rg\phi \bra{S_{f_\pi}^t(x)}_{\pi,\epsilon} = \frac{1}{\epsilon}\bra{ T_\pi^\epsilon - I} \rg\phi \bra{S_{f_\pi}^t(x)}. $$
    Hence we have
    $$\begin{aligned}
        A_\pi \rg\phi(x)_{\pi,\epsilon,\epsilon} &= \frac{1}{\epsilon}\int_0^\epsilon \frac{1}{\epsilon}\bra{ T_\pi^\epsilon - I} \rg\phi \bra{S_{f_\pi}^t(x)} \xD{t} 
        = \frac{1}{\epsilon^2}\int_0^\epsilon \Bra{ \rg\phi\bra{S_{f_\pi}^{t+\epsilon}(x)} - \rg\phi\bra{S_{f_\pi}^{t}(x)} } \xD{t} \\
        &= \frac{1}{\epsilon}\Bra{ \frac{1}{\epsilon}\int_0^\epsilon T_\pi^t \rg\phi\bra{S_{f_\pi}^{\epsilon}(x)} \xD{t} - \frac{1}{\epsilon}\int_0^\epsilon T_\pi^t \rg\phi(x) \xD{t} } = \frac{1}{\epsilon}\Bra{ \rg\phi\bra{S_{f_\pi}^\epsilon(x)}_{\pi,\epsilon} - \rg\phi\bra{x}_{\pi,\epsilon}}.
    \end{aligned}$$ 
    By applying $A_\pi$ again, we obtain 
    $$\begin{aligned}
    A_\pi^2 \rg\phi(x)_{\pi,\epsilon,\epsilon} &= \frac{1}{\epsilon}\Bra{ A_\pi \rg\phi\bra{S_{f_\pi}^\epsilon(x)}_{\pi,\epsilon} - A_\pi \rg\phi\bra{x}_{\pi,\epsilon}}
    = \frac{1}{\epsilon}\Bra{ \frac{1}{\epsilon}(T_\pi^\epsilon-I)\rg\phi\bra{S_{f_\pi}^\epsilon(x)} - \frac{1}{\epsilon}(T_\pi^\epsilon-I)\rg\phi(x) } \\
    &= \frac{1}{\epsilon^2}\Bra{ T_\pi^\epsilon(T_\pi^\epsilon-I) - (T_\pi^\epsilon-I)} \rg\phi(x) = \frac{1}{\epsilon^2} \bra{T_\pi^{2\epsilon} - 2T_\pi^\epsilon + I}\rg\phi(x).
    \end{aligned}$$
    Let us verify from the definition of short-time average that 
    $$T_\pi^s \rg\phi(x)_{\pi,\epsilon} = \frac{1}{\epsilon}\int_0^\epsilon T_\pi^s \rg\phi \bra{S_{f_\pi}^t(x)} \xD{t} 
    = \frac{1}{\epsilon}\int_0^\epsilon \rg\phi \bra{S_{f_\pi}^s S_{f_\pi}^t(x)} \xD{t} 
    = \frac{1}{\epsilon}\int_0^\epsilon \rg\phi \bra{S_{f_\pi}^{s+t}(x)} \xD{t} = \rg\phi\bra{S_{f_\pi}^s(x)}_{\pi,\epsilon}.$$
    \par To bound the difference between $\rg\phi(x)_{\pi,\epsilon,\epsilon}$ and $\rg\phi(x)$, we first note that $\abs{S_{f_{\pi}}^t (x) - x} \lesssim t|x|$ at small $t$ due to the standing regularity assumption. Hence, 
    $$\norm{\rg\phi\bra{S_{f_{\pi}}^t (x)} - \rg\phi(x)} \lesssim t |x|, $$
    which implies 
    $$ \norm{\rg\phi(x)_{\pi,\epsilon} - \rg\phi(x)} \leq \frac{1}{\epsilon}\int_0^\epsilon \norm{\rg\phi\bra{S_{f_{\pi}}^t (x)} - \rg\phi(x)} \xD{t} \lesssim \frac{1}{\epsilon}\int_0^\epsilon t|x| \xD{t} \lesssim \epsilon|x|, $$
    and subsequently 
    $$ \begin{aligned}
        \norm{\rg\phi(x)_{\pi,\epsilon,\epsilon} - \rg\phi(x)} 
        &\leq \frac{1}{\epsilon}\int_0^\epsilon \norm{\rg\phi\bra{S_{f_{\pi}}^t (x)}_{\pi,\epsilon} - \rg\phi(x)} \xD{t} \\
        &\leq \frac{1}{\epsilon}\int_0^\epsilon \Bra{ \norm{\rg\phi\bra{S_{f_{\pi}}^t(x)}_{\pi,\epsilon} - \rg\phi\bra{S_{f_{\pi}}^t(x)}} + \norm{\rg\phi\bra{S_{f_{\pi}}^t(x)} - \rg\phi(x)} } \xD{t} \\
        &\lesssim \frac{1}{\epsilon}\int_0^\epsilon \bra{\epsilon|S_{f_\pi}^t(x)| + t|x|} \xD{t} \lesssim \epsilon|x|. 
    \end{aligned}$$
    The lemma is thus proved. 
\end{proof}

Now we can prove the error bound on \eqref{eq:stabilizing.5}. 
\begin{proposition}[Error in differential Lyapunov descent for the virtual Yosida-approximated flow]
\label{prop:error-pre-integral}
    Suppose that $0 \preceq P, Q\in \mc{B}_1(\rg H^s(\mbb{X}))$, and let the sampling strategy satisfy Assumption \ref{assum:sampling.strategy}. Suppose also
    \begin{equation}\label{eq:stabilizing.error-pre-integral.condition}
        \ipp{\rg\phi(x)}{\tilde\Delta_{\pi,\lambda,\theta}}{\rg\phi(x)}\leq 0, \ \forall x\in \mbb{X},
    \end{equation} 
    where $\tilde\Delta_{\pi,\lambda,\theta} = PA_{\pi, \lambda,\theta} + A_{\pi, \lambda,\theta}^* P + (1-\varepsilon_Q)Q$ with some $\varepsilon_Q\in(0, 1)$, and $s\geq2$. Then for all $x\in \mbb{X}$ and $t\geq 0$, we have
    \begin{equation}\label{eq:stabilizing.error-pre-integral}
        \ipp{T_{\pi,\lambda,\theta}^t\rg\phi(x)}{\tilde\Delta_{\pi,\lambda,\theta}}{T_{\pi,\lambda,\theta}^t\rg\phi(x)} 
    \lesssim \frac{t^{1/3} \mr{e}^{2\max(\omega_\pi, \omega_{\pi,\lambda})t}  }{(\lambda-\omega_\pi)^{1/3}}\norm{\tilde\Delta_{\pi,\lambda,\theta}} |x|^2.
    \end{equation}
\end{proposition}
\begin{proof}
    Since $T_\pi^t \rg\phi(x) = \rg\phi\bra{S_{f_\pi}^t(x)}$, we have 
    $$\ipp{T_\pi^t \rg\phi(x)}{\tilde\Delta_{\pi,\lambda,\theta}}{T_\pi^t \rg\phi(x)}\leq 0, \ \forall x\in \mbb{X}.$$
    Considering the error induced by replacing $\rg\phi(x)$ with $\rg\phi(x)_{\pi,\epsilon,\epsilon} \in \mc{D}(A_\pi^2)$, we have
    $$ \begin{aligned}
         &\abs{ \ipp{T_\pi^t \rg\phi(x)_{\pi,\epsilon,\epsilon}}{\tilde\Delta_{\pi,\lambda,\theta}}{T_\pi^t \rg\phi(x)_{\pi,\epsilon,\epsilon}} - \ipp{T_\pi^t \rg\phi(x)}{\tilde\Delta_{\pi,\lambda,\theta}}{T_\pi^t \rg\phi(x)}} \\
        & \quad \leq \norm{T_\pi^t}^2 \norm{\tilde\Delta_{\pi,\lambda,\theta}} \bra{ 2\norm{\rg\phi(x)_{\pi,\epsilon,\epsilon} - \rg\phi(x)} \norm{\rg\phi(x)} + \norm{\rg\phi(x)_{\pi,\epsilon,\epsilon} - \rg\phi(x)}^2} 
        \lesssim \mr{e}^{2\omega_\pi t} \norm{\tilde\Delta_{\pi,\lambda,\theta}} \bra{2\epsilon + \epsilon^2}|x|^2.
    \end{aligned}$$
    Similarly we have
    $$ \begin{aligned}
         &\abs{ \ipp{T_{\pi,\lambda,\theta}^t \rg\phi(x)_{\pi,\epsilon,\epsilon}}{\tilde\Delta_{\pi,\lambda,\theta}}{T_{\pi,\lambda,\theta}^t \rg\phi(x)_{\pi,\epsilon,\epsilon}} - \ipp{T_{\pi,\lambda,\theta}^t \rg\phi(x)}{\tilde\Delta_{\pi,\lambda,\theta}}{T_{\pi,\lambda,\theta}^t \rg\phi(x)}} \\
        & \quad \leq \norm{T_{\pi,\lambda,\theta}^t}^2 \norm{\tilde\Delta_{\pi,\lambda,\theta}} \bra{ 2\norm{\rg\phi(x)_{\pi,\epsilon,\epsilon} - \rg\phi(x)} \norm{\rg\phi(x)} + \norm{\rg\phi(x)_{\pi,\epsilon,\epsilon} - \rg\phi(x)}^2} 
        \lesssim \mr{e}^{2\omega_{\pi,\lambda} t} \norm{\tilde\Delta_{\pi,\lambda,\theta}} \bra{2\epsilon + \epsilon^2}|x|^2.
    \end{aligned}$$
    Because $A_\pi^2 \rg\phi(x)_{\pi,\epsilon,\epsilon} = \frac{1}{\epsilon^2} \bra{T_\pi^{2\epsilon} - 2T_\pi^\epsilon + I} \rg\phi(x)$, we have $\norm{\rg\phi(x)_{\pi,\epsilon,\epsilon}}_{\mc{G}_2} \lesssim \epsilon^{-2} \norm{\rg\phi(x)}$, and hence by Corollary \ref{cor:Yosida.DD}, 
    $$ \norm{ T_\pi^t \rg\phi(x)_{\pi,\epsilon,\epsilon} - T_{\pi,\lambda,\theta}^t \rg\phi(x)_{\pi,\epsilon,\epsilon}} \lesssim \frac{1}{\lambda-\omega_\pi} t\mr{e}^{\max(\omega_\pi, \omega_{\pi,\lambda})t} \norm{\rg\phi(x)_{\pi,\epsilon,\epsilon}}_{\mc{G}_2} \lesssim \frac{1}{\lambda-\omega_\pi} \frac{1}{\epsilon^2} t\mr{e}^{\max(\omega_\pi, \omega_{\pi,\lambda})t} \norm{\rg\phi(x)}. $$
    This implies 
    $$ \begin{aligned}
        &\abs{ \ipp{T_{\pi,\lambda,\theta}^t \rg\phi(x)_{\pi,\epsilon,\epsilon}}{\tilde\Delta_{\pi,\lambda,\theta}}{T_{\pi,\lambda,\theta}^t \rg\phi(x)_{\pi,\epsilon,\epsilon}} - \ipp{T_{\pi}^t \rg\phi(x)_{\pi,\epsilon,\epsilon}}{\tilde\Delta_{\pi,\lambda,\theta}}{T_{\pi}^t \rg\phi(x)_{\pi,\epsilon,\epsilon}} } \\
        &\quad \lesssim \frac{1}{\lambda-\omega_\pi} \frac{1}{\epsilon^2} t\mr{e}^{\max(\omega_\pi, \omega_{\pi,\lambda})t} \norm{\rg\phi(x)} \norm{\tilde\Delta_{\pi,\lambda,\theta}} \max\bra{ \norm{T_{\pi,\lambda,\theta}^t \rg\phi(x)_{\pi,\epsilon,\epsilon}}, \norm{T_{\pi}^t \rg\phi(x)_{\pi,\epsilon,\epsilon}}} \\ 
        &\quad \lesssim \frac{1}{\lambda-\omega_\pi} \frac{1}{\epsilon^2} t\mr{e}^{2\max(\omega_\pi, \omega_{\pi,\lambda})t}  \norm{\tilde\Delta_{\pi,\lambda,\theta}} |x|^2
    \end{aligned} $$
    Putting the three estimations together, we have
    $$ \ipp{T_{\pi,\lambda,\theta}^t\rg\phi(x)}{\tilde\Delta_{\pi,\lambda,\theta}}{T_{\pi,\lambda,\theta}^t\rg\phi(x)} 
    \lesssim \bra{\epsilon + \frac{t}{\epsilon^2(\lambda-\omega_\pi)}} \mr{e}^{2\max(\omega_\pi, \omega_{\pi,\lambda})t}  \norm{\tilde\Delta_{\pi,\lambda,\theta}} |x|^2. $$
    Since $\epsilon$ can be chosen on any $x\in \mbb{X}$ arbitrarily, we let $\epsilon$ be the one that minimizes the right-hand side, arriving at 
    $$ \ipp{T_{\pi,\lambda,\theta}^t\rg\phi(x)}{\tilde\Delta_{\pi,\lambda,\theta}}{T_{\pi,\lambda,\theta}^t\rg\phi(x)} 
    \lesssim \frac{t^{1/3} \mr{e}^{2\max(\omega_\pi, \omega_{\pi,\lambda})t}  }{(\lambda-\omega_\pi)^{1/3}}\norm{\tilde\Delta_{\pi,\lambda,\theta}} |x|^2.$$
    This concludes the proof. 
\end{proof}

Now we can consider 
$$\tilde{W}_{\pi,\lambda,\theta}^t = T_{\pi,\lambda,\theta}^{t*} P T_{\pi,\lambda,\theta}^t - P + \int_0^t T_{\pi,\lambda,\theta}^{\tau*}(1-\varepsilon_Q)Q T_{\pi,\lambda,\theta}^\tau \xD{\tau}. $$
Utilizing the conclusion of the above proposition, we have the following error bound on \eqref{eq:stabilizing.4}. 
\begin{proposition}[Error on Lyapunov descent for the virtual Yosida-approximated flow]
\label{prop:error-integral}
    Suppose that $0 \preceq P, Q\in \mc{B}_1(\rg H^s(\mbb{X}))$, and let the sampling strategy satisfy Assumption \ref{assum:sampling.strategy}. Suppose also $s\geq2$ and \eqref{eq:stabilizing.error-pre-integral.condition} holds. Then for all $x\in \mbb{X}$ and $t\geq 0$, we have
    \begin{equation}
        \ipp{\rg\phi(x)}{\tilde{W}_{\pi,\lambda,\theta}^t}{\rg\phi(x)} \lesssim \frac{t^{4/3} \mr{e}^{2\max(\omega_\pi, \omega_{\pi,\lambda})t}  }{(\lambda-\omega_\pi)^{1/3}}\norm{\tilde\Delta_{\pi,\lambda,\theta}} |x|^2. 
    \end{equation}
\end{proposition}
\begin{proof}
    Let $\varphi(t)$ be the $t$-dependent function on the left-hand side of the formula to be proved. We can easily check that 
    $$\varphi'(t) = \ipp{T_{\pi,\lambda,\theta}^t \rg\phi(x)}{\tilde{\Delta}_{\pi,\lambda,\theta}}{T_{\pi,\lambda,\theta}^t \rg\phi(x)}  .$$
    The conclusion to be proved thus follows from integrating the inequality of the previous proposition from $0$ to $t$, which results in the change of the $t^{1/3}$ term into $t^{4/3}$. 
\end{proof}

This allows us to establish a bound on the violation of the Lyapunov descent condition \eqref{eq:stabilizing.1}. Let us denote 
$$ \tilde{W}_{\pi}^t = T_\pi^{t*}PT_\pi^t - P + \int_0^t T_{\pi}^{\tau*} (1-\varepsilon_Q)Q T_{\pi}^{\tau} \xD{\tau}.$$
\begin{proposition}[Error on Lyapunov descent]
\label{prop:error-Lyapunov}
    Suppose that $0 \preceq P, Q\in \mc{B}_1(\rg H^s(\mbb{X}))$, and let the sampling strategy satisfy Assumption \ref{assum:sampling.strategy}. Suppose also $s\geq2$ and
    $$\ipp{\rg\phi(x)}{\tilde\Delta_{\pi,\lambda,\theta}}{\rg\phi(x)}\leq 0, \ \forall x\in \mbb{X}.$$ Then for all $x\in \mbb{X}$ and $t\geq 0$, we have 
    $$ \ipp{\rg\phi(x)}{\tilde{W}_\pi^t}{\rg\phi(x)} \lesssim \frac{t^{4/3}}{(\lambda-\omega_\pi)^{1/3}} \max\bra{\mr{e}^{(2/3) \max(\omega_\pi, \omega_{\pi,\lambda})t}, \mr{e}^{2\max(\omega_\pi, \omega_{\pi,\lambda})t}} |x|^2 . $$
\end{proposition}
\begin{proof}
    As before, we will use the twice short-time average $\rg\phi(x)_{\pi,\epsilon,\epsilon} \in \mc{D}(A_\pi^2)$. We first check the error caused by substituting $\rg\phi(x)$ with $\rg\phi(x)_{\pi,\epsilon,\epsilon}$: 
    $$ \abs{ \ipp{\rg\phi(x)_{\pi,\epsilon,\epsilon}}{\tilde{W}_\pi^t}{\rg\phi(x)_{\pi,\epsilon,\epsilon}} - \ipp{\rg\phi(x)}{\tilde{W}_\pi^t}{\rg\phi(x)}} \lesssim \epsilon|x|^2 \|\tilde{W}_\pi^t\|, $$
    where the norm of $\tilde{W}_\pi^t$ is bounded by 
    $$ \norm{\tilde{W}_\pi^t} \leq \bra{\norm{T_\pi^t}^2 +1}\|P\| + \int_0^t \norm{T_\pi^\tau}^2 \|(1-\varepsilon_Q)Q\|\xD{\tau} \lesssim \mr{e}^{2\omega_\pi t}+1.$$
    That is, 
    $$ \abs{ \ipp{\rg\phi(x)_{\pi,\epsilon,\epsilon}}{\tilde{W}_\pi^t}{\rg\phi(x)_{\pi,\epsilon,\epsilon}} - \ipp{\rg\phi(x)}{\tilde{W}_\pi^t}{\rg\phi(x)}} \lesssim \epsilon\bra{\mr{e}^{2\omega_\pi t}+1}|x|^2.$$
    Similarly, 
    $$ \abs{ \ipp{\rg\phi(x)_{\pi,\epsilon,\epsilon}}{\tilde{W}_{\pi,\lambda,\theta}^t}{\rg\phi(x)_{\pi,\epsilon,\epsilon}} - \ipp{\rg\phi(x)}{\tilde{W}_{\pi,\lambda,\theta}^t}{\rg\phi(x)}} \lesssim \epsilon\bra{\mr{e}^{2\omega_{\pi,\lambda} t}+1}|x|^2.$$
    Also, by the same reasoning as in the previous proposition, invoking Corollary \ref{cor:Yosida.DD}, we have 
    $$\begin{aligned}
        & \abs{ \ipp{\rg\phi(x)_{\pi,\epsilon,\epsilon}}{\tilde{W}_\pi^t}{\rg\phi(x)_{\pi,\epsilon,\epsilon}} - \ipp{\rg\phi(x)_{\pi,\epsilon,\epsilon}}{\tilde{W}_{\pi,\lambda,\theta}^t}{\rg\phi(x)_{\pi,\epsilon,\epsilon}} } \\
        & \quad = \abs{\ipp{T_\pi^t \rg\phi(x)_{\pi,\epsilon,\epsilon}}{P}{T_\pi^t \rg\phi(x)_{\pi,\epsilon,\epsilon}} - \ipp{T_{\pi,\lambda,\theta}^t \rg\phi(x)_{\pi,\epsilon,\epsilon}}{P}{T_{\pi,\lambda,\theta}^t \rg\phi(x)_{\pi,\epsilon,\epsilon}}} \\
        & \qquad + \int_0^t \abs{\ipp{T_\pi^\tau \rg\phi(x)_{\pi,\epsilon,\epsilon}}{(1-\varepsilon_Q)Q}{T_\pi^\tau \rg\phi(x)_{\pi,\epsilon,\epsilon}} - \ipp{T_{\pi,\lambda,\theta}^\tau \rg\phi(x)_{\pi,\epsilon,\epsilon}}{(1-\varepsilon_Q)Q}{T_{\pi,\lambda,\theta}^\tau \rg\phi(x)_{\pi,\epsilon,\epsilon}}} \xD{\tau} \\
        & \quad \lesssim \frac{t\mr{e}^{2\max(\omega_\pi, \omega_{\pi,\lambda})t}}{\lambda-\omega_\pi} \frac{|x|^2}{\epsilon^2} \|P\| + \int_0^t \frac{\tau\mr{e}^{2\max(\omega_\pi, \omega_{\pi,\lambda})\tau}}{\lambda-\omega_\pi} \frac{|x|^2}{\epsilon^2} (1-\varepsilon_Q)\|Q\| \xD{\tau} 
        \lesssim \frac{t^2 \mr{e}^{2\max(\omega_\pi, \omega_{\pi,\lambda})t}}{\epsilon^2 (\lambda-\omega_\pi)} |x|^2.
    \end{aligned}$$
    Therefore, combining the above three estimates, we have
    $$ \abs{ \ipp{\rg\phi(x)}{\tilde{W}_\pi^t}{\rg\phi(x)} - \ipp{\rg\phi(x)}{\tilde{W}_{\pi,\lambda,\theta}^t}{\rg\phi(x)} } \lesssim \epsilon\bra{\mr{e}^{2\max(\omega_\pi, \omega_{\pi,\lambda})t}+1}|x|^2 + \frac{t^2 \mr{e}^{2\max(\omega_\pi, \omega_{\pi,\lambda})t}}{\epsilon^2 (\lambda-\omega_\pi)} |x|^2,$$
    where by optimizing the choice of $\epsilon$ we obtain
    $$ \abs{ \ipp{\rg\phi(x)}{\tilde{W}_\pi^t}{\rg\phi(x)} - \ipp{\rg\phi(x)}{\tilde{W}_{\pi,\lambda,\theta}^t}{\rg\phi(x)} } \lesssim \frac{t^{2/3}}{(\lambda-\omega_\pi)^{1/3}} \max\bra{\mr{e}^{(2/3)\max(\omega_\pi, \omega_{\pi,\lambda})t}, \mr{e}^{2\max(\omega_\pi, \omega_{\pi,\lambda})t}} |x|^2. $$
    Combining with the conclusion of Proposition \ref{prop:error-integral}, we obtain the desired result. 
\end{proof}

Finally, we have the following conclusion that summarizes the above analysis in this subsection. 
\begin{theorem}[Data-driven stabilization]
\label{th:stabilization}
    Suppose that 
    \begin{enumerate}[nolistsep]
        \item the system \eqref{eq:system} satisfies Assumption \ref{assum:regularity} and $s\geq2$, 
        \item the sample $\mbb{S}$ satisfies Assumption \ref{assum:sampling.strategy}, 
        \item $q(x) \gtrsim c_Q|x|^2$ for $x\in \mbb{X}$, where $c_Q$ is sufficiently large vis-\`{a}-vis $r_{\mbb{S}}$, 
        \item $0 \preceq P, Q\in \mc{B}_1(\rg H^s(\mbb{X}))$ and policy $\pi\in [\rg C^s(\mbb{X})]^{d_u}$ together satisfy the data-based differential Lyapunov descent condition \eqref{eq:stabilizing.6}, and moreover,
        \item $\|T_\pi^t\|\lesssim \mr{e}^{\omega_\pi t}$ where $\omega_\pi<0$. 
    \end{enumerate}
    Then there exists an $\varepsilon_Q\in(0,1)$, such that for all $x\in \mbb{X}$ and $t\geq 0$, we have
    \begin{equation}
        v\bra{S_{f_\pi}^t(x)} \leq v(x) - \int_0^t q\bra{S_{f_\pi}^\tau(x)} \xD{\tau} + \mr{const}\cdot t^{4/3} \exp\bra{\frac{2}{3} \frac{\lambda\omega_\pi}{\lambda-\omega_\pi} t} |x|^2, 
    \end{equation}
    and hence $\lim_{t\to \infty} S_{f_\pi}^t(x) = 0$. 
\end{theorem}
\begin{remark}[Dependence of data-driven stabilization performance on closed-loop spectrum]\label{rem:spectrum}
    The conclusion on asymptotic stability hinges on a negative growth rate of the closed-loop semigroup: $\omega_\pi<0$, which then implies that $\omega_{\pi,\lambda} = \frac{\lambda\omega_\pi}{\lambda-\omega_\pi} < 0$. 
    In fact, the exponential stability of $\{T_\pi^t\}_{t\geq 0}$ is a stronger requirement than the Lyapunov condition \eqref{eq:stabilizing.0} or \eqref{eq:stabilizing.1} itself. Hence, by solving \eqref{eq:stabilizing.6}, it is not automatically guaranteed that the closed-loop semigroup becomes exponentially stable. 
    This is indeed where the complexity of the nonlinear system manifests itself, and the author suggests that it entails a fundamental limitation of the operator method --- the verification of operator-based conditions based on operator semigroups would involve the spectrum in infinite dimensions. 
    Although it is possible to \emph{a posteriori} study the closed-loop spectrum, imposing a spectrum constraint on finite data at the synthesis stage does not appear to be easy. 
    \par On the other hand, the spectrum of the drift generator $A_0$ usually contains only a finite number of discrete points (i.e., eigenvalues) to the right of a vertical line $\{\lambda\in \mC: \Re \lambda = -\alpha\}$ for some $\alpha>0$. Hence, by considering these modes, it may be theoretically possible to synthesize a controller $\pi$ to make the closed-loop semigroup an exponentially stable one. Readers may refer to, e.g., Curtain and Zwart \cite{curtain2020introduction} for the case of a fully linear infinite-dimensional system ($\dot{x} = Ax + Bu$). However, the data-driven synthesis would invoke the problem of whether the spectrum of the unbounded infinitesimal generators can be indeed well approximated under data. 
    The author leaves this problem open. 
\end{remark}

\subsection{Data-Driven Solution} 
Now we consider how to convert the condition \eqref{eq:stabilizing.6}, expressed in an operator form and imposed as inequalities on $n$ sample points, into a matrix form that can be handled by \emph{convex optimization} solvers. We aim to find operators $P\in \mc{B}_1(\mc{H}_{\rg\kappa}(\mbb{X}))$ and operators $M_{\pi_j}^*$ ($j=1,\cdots,d_u$) with each $\pi_j \in \rg C^s(\mbb{X})$, such that 
\begin{equation}\label{eq:stabilizing.7}
    \ipp{\rg\phi(x_i)}{P A_{0,\lambda,\theta} + A_{0,\lambda,\theta}^*P +Q}{\rg\phi(x_i)} + \sum_{j=1}^{d_u} \ipp{\rg\phi(x_i)}{ P \rg\iota^* A_{j,\lambda,\theta}M_{\pi_j}^* + M_{\pi_j}A_{j,\lambda,\theta}^* \rg\iota P}{\rg\phi(x_i)} \leq 0
\end{equation}
holds for all $i=1,\cdots,n$. To this end, since there are only a finite number of constraints regarding $\rg\phi(x_i)$ ($i=1,\cdots,n$), it suffices to solve for $P$ and $\{M_{\pi_j}\}_{j=1}^{d_u}$ such that $P A_{0,\lambda,\theta} + A_{0,\lambda,\theta}^*P +Q$ as well as $P \rg\iota^* A_{j,\lambda,\theta}M_{\pi_j}^*$ (for $j=1,\cdots,d_u$) can be restricted to an operator on $\mc{H}_{\rg\kappa}(\mbb{X})_n := \overline{\mr{span}} \BRA{\rg\phi(x_i)}_{i=1}^n$. 
Hence, we may set 
$$P = \sum_{i=1}^n \sum_{i'=1}^n \ms{P}_{ii'} \rg\phi(x_i)\times \rg\phi(x_{i'}). $$

\par For $A_{0,\lambda,\theta}$, we verify that for all $i,i'=1,\cdots,n$:
\begin{equation}\label{eq:stabilizing.synthesis.kernel.1}
    \ip{\rg\phi(x_i)}{A_{0,\lambda,\theta} \rg\phi(x_{i'})} = \lambda^2\int_0^\theta \mr{e}^{-\lambda t} \rg\kappa\bra{x_i, S_{f_0}^t(x_{i'})} \xD{t} - \lambda \rg\kappa(x_i, x_{i'}) =: \ms{A}^0_{ii'}. 
\end{equation} 
As for $M_{\pi_j}^*$, we note that $M_{\pi_j}^*\rg\phi(x_i) = \pi_j(x_i) \phi(x_i)$. In addition, $\rg\iota^*\phi(x) = \rg\phi(x)$. Then we have 
\begin{equation}\label{eq:stabilizing.synthesis.kernel.2}
\begin{aligned}
    \ip{\rg\phi(x_i)}{\rg\iota^* A_{j,\lambda,\theta}M_{\pi_j}^* \rg\phi(x_{i'})} = \pi_j(x_{i'}) \Bra{ \lambda^2\int_0^\theta \mr{e}^{-\lambda t} \rg\kappa\bra{x_i, S_{f_j}^t(x_{i'})} \xD{t} - \lambda \rg\kappa(x_i, x_{i'}) } =: \pi_j(x_{i'}) \ms{A}^j_{ii'}
\end{aligned}
\end{equation} 
Hence, by introducing the kernel matrix $\ms{G} = [\rg\kappa(x_i, x_{i'})]_{i,i'=1}^n$, denoting its $i$-th column as $\ms{g}_i$ (and hence its $i$-th row is $\ms{g}_i^\top$), and writing $\ms{q} = [\ms{q}_i]_{i=1}^n = [q(x_i)]_{i=1}^n$, we convert \eqref{eq:stabilizing.7} to 
\begin{equation*}
    \Bra{ \bra{\ms{G}\ms{P}\ms{A}^0 + \ms{A}^{0\top}\ms{P}\ms{G}} + \sum_{j=1}^{d_u} \pi_j(x_i) \bra{\ms{G}\ms{P}\ms{A}^j + \ms{A}^{j\top}\ms{P}\ms{G} } }_{ii} + \ms{q}_i \leq 0, \enspace i=1,\cdots,n.
\end{equation*}
Let us write $\ms{u}_j = (\pi_j(x_1), \cdots, \pi_j(x_n)) \in \mR^n$ and $\Diag \ms{u}_j$ the corresponding diagonal matrix. Then the condition is written as 
\begin{equation}\label{eq:stabilizing.bilinear}
    \Bra{ \bra{\ms{G}\ms{P}\ms{A}^0 + \ms{A}^{0\top}\ms{P}\ms{G}} + \sum_{j=1}^{d_u} \bra{\ms{G}\ms{P}\ms{A}^j (\Diag \ms{u}_j) + (\Diag \ms{u}_j) \ms{A}^{j\top}\ms{P}\ms{G} } }_{ii} + \ms{q}_i \leq 0, \enspace i=1,\cdots,n.
\end{equation}
At this point, the condition \eqref{eq:stabilizing.bilinear} contains bilinearity, since both $\ms{P}$ and $\ms{u}_j$ ($j=1,\cdots,d_u$) are decision variables. 

\begin{remark}[Removal of bilinearity in classical linear control]
\label{rem:bilinear-to-linear}
    It is instrumental to recall the convexification routine used in stabilizing controller synthesis for finite-dimensional linear systems. For a finite-dimensional linear system $\dot{x} = Ax + Bu$, to obtain a feedback gain $F$ such that the closed-loop system $\dot{x} = (A+BF) x$ is stabilized, we use a descent rate $q(x) = x^\top Qx$ (where $Q$ is positive definite) and formulate the condition for the quadratic Lyapunov function candidate $v(x) = x^\top Px$:
    $$P(A+BF) + (A+BF)^\top P + Q \preceq 0. $$
    The trick to remove bilinearity is to left- and right-multiply by $P^{-1}$ and then regard $P^{-1}:=X$ and $FP^{-1} = U$ as the decision variables \cite{caverly2019lmi}:
    $$(AX + BU) + (AX+BU)^\top + XQX \preceq 0,$$
    which is, by a standard Schur complement argument, equivalent to a linear matrix inequality (LMI): 
    $$\begin{bmatrix}
        -(AX + BU) - (AX+BU)^\top & XQ^{1/2} \\ Q^{1/2}X & I
    \end{bmatrix} \succeq 0.$$
    It then only boils down to solving this LMI and recovering $P = X^{-1}$ and $F = UP$. 
\end{remark}

The difficulty that we are encountering is the use of $\Diag$. Because of this, we cannot simply left- and right-multiply by the inverse Lyapunov-related term $(\ms{GPG})^{-1}$ and then combine that term with $\Diag \ms{u}_j$. Indeed, generally for a matrix $\ms{B}$, vector $\ms{u}$, and a symmetric matrix $\ms{X}$, we must admit that $\ms{B}(\Diag \ms{u})\ms{X} \not\equiv \ms{B} \Diag \ms{Xu}$. 

Now let us note that the following statement is not difficult to verify: for any symmetric matrix whose diagonal elements are all nonnegative, there is another symmetric matrix with diagonal elements all equal to zero that, when added to the first matrix, results in a positive semidefinite matrix. 
Let us denote by $\operatorname{diag}(\ms M)$ the vector formed by the diagonal elements of a matrix $\ms M$, and $\Diag \ms{M}$ the diagonal matrix that contains the diagonal of $\ms{M}$, namely the matrix obtained by deleting all off-diagonal elements of $\ms{M}$. 
Thus, under Assumption \ref{assum:kernel.matrix.PSD}, we can create (i) a fixed symmetric matrix $\ms{Q} \in \mR^{n\times n}$ such that $\diag \ms{Q} = \ms{q}$, as well as (ii) for every $j=1,\cdots,d_u$, a free matrix variable $\ms{F}^j \in \mbb{R}^{n\times n}$ that replaces the diagonal matrix $\Diag \ms{u}_j$, with the constraint that the diagonal of the summation over $j$ does not change upon left-multiplication by $\ms{G}\ms{P}\ms{A}^j$ in each of the summands. In particular, if $q(x) = \sum_{k=1}^{n_q} q_k(x)^2$ with $q_1, \cdots, q_{n_q} \in \rg H^s(\mbb{X})$, then we can let $\ms{Q}$ be the matrix with its $(i,i')$-th entry equal to $\ipp{\rg\phi(x_i)}{Q}{\rg\phi(x_{i'})}$. For example, for $q(x) = |x|^2$, we can make $\ms{Q}_{ii'} = x_i \cdot x_{i'}$ ($i,i'=1,\cdots,n$). 

That is, we convert the condition \eqref{eq:stabilizing.bilinear} into 
\begin{equation*}
\begin{aligned}
    & \bra{\ms{G}\ms{P}\ms{A}^0 + \ms{A}^{0\top}\ms{P}\ms{G}} + \sum_{j=1}^{d_u} \bra{\ms{G}\ms{P}\ms{A}^j \ms{F}^j + \ms{F}^{j\top} \ms{A}^{j\top}\ms{P}\ms{G} } + \ms{Q} \preceq 0, 
    \\
    & \sum_{j=1}^{d_u} \Diag \ms{GPA}^j \ms{F}^j = \sum_{j=1}^{d_u}  \Diag \bra{\ms{GPA}^j \Diag \ms{u}_j} = \sum_{j=1}^{d_u} \bra{\Diag \ms{GPA}^j} \Diag \ms{u}_j. 
\end{aligned}
\end{equation*}

\begin{assumption}[Characteristic kernel]
\label{assum:kernel.matrix.PSD}
    The kernel matrix $\ms{G}$ is strictly positive definite. 
\end{assumption}
This is the case as long as $0\notin\mbb{S}$ (in the computation phase, although in theoretical error analysis we can supplement $0$ to $\mbb{S}$), since the Sobolev kernel is a characteristic kernel \cite{steinwart2008support}. 
At this point, we are ready to apply the conventional trick in Remark \ref{rem:bilinear-to-linear}. 
We left- and right-multiply the matrix inequality by $(\ms{G}\ms{P}\ms{G})^{-1}$, do the variable substitution: 
$$\ms{X}:= (\ms{G}\ms{P}\ms{G})^{-1}, \quad 
\ms{F}^j \ms{X} = \ms{U}^j, \, j=1,\cdots,d_u, $$
and denote 
\begin{equation}\label{eq:stabilizing.AB}
    \ms G^{-1} \ms A^0 =: \ms{A}, \ 
    \ms G^{-1} \ms A^j =: \ms{B}^j, \, j=1,\cdots,d_u.
\end{equation}
Then the matrix inequality can be expressed as 
\begin{equation*}
\begin{aligned}
    & \bra{\ms{A}\ms{X} + \ms{X}\ms{A}^\top} + \sum_{j=1}^{d_u} \bra{\ms{B}^j\ms{U}^j + \ms{U}^{j\top}\ms{B}^{j\top} } + \ms{XQX} \preceq 0.  
\end{aligned}
\end{equation*}
By using the Schur complement technique again, the matrix inequality becomes a linear matrix inequality: 
\begin{equation}\label{eq:stabilizing.convex}
    \begin{bmatrix}
    -\bra{\ms{A}\ms{X} + \ms{X}\ms{A}^\top} - \sum_{j=1}^{d_u} \bra{\ms{B}^j\ms{U}^j + \ms{U}^{j\top}\ms{B}^{j\top} } & \ms{X}\ms{Q}^{1/2} \\ \ms{Q}^{1/2}\ms{X} & \ms{I}
\end{bmatrix} \succeq 0 .
\end{equation}
Solving \eqref{eq:stabilizing.convex} requires only a semidefinite program (SDP) solver.\footnote{In the Schur complement reformulation, the $Q^{1/2}$ can be replaced by $L$ in the bottom-right block and by $L^\top$ in the lower-left block, where $L$ is specified by the Cholesky decomposition $Q=LL^\top$. } 

To facilitate the solution, we also supply an objective to form a complete SDP problem and conditioning constraints to avoid the trivial solution ($\ms{X}=0$, which is an artifact resulting from inverting $\ms{GPG}$) or an unbounded solution: 
\begin{equation}\label{eq:stabilizing.SDP}
    \min -\log \det \ms{X} \quad \mr{s.t.} \, \eqref{eq:stabilizing.convex}, \, \alpha_{\ms{X}}\ms{I} \preceq \ms{X} \preceq \beta_{\ms{X}}\ms{I}, \, \ms{GXG} \preceq \gamma_{\ms{X}}\ms{P}_0^{-1}.
\end{equation}
The objective of maximizing $\log\det \ms{X}$ is motivated by the fact that the diagonal of $\ms{GPG}$ stands for the Lyapunov function values at the sample points. As we set an ``at least'' decay rate, we hope for an ``at most'' Lyapunov function to prevent the synthesis of an over-aggressive controller. Here $\ms{P}_0$ represents a reference Lyapunov function, e.g., that of the linearized LQR, for which $\ms{G}\ms{P}_0\ms{G} = [x_i^\top P_\mr{LQR}\, x_{i'}]_{i,i'=1}^n$. The last constraint is therefore read as $\ms{P} \succeq (1/\gamma_{\ms{X}})\ms{P}_0$, namely that the Lyapunov functions should still be ``large enough''. 
Once this optimization problem is solved, we recover 
\begin{equation}\label{eq:stabilizing.recovery.1}
    \ms{P} = (\ms{GXG})^{-1} 
\end{equation}
and $\ms{u}_j$ ($j=1,\cdots,d_u$) from the following equations ($n$ in total): 
\begin{equation*}
    \sum_{j=1}^{d_u} \Diag \ms{GPG}\ms{B}^j \ms{U}^j\ms{GPG} = \sum_{j=1}^{d_u}  \bra{\Diag \ms{GPG}\ms{B}^j} \Diag \ms{u}_j.
\end{equation*}
For the above equation on $\ms{u}_j$ ($j=1,\cdots,d_u$) to be solvable, we need to make the following \textit{ad hoc} assumption on the solution $\ms{P}$ corresponding to the Lyapunov function, and thus $\ms{GPA}^j$ regarding the Yosida-approximated rate of change caused by the actuation in each input component. 
\begin{assumption}[Actuatability of Lyapunov function at sample points]\label{assum:recoverability}
    For each $i=1,\cdots,n$, there must be at least one among $(\ms{GPA}^j)_{ii}$ ($j=1,\cdots,d_u$) that is nonzero. 
\end{assumption}
Thus, the recovery of $\ms{u}_j$ ($j=1,\cdots,d_u$) can follow such a way that resembles a constructive formula of Lyapunov-based control law:
\begin{equation}\label{eq:stabilizing.recovery.2}
     (\ms{u}_j)_i = \bra{\ms{GPG}\ms{B}^j}_{ii}  \frac{\sum_{k=1}^{d_u} \bra{\ms{GPG}\ms{B}^k\ms{U}^k \ms{GPG}}_{ii}}{\sum_{k=1}^{d_u} \bra{\ms{GPG}\ms{B}^k}_{ii}^2} , 
     \quad i=1,\cdots,n.
\end{equation}
Nevertheless, if Assumption \ref{assum:recoverability} is not satisfied, we can still recover by finding any feasible solution for $\ms{u}_j$ ($j=1,\cdots,d_u$) that satisfies \eqref{eq:stabilizing.bilinear}. 

Let us summarize the steps as below.
\begin{enumerate}[nolistsep]
    \item Compute the kernel matrix $\ms{G}$, the ``kernel velocity'' matrices $\ms{A}^0$ by \eqref{eq:stabilizing.synthesis.kernel.1} and $\ms{A}^j$ ($j=1,\cdots,d_u$) by \eqref{eq:stabilizing.synthesis.kernel.2}, and thus $\ms{A}$ and $\ms{B}^j$ ($j=1,\cdots,d_u$) by \eqref{eq:stabilizing.AB}. 
    \item Calculate decay rates at all sample points $\ms{q}$, and create matrix $\ms{Q}$.
    \item Solve the SDP \eqref{eq:stabilizing.SDP} for $(\ms{X}, \ms{U}^1, \cdots, \ms{U}^{d_u})$. 
    \item Obtain $\ms{P}$ from \eqref{eq:stabilizing.recovery.1} and $\ms{u}_j$ ($j=1,\cdots,d_u$) from  \eqref{eq:stabilizing.recovery.2}. 
    \item Thus, the control Lyapunov function is 
    $$v(x) = \ipp{\rg\phi(x)}{\sum_{i=1}^n \sum_{i'=1}^n \ms{P}_{ii'}\rg\phi(x_i)\times \rg\phi(x_{i'})}{\rg\phi(x)} = \sum_{i=1}^n \sum_{i'=1}^n \ms{P}_{ii'}\rg\kappa(x, x_i) \rg\kappa(x, x_{i'}) = \ms{g}(x)^\top \ms{P}\ms{g}(x), $$
    where $\ms{g}(x) = \Bra{\rg\kappa(x,x_i)}_{i=1}^n$, and the control input at each sample point is determined: $\pi_j(x_i) = (\ms{u}_j)_i$ ($i=1,\cdots,n$, $j=1,\cdots,d_u$). 
\end{enumerate}

\begin{remark}[Removability of bilinearity due to linear--Sobolev RKHS being an algebra]
\label{rem:algebra}
    The formulation of a fully convexified problem \eqref{eq:stabilizing.convex} is essentially due to the merit of infinite-dimensional Sobolev and linear--Sobolev RKHSs. 
    Superficially, in the RKHS, the actions of the adjoints of multiplication operators ($M_{\pi_j}^*$) on the canonical features $\rg\phi(x)$ result only in a multiplication of the policy evaluated at the state ($\pi_j(x)\phi(x)$). Thus, the trick of merging $P^{-1}$ and the policy, as used for finite-dimensional linear systems (see Remark \ref{rem:bilinear-to-linear}), can be applied in a formally similar way. 
    The ``bilinearity'' issue encountered in the literature (e.g., in Goswami et al. \cite{goswami2021bilinearization} or Str{\"a}sser et al. \cite{strasser2024koopman}) is avoided because $\pi_1, \cdots, \pi_{d_u}$ do not need to be written as a linear form of $\rg\phi(x)$ (e.g., $\pi(x) = F\rg\phi(x)$ for some feedback gain operator $F$), and even if they are written in this way, \emph{a bilinear form of $\rg\phi(x)$, as a $\rg H^s$-class function of $x$, is still a linear form of $\rg\phi(x)$}. In other words, the linear--Sobolev function space itself is a \emph{vector space and also an algebra under multiplication}.  
    In contrast, this algebra property does not hold under any finite-dimensional feature. That is, where a finite dictionary of basis functions or Koopman eigenfunctions ($\{\psi_i\}_i$) is used, the resulting bilinearity in features $\pi_j(x) L_j\psi_i(x) = F_j \psi(x) L_j\psi_i(x)$ tends to be no longer reducible to linear combinations of $\{\psi_i\}_i$. The use of kernel quadratic and SOS forms to write out the closed-loop Lyapunov descent conditions is also largely impossible in general. 
\end{remark}

\begin{remark}[Recovery of policy from sample points]
\label{rem:policy.recovery}
    We note that the solution of the data-dependent constraints \eqref{eq:stabilizing.convex} in fact only determines the policy values at the sampled points, instead of the policy as a $d_u$-dimensional vector-valued function over $\mbb{X}$. Since we only require that $\pi\in [\rg C^s(\mbb{X})]^{d_u}$, we can use any interpolation methods on each of $\pi_j$ ($j=1,\cdots,d_u$) with the given values. 
    For example, we can use a ``linear--Gaussian kernel'' (where the Gaussian kernel generates an RKHS of $C^\infty$ functions, and the linear part provides the component functions as factors): $\rg\kappa_{\mr{Gauss}} (x,x') = (x\cdot x')\exp\bra{-|x-x'|^2/\ell^2}$, or a linear--Sobolev kernel with a higher smoothness index $s'=s+d_x/2$, and perform a kernel interpolation. In other words, we let 
    $$ \pi_j(x) = \sum_{i=1}^n \ms{c}_{ij} \rg\kappa_{\mr{Gauss}} (x,x_i), $$
    and force $\pi_j(x_i)$ to equal the determined values from \eqref{eq:stabilizing.convex}.  
\end{remark}

\section{Synthesis of Optimal Controllers: A Koopman--Riccati Method}
Naturally, one needs more than stabilization in controller synthesis. When a cost metric is associated with the policy, we have an optimal control problem, which is what we consider in this section. 
\begin{problem}[Optimal control]\label{prob:optimal}
    Find a control policy $\pi \in [\rg C^s(\mbb{X})]^{d_u}$, such that for any $x\in \mbb{X}$, the closed-loop orbit issued from $x$ achieves the minimum accumulated cost ($l(x, u) = q(x)+|u|^2$) over the infinite horizon:
    \begin{equation}\label{eq:optimal.0}
            v(x) = \inf_{\pi \in [\rg C^s(\mbb{X})]^{d_u}} \, v_\pi(x), \quad v_\pi(x) = \int_0^\infty \bra{ q\bra{S_{f_\pi}^t(x)} + 
            \sum_{j=1}^{d_u} \pi_j\bra{S_{f_\pi}^t(x)}^2} \,\xD{t}, 
    \end{equation}
    in which the state-dependent \emph{stage cost} $q(x)$ is a kernel quadratic form $q(x) = \ipp{\rg\phi(x)}{Q}{\rg\phi(x)}$, $Q\succeq 0$. 
\end{problem}
Even though the infimum is not necessarily attainable in the class of $\pi \in [\rg C^s(\mbb{X})]^{d_u}$, in a data-driven setting, we are only interested in reaching a small bounded error from the infimum. 

\subsection{Optimality Condition and Operator Riccati Equation}
For a given control policy $\pi \in [\rg C^s(\mbb{X})]^{d_u}$, since $\rg C^s(\mbb{X}) \subseteq \rg H^s(\mbb{X})$, the \emph{policy evaluation} gives:
$$v_\pi(x) = \int_0^\infty \Bra{ \ipp{T_\pi^t \rg\phi(x)}{Q}{T_\pi^t \rg\phi(x)} + \sum_{j=1}^{d_u} \abs{\ip{\pi_j}{T_\pi^t \rg\phi(x)}}^2} \xD{t} = \int_0^\infty \ipp{T_\pi^t \rg\phi(x)}{Q + \sum_{j=1}^{d_u}\pi_j \times \pi_j}{T_\pi^t \rg\phi(x)} \xD{t} $$
if the integral on the right-hand side is indeed finite. In this case, $v_\pi(x)$ is a kernel SOS form, which can be represented by
\begin{equation}\label{eq:policy.evaluation}
    v_\pi(x) = \ipp{\rg\phi(x)}{P_\pi}{\rg\phi(x)}, \ \text{ where } P_\pi = \int_0^\infty T_\pi^{t*} \bra{Q + \sum_{j=1}^{d_u}\pi_j \times \pi_j} T_\pi^t \xD{t}.
\end{equation} 
To deduce an optimality condition, we should first examine the structure of the set of admissible policies. We let $Q\in \mc{B}_1(\rg H^s(\mbb{X}))$ be a fixed symmetric positive semidefinite operator, and denote the collection of all ``kernel SOS stabilizing'' policies as  
$$ \mc{S} = \BRA{\pi \in [\rg C^s(\mbb{X})]^{d_u} : \int_0^\infty T_\pi^{t*} \bra{Q + \sum_{j=1}^{d_u}\pi_j \times \pi_j} T_\pi^t \xD{t} \in \mc{B}_1(\rg H^s(\mbb X))}, $$
and the collection of all ``exponentially stabilizing'' policies as 
$$ \mc{S}_\mr{e} = \BRA{\pi \in [\rg C^s(\mbb{X})]^{d_u} : \|T_\pi^t\| \lesssim \mr{e}^{\omega t} \text{ for some } \omega<0} \subseteq \mc{S}.$$

\begin{proposition}[Geometry of the classes of kernel SOS stabilizing and exponentially stabilizing policies] \label{prop:optimal.geometry}
    Suppose that there exists a constant $c_Q>0$ such that $q(x) \geq c_Q|x|^2$ for all $x\in \mbb{X}$. 
    Then $\mc{S}$ is a connected set in the $\rg C^s$ topology, and $\mc{S}_\mr{e} \subseteq \mc{S}$ is an open set in which $P_\pi$ is continuous with respect to $\pi$. 
\end{proposition}
\begin{proof}
    Note that for any $x\in \mbb{X}$, $\ipp{\rg\phi(x)}{\pi_j\times \pi_j}{\rg\phi(x)} = |\pi_j(x)|^2$, and hence for $\pi_j \in \rg H^s(\mbb{X})$, $|\pi_j(x)|^2 \lesssim |x|^2$. We assert that for sufficiently small $a>0$, $Q - a\pi_j\times \pi_j \in \mc{B}_1$ and is positive semidefinite. To show this, we let $Q_0$ be the finite-rank (and hence trace-class) operator such that $\ipp{\rg\phi(x)}{Q_0}{\rg\phi(x)} = |x|^2$. 
    For notational simplicity, consider $d_x=1$. By writing $\pi_j(x) = x\bar\pi_j(x)$ with $\bar\pi_j \in H^s(\mbb{X})$, we get $|x|^2 - a|\pi_j(x)|^2 = |x|^2 - a|x \bar\pi_j(x)|^2 = |x|^2 (1-a\bar\pi_j(x)^2)$. To see that $1-a\bar\pi_j(x)^2$ is a kernel SOS form under the Sobolev kernel, we choose $c_j$ as any constant, which is between $(0, 1)$ such that $1-a\bar\pi_j(x)^2 > c_j$ for all $x\in \mbb{X}$, and rewrite 
    $$ \frac{1-a\bar\pi_j(x)^2}{c_j} = \frac{1}{1-\omega_j(x)^2}, \, \text{ where } \omega_j(x) = \sqrt{1-\frac{c_j}{1-a^2\bar\pi_j(x)^2}}.$$ 
    It then can be seen that $\omega_j \in H^s(\mbb{X})$ and therefore the right-hand side can be expressed as $1+\omega_j^2 + \omega_j^4 + \cdots$, as a kernel SOS form. Therefore, $|x|^2 - a|\pi_j(x)|^2$ is a kernel SOS form under the linear--Sobolev kernel and hence $Q_0 - a\pi_j\times \pi_j$ is a kernel SOS form. The same argument applies to $Q-aQ_0$ when $a$ is sufficiently small. Thus, the assertion is proved, which implies that 
    $$\mc{S} = \BRA{\pi \in [\rg C^s(\mbb{X})]^{d_u} : \int_0^\infty T_\pi^{t*}Q T_\pi^t \xD{t}\in \mc{B}_1(\rg H^s(\mbb X))}. $$
    In other words, the kernel SOS stabilizing policies are such that in closed loop, the state cost is finite, as the input cost is automatically dominated by the former. 
    \par Then we see that $\pi\in \mc{S}$ if and only if there exists a $0\preceq S_\pi\in \mc{B}_1$ such that the Lyapunov inequality 
    $$S_\pi A_\pi + A_\pi^*S_\pi + Q \preceq 0$$
    holds when restricted to $\mc{D}(A_\pi)$, i.e., for all $z\in \mc{D}(A_\pi)$, $\ipp{z}{S_\pi A_\pi + A_\pi^*S_\pi + Q}{z}\leq 0$. By defining the unbounded operators $\tilde{S}_\pi = S_\pi^{-1}$ and $\tilde{Q} = Q^{-1}$, the above inequality can be rewritten as $A_\pi \tilde{S}_\pi + \tilde{S}_\pi A_\pi^* + \tilde{S}_\pi \tilde{Q} \tilde{S}_\pi \preceq 0$. Recalling the affine relation \eqref{eq:PF.generator.closed-loop} and by introducing $\tilde{M}_j = M_{\pi_j}^*\tilde{S}_\pi$, we have the following equivalent Lyapunov inequality
    $$\begin{bmatrix}
        A_0\tilde{S}_\pi + \sum_{j=1}^{d_u} \rg\iota^* A_j\tilde{M}_j & \tilde{S}_\pi \\ \tilde{S}_\pi & \tilde{Q}
    \end{bmatrix} \preceq 0. $$
    Now we see that the set of solutions for $(\tilde{S}_\pi, \tilde{M}_j)$, if they exist, must form a convex set (and hence a connected one). The relation between $(\pi, S_\pi)$ and $(\tilde{S}_\pi, \tilde{M}_1, \cdots, \tilde{M}_{d_u})$ is continuous. Thus $\mc{S}$ is the continuous pre-image of a convex set. 
    \par Suppose that $\pi\in \mc{S}_\mr{e}$ and hence for some $\omega_\pi<0$, $\|T_\pi^t\| \leq c_\pi\mr{e}^{\omega_\pi t}$. Consider a small $\rg C^s$-perturbation $\varpi$ of the policy $\pi$, which results in an $\rg H^s$-norm-bounded perturbation of $A_\pi$. Let $\|A_{\pi+\varpi} - A_\pi \|\leq b$. According to the bounded perturbation theorem (Phillips' theorem) of strongly continuous semigroups, we have $\|T_{\pi+\varpi}^t\| \leq c_\pi\mr{e}^{(\omega_\pi + c_\pi b)t}$ and hence when $b$ is small enough, $\omega_\pi + c_\pi b$ is still negative. Thus $\pi+\varpi \in \mc{S}_\mr{e}$. 
    Upon this perturbation, 
    $$\|T_{\pi+\varpi}^t-T_\pi^t\| \leq c_\pi \mr{e}^{\omega_\pi t}\bra{\mr{e}^{c_\pi bt}-1} \lesssim \mr{e}^{(\omega_\pi+c_\pi b)t} \to 0, \ t\to\infty.$$
    Hence, it can be shown that $\|P_{\pi+\varpi} - P_\pi\|_{\mc{B}_1} \leq b$. This implies the continuity of $P_\pi$ with respect to $\pi$. 
\end{proof} 

The above geometric properties give hope for differentiating the policy evaluation $P_\pi$ with respect to $\pi$, when $\pi$ is an interior point of $\mc{S}$. Since we only proved that $\mc{S}_\mr{e}$ is an open subset of $\mc{S}$, we shall focus on the class of exponentially stabilizing policies, in order to derive an optimality condition by differentiation.  
\begin{proposition}[Differentiability of cost evaluation]
    On $\mc{S}_\mr{e}$, $P_\pi$ is Fr{\'e}chet differentiable with respect to $\pi$, i.e., the Fr{\'e}chet derivative exists in strong operator topology on $\mc{D}(A_\pi)$:
    $$\mr{D}P_\pi(\varpi) = \lim_{\varepsilon\downarrow 0} \frac{1}{\varepsilon}\bra{ P_{\pi+\varepsilon\varpi} - P_\pi}. $$
    Precisely, $\forall z\in \mc{D}(A_\pi)$, $\lim_{\varepsilon\downarrow 0} \frac{1}{\varepsilon}\bra{ P_{\pi+\varepsilon\varpi}z - P_\pi z}$ exists in $\rg H^s(\mbb{X})$. 
    It satisfies the following Lyapunov equation of unbounded operators:
    \begin{equation}\label{eq:Frechet.derivative}
        \mr{D}P_\pi(\varpi)A_\pi + A_\pi^* \mr{D}P_\pi(\varpi) + \sum_{j=1}^{d_u} \bra{P_\pi \rg\iota^* A_jM_{\varpi_j}^* + M_{\varpi_j} A_j^\ast \rg\iota P_\pi + \varpi_j\times\pi_j + \pi_j\times\varpi_j} = 0 \text{ on } \mc{D}(A_\pi). 
    \end{equation}
\end{proposition}
\begin{proof}
    For $P_\pi$ we have the following Lyapunov equation
    $$P_\pi A_\pi + A_\pi^*P_\pi + \bra{Q + \sum_{j=1}^{d_u}\pi_j \times\pi_j} = 0 \text{ on } \mc{D}(A_\pi).$$
    This is derived from its definition \eqref{eq:policy.evaluation} by left-multiplying $A_\pi^*$ and right-multiplying $A_\pi$, using the relation $\xDD{}{t}T_\pi^t = T_\pi^tA_\pi = A_\pi T_\pi^t$, and integrating. 
    Then, imposing a perturbation from $\pi$ to $\pi + \varepsilon \varpi$ with $\varepsilon>0$ and $\varpi\in [\rg C^s(\mbb{X})]^{d_u}$, we have the Lyapunov equation corresponding to $P_{\pi+\ve\varpi}$. Subtracting the two equations, we obtain
    $$\begin{aligned}
        \bra{P_{\pi+\ve\varpi} - P_\pi} A_{\pi+\ve\varpi} + A_{\pi+\ve\varpi}^* \bra{P_{\pi+\ve\varpi}-P_\pi} + P_\pi \bra{A_{\pi+\ve\varpi} - A_\pi} + \bra{A_{\pi+\ve\varpi}^* - A_\pi^*}P_\pi & \\
        + \sum_{j=1}^{d_u}\bra{\ve \varpi_j \times\pi_j + \ve \pi_j\times \varpi_j + \ve^2 \varpi_j\times\varpi_j} &= 0 \text{ on } \mc{D}(A_\pi).
    \end{aligned} $$
    With $A_{\pi+\ve\varpi} - A_\pi = \ve \sum_{j=1}^{d_u} \rg\iota^* A_j M_{\varpi_j}^*$ on $\mc{D}(A_\pi)$, dividing by $\ve$, we have
    $$\begin{aligned}
        \frac{1}{\ve}\bra{P_{\pi+\ve\varpi} - P_\pi} A_{\pi+\ve\varpi} + A_{\pi+\ve\varpi}^* \frac{1}{\ve}\bra{P_{\pi+\ve\varpi}-P_\pi} 
        + P_\pi \sum_{j=1}^{d_u} \rg\iota^* A_j M_{\varpi_j}^* + \sum_{j=1}^{d_u} M_{\varpi_j} A_j^\ast \rg\iota P_\pi & \\
        + \sum_{j=1}^{d_u}\bra{\varpi_j \times\pi_j + \pi_j\times \varpi_j + \ve\varpi_j\times\varpi_j} &= 0 \text{ on } \mc{D}(A_\pi).
    \end{aligned} $$
    When $\varepsilon\to0$, combining the summations gives a finite-rank operator on $\mc{D}(A_\pi)$, and $A_{\pi+\ve\varpi}$ converges to $A_\pi$ strongly on $\mc{D}(A_\pi)$. Thus, $\frac{1}{\ve}\bra{P_{\pi+\ve\varpi} - P_\pi}$ has a limit in strong operator topology as $\ve\to 0$, and the Lyapunov equality that it satisfies is as in the conclusion. This completes the proof. 
\end{proof}

\begin{corollary}[Stationarity condition]
    Suppose that $\pi\in \mc{S}_\mr{e}$. Then $\mr{D}P_\pi(\varpi) = 0$ for all $\varpi\in [\rg C^s(\mbb{X})]^{d_u}$ if and only if 
    \begin{equation}\label{eq:optimal.stationary}
        P_\pi \rg\iota^* A_jM_{\varpi_j}^* + M_{\varpi_j}A_j^* \rg\iota P_\pi + \varpi_j\times \pi_j + \pi_j\times\varpi_j = 0 \text{ on } \mc{D}(A_\pi), \ \forall \varpi_j\in \rg C^s(\mbb{X}), j=1,\cdots,d_u. 
    \end{equation}
    That is, $(P_\pi \rg\iota^* A_jM_{\varpi_j}^* + M_{\varpi_j}A_j^* \rg\iota P_\pi + \varpi_j\times \pi_j + \pi_j\times\varpi_j)z = 0$ for all $z\in \mc{D}(A_\pi)$, or equivalently, $\ipp{z}{P_\pi \rg\iota^* A_jM_{\varpi_j}^* + M_{\varpi_j}A_j^* \rg\iota P_\pi + \varpi_j\times \pi_j + \pi_j\times\varpi_j}{z'}=0$ for all $z,z'\in \mc{D}(A_\pi)$. 
\end{corollary}

To computationally impose the stationarity condition \eqref{eq:optimal.stationary} as a matrix algebraic Riccati equation, we will again adopt the Yosida approximation and its data-based version. The Yosida approximation aims to remove the difficulty of unbounded operators that were obtained as a by-product of differentiation, and the data-based approximation reduces the operator condition to one on a matrix.

\subsection{Yosida Approximation and Its Data-Based Approximation}
First, by finite-time Yosida approximations, we replace the $A_j$ ($j=1,\cdots, d_u$) by $A_{j,\lambda,\theta}$, and hence arrive at the following approximate form of the stationarity condition:
$$ P_\pi \rg\iota^* A_{j, \lambda,\theta}M_{\varpi_j}^* + M_{\varpi_j} A_{j, \lambda,\theta}^* \rg\iota P_\pi + \varpi_j\times \pi_j + \pi_j\times\varpi_j = 0,\quad\forall \varpi_j\in \rg C^s(\mbb{X}), j=1,\cdots,d_u. $$
The condition is no longer restricted to $\mc{D}(A_\pi)$ as a dense subspace of $\rg H^s(\mbb{X})$, since all operators involved in the approximation are bounded. 
In other words, the operator that is set to zero in the above formula is interpreted as $\ipp{z}{\cdot}{z'} = 0$ for all $z,z'\in \rg H^s(\mbb X)$. Since $\rg H^s(\mbb{X})$ is the closed span of all $\rg\phi(x)$, $x\in \mbb{X}$, we can write 
\begin{equation}\label{eq:optimal.2}
    \ipp{\rg\phi(x)}{P_\pi \rg\iota^* A_{j, \lambda,\theta} M_{\varpi_j}^* + M_{\varpi_j} A_{j, \lambda,\theta}^* \rg\iota P_\pi + \varpi_j\times \pi_j + \pi_j\times\varpi_j}{\rg\phi(x')} = 0, \ \forall x,x'\in \mbb{X}, \ \forall \varpi_j\in \rg C^s(\mbb{X}), j=1,\cdots,d_u. 
\end{equation}
In fact, the operator that is set to zero belongs to the trace class, since $P_\pi\in \mc{B}_1(\rg H^s(\mbb{X}))$ when $\pi\in \mc{S}_\mr{e}$. 
It thus determines a kernel quadratic form as a bivariate function of $x$ and $x'$ that equals $0$ everywhere on $\mbb{X} \times \mbb{X}$. 

\par By condition \eqref{eq:optimal.2}, if we let $x'=x$, then by calculating $\ipp{\rg\phi(x)}{\pi_j\times\varpi_j}{\rg\phi(x)} = \pi_j(x)\varpi_j(x)$ and $M_{\varpi_j}^*\rg\phi(x) = \varpi_j(x)\phi(x)$, we arrive at
\begin{equation*}
    \varpi_j(x) \Bra{ 2\pi_j(x) + 2\ip{\rg\phi(x)}{P_\pi \rg\iota^*A_{j,\lambda,\theta} \phi(x)} } = 0, \ \forall x\in \mbb{X}, \ \forall \varpi_j\in \rg C^s(\mbb{X}), j=1,\cdots,d_u. 
\end{equation*}
Due to the arbitrariness of the choice of $\varpi_j$, the policy that meets the Yosida-approximated stationarity condition is linked to the cost operator $P_\pi$ according to:
\begin{equation}\label{eq:optimal.policy.Yosida}
    \pi_j(x) = -\ip{\rg\phi(x)}{P_\pi \rg\iota^*A_{j,\lambda,\theta} \phi(x)}, \, \forall x\in \mbb{X}, \ j=1,\cdots,d_u. 
\end{equation}
\begin{remark}[Lower regularity than required after Yosida approximation]
    The right-hand side determined above may not be a $\rg C^s(\mbb{X})$-policy, but instead only a $\rg H^s(\mbb{X})$ one. The regularity is deteriorated due to the Yosida approximations, which allows the flow in the RKHS to be well-defined even without a $\rg C^s(\mbb{X})$-policy; indeed, since $A_{0,\lambda,\theta}$ and $A_{j,\lambda,\theta}$ ($j=1,\cdots,d_u$) are all bounded operators, the constructed $A_{\pi,\lambda,\theta}$ is a bounded operator on $\rg H^s(\mbb{X})$ as long as $\pi\in [\rg H^s(\mbb{X})]^{d_u}$. This is, however, not allowed by the original constraint \eqref{eq:optimal.stationary}. 
\end{remark}

We know from earlier discussions that the policy evaluation gives the Lyapunov equality that relates $P_\pi$ to the policy $\pi$: $P_\pi A_\pi + A_\pi^*P_\pi + \bra{Q + \sum_{j=1}^{d_u}\pi_j \times\pi_j} = 0$ on $\mc{D}(A_\pi)$. By Yosida approximation, it is approximated as $P_\pi A_{\pi,\lambda,\theta} + A_{\pi,\lambda,\theta}^* P_\pi + \bra{Q + \sum_{j=1}^{d_u}\pi_j \times\pi_j} = 0$. That is, 
$$ \ipp{\rg\phi(x)}{P_\pi A_{\pi,\lambda,\theta} + A_{\pi,\lambda,\theta}^* P_\pi + Q + \sum_{j=1}^{d_u}\pi_j \times\pi_j}{\rg\phi(x')} = 0, \ \forall x,x'\in \mbb{X}.$$
By plugging in the policy $\pi$ as determined in \eqref{eq:optimal.policy.Yosida}, we have the following equality as the approximate stationarity condition pointwise on $x\in \mbb{X}$:
\begin{equation}\label{eq:optimal.quartic}
    \begin{aligned}
        & \ipp{\rg\phi(x)}{P_\pi A_{0,\lambda,\theta} + A_{0,\lambda,\theta}^*P_\pi + Q}{\rg\phi(x')} \\
        &\quad - \sum_{j=1}^{d_u} \ip{\rg\phi(x')}{P_\pi \rg\iota^* A_{j,\lambda,\theta}\phi(x)} \ip{\rg\phi(x)}{P_\pi \rg\iota^* A_{j,\lambda,\theta}\phi(x)} \\
        & \quad - \sum_{j=1}^{d_u} \ip{\rg\phi(x)}{P_\pi \rg\iota^* A_{j,\lambda,\theta}\phi(x')} \ip{\rg\phi(x')}{P_\pi \rg\iota^* A_{j,\lambda,\theta}\phi(x')} \\
        & \quad + \sum_{j=1}^{d_u} \ip{\rg\phi(x)}{P_\pi \rg\iota^* A_{j,\lambda,\theta}\phi(x)} \ip{\rg\phi(x')}{P_\pi \rg\iota^* A_{j,\lambda,\theta}\phi(x')} = 0, \ \forall x, x'\in \mbb{X}. 
    \end{aligned}
\end{equation}
At first glance, each term under the summation symbol involves a canonical feature of $x$ twice and a canonical feature of $x'$ twice, thus appearing to be a ``quartic'' form. However, we shall again realize that bilinearity in infinite-dimensional RKHSs is only an artifact, as already pointed out in Remark \ref{rem:algebra}. For example, the bilinear term $\ip{\rg\phi(x)}{P_\pi \rg\iota^* A_{j,\lambda,\theta}\phi(x)}$ is clearly a function of $x$ that belongs to the linear--Sobolev space $\rg H^s(\mbb{X})$; hence there must be a unique vector in $\rg H^s(\mbb{X})$ such that this bilinear term is the inner product of such a vector and $\rg\phi(x)$; as such, a bilinear form is a linear form. Even when it encounters $\rg\phi(x)$ again, the trilinear form that has been reduced to a bilinear form is further reduced to a linear form. 

\par The following notions conceptually explain the procedure to reduce \eqref{eq:optimal.quartic} to a bilinear form in $\rg\phi(x)$ and $\rg\phi(x')$. 
\begin{definition}[The trick of $\flat$-transform]\label{def:musical-flat}
    For any given operator $\Omega \in \mc{B}(H^s(\mbb{X}), \rg H^s(\mbb{X}))$, $x\mapsto \ip{\rg\phi(x)}{\Omega \phi(x)}$ is a function in $\rg H^s(\mbb{X})$ and hence (by the Riesz--Fr\'echet theorem) there exists a unique $\omega\in \rg H^s(\mbb{X})$ such that 
    $$\ip{\omega}{\rg\phi(x)} = \ip{\rg\phi(x)}{\Omega\phi(x)}, \, \forall x\in \mbb{X}.$$
    This $\omega$ determined by $\Omega$ is denoted as $\Omega^\flat$. (We use the flat symbol to indicate the ``flattening'' of an ``operator'' to a ``vector''.) 
\end{definition}
It becomes obvious that the flattening mapping $\flat$ maps any operator of the form $g\times h$ ($g\in \rg H^s(\mbb{X})$, $h\in H^s(\mbb{X})$) into the element $gh$ in $\rg H^s(\mbb{X})$, i.e., $(g\times h)^\flat = gh$. Hence, $$\ip{\rg\phi(x)}{P_\pi \rg\iota^* A_{j,\lambda,\theta}\phi(x)} = \ip{\rg\phi(x)}{\bra{P_\pi \rg\iota^* A_{j,\lambda,\theta}}^\flat}. $$
Then we can also reduce the ``kernel cubic form'' $\Omega \phi(x)\times \rg\phi(x) \Omega\phi(x)$ into a quadratic form, and subsequently into a linear form. 
\begin{definition}[Tensor product of a vector and an operator, and its $\flat$-transform]\label{def:tensor}
    For any given $\psi\in \rg H^s(\mbb{X})$ and $\Omega \in \mc{B}(H^s(\mbb{X}), \rg H^s(\mbb{X}))$, define $\psi\times \Omega$ as the linear mapping that satisfies $(\psi\times\Omega)g = \psi\times \Omega g \in \mc{B}(\rg H^s(\mbb{X}), \rg H^s(\mbb{X}))$ for all $g\in H^s(\mbb{X})$, called a \emph{third-order tensor}. 
    Thus, $\ip{\rg\phi(x)}{\bra{(\psi\times \Omega)\phi(x)}h}$ defines a bounded linear functional on $h\in \rg H^s(\mbb{X})$ that is also a $\rg H^s$-function of $x$, which we write as $\ip{h}{(\psi\times\Omega)^\flat \rg\phi(x)}$. 
\end{definition}

Now, $\Omega\phi(x)\times \rg\phi(x)\Omega\phi(x) = \Omega \phi(x) \ip{\omega}{\rg\phi(x)}$, which brings any $g$ to $\ip{\rg\phi(x)}{\omega} \ip{\Omega \phi(x)}{g} = \ip{\rg\phi(x)}{(\omega\times \Omega\phi(x))g }$. Hence, 
$$\Omega\phi(x)\times \rg\phi(x)\Omega\phi(x) = ((\omega\times \Omega)\phi(x))^* \rg\phi(x) = (\omega\times\Omega)^\flat \rg\phi(x) = (\Omega^\flat \times\Omega)^\flat \rg\phi(x).$$ 
Consequently, the kernel quartic forms in \eqref{eq:optimal.quartic} boil down to kernel quadratic forms, which result in the following \emph{operator Riccati equation}:
\begin{equation}\label{eq:optimal.Riccati}
\begin{aligned}
    \ipp{\rg\phi(x)}{P_\pi A_{0,\lambda,\theta} + A_{0,\lambda,\theta}^*P_\pi + \sum_{j=1}^{d_u} \Bra{ 
    \Omega_j^\flat \times \Omega_j^\flat - 
    (\Omega_j^\flat \times\Omega_j)^\flat - 
    (\Omega_j^\flat \times\Omega_j)^{\flat *} 
    } + Q }{\rg\phi(x')}= 0, \ \\
    \qquad \forall x,x'\in \mbb{X}, \ \text{ where } \Omega_j = P_\pi \rg\iota^* A_{j,\lambda,\theta}, \, j=1,\cdots,d_u. 
\end{aligned}
\end{equation}
The condition is still called a Riccati equation for $P_\pi$, since for the operators in the middle, the first two terms depend on $P_\pi$ linearly like in a Lyapunov equality, and the term under summation depends on $\Omega_j$ bilinearly, even though in our case, it can no longer be conveniently expressed as $-P_\pi(\star)P_\pi$ for some symmetric operator $(\star)$ as in the usual finite-dimensional algebraic Riccati equation. 

\par Finally, given a finite sample $\mbb{S} = \{x_i\}_{i=1}^n \subset \mbb{X}$ that is assumed to be obtained in a way satisfying Assumption \ref{assum:sampling.strategy}, we impose the operator Riccati equation as a condition on pairs of points in $\mbb{S}$, i.e., for $(x,x') = (x_i, x_{i'})$, $i,i'=1,\cdots,n$, thus approximating it as a matrix Riccati equation. 
This is to be discussed in Subsection \ref{subsec:optimal.solution}.

\subsection{Error Analysis and Control Performance Bound}
Suppose that the operator Riccati equation \eqref{eq:optimal.Riccati} is satisfied by all pairs of points in $\mbb{S}\times \mbb{S}$. Its violation over all pairs in $\mbb{X}\times \mbb{X}$ should then be bounded according to the fill distance parameter $r_{\mbb{S}}$ as laid out in Assumption \ref{assum:sampling.strategy}. 
To see this, let us write
$$ V_{\lambda,\theta}(P_\pi) := P_\pi A_{0,\lambda,\theta} + A_{0,\lambda,\theta}^*P_\pi + \sum_{j=1}^{d_u} \Bra{ \Omega_j^\flat \times \Omega_j^\flat - (\Omega_j^\flat \times\Omega_j)^\flat - (\Omega_j^\flat \times\Omega_j)^{\flat *} } + Q, $$
which is a symmetric trace-class operator on $\rg H^s(\mbb{X})$. 

\begin{proposition}[Error bound on the pointwise violation of the operator Riccati equation]
    Suppose that Assumption \ref{assum:sampling.strategy} holds and \eqref{eq:optimal.Riccati} is satisfied by all $(x_i, x_{i'})\in \mbb{S}\times \mbb{S}$ ($i,i'=1,\cdots,n$). Then for all $x, x'\in \mbb{X}$, 
    \begin{equation} 
        \abs{ \ipp{\rg\phi(x)}{V_{\lambda,\theta}(P_\pi)}{\rg\phi(x')} } \lesssim r_{\mbb{S}}^{\min(1,\alpha)} \| V_{\lambda,\theta}(P_\pi) \|_{\mc{B}_1}  |x||x'|, 
    \end{equation}
    where $\alpha \in (0, s-d_x/2)$ is a constant. 
\end{proposition}
\begin{proof}
    The proof uses the same ideas as in Proposition \ref{prop:error-sample}. By spectral decomposition, we write $V_{\lambda,\theta}(P_\pi) = \sum_{k=1}^\infty \mu_k v_k\times v_k$, where $\{v_k\}_{k=1}^\infty$ is an orthonormal basis in $\rg H^s(\mbb{X})$. By a data-driven solution, we find a $P_\pi$ (and an associated $\pi$ in \eqref{eq:optimal.policy.Yosida}) such that $\sum_{k=1}^\infty \mu_k v_k(x_i)v_k(x_{i'}) = 0$ for all $i,i'=1,\cdots,n$. 
    For $x$ and similarly for $x'$, there must exist a point $x_i$ and respectively a point $x_{i'}$ in $\mbb{S}$ satisfying $|x-x_i|\leq r_\mbb{S}|x|$ and $|x'-x_{i'}| \leq r_{\mbb{S}}|x'|$. Hence $|v_k(x) - v_k(x_i)|\lesssim r_\mbb{S}|x| + (1+r_\mbb{S})r_\mbb{S}^\alpha |x|^{\alpha+1}$, and similarly $|v_k(x') - v_k(x_{i'})|\lesssim r_\mbb{S}|x'| + (1+r_\mbb{S})r_\mbb{S}^\alpha |x'|^{\alpha+1}$. 
    Hence, 
    $$\abs{ v_k(x)v_k(x') - v_k(x_i)v_k(x_{i'}) } \lesssim 
    |v_k(x')| \bra{r_\mbb{S}|x| + (1+r_\mbb{S})r_\mbb{S}^\alpha |x|^{\alpha+1}} + 
    |v_k(x_i)| \bra{r_\mbb{S}|x'| + (1+r_\mbb{S})r_\mbb{S}^\alpha |x'|^{\alpha+1}}. $$
    Relaxing $|v_k(x')|$ by $|x'|$ and $v_k(x_i)$ by $|x|$ is permissible, since $|x_i| \lesssim |x|$ and $v_k$ is a $\rg H^s(\mbb{X})$-class function with norm $1$. Therefore, $\abs{ v_k(x)v_k(x') - v_k(x_i)v_k(x_{i'}) } \lesssim  r_{\mbb{S}}^{\min(1,\alpha)} |x||x'|$. Summing up over $k=1,\cdots,\infty$ results in a constant coefficient of $\sum_{k=1}^\infty |\mu_k| = \| V_{\lambda,\theta}(P_\pi) \|_{\mc{B}_1}$. 
\end{proof}

The foregoing proposition determines a bound, proportional to the state norms, on the violation of the Yosida-approximated Riccati equation. This is hence a suboptimality guarantee on a virtual Yosida-approximated optimal control problem, in which the state evolves under the approximate semigroup $\{T_{\pi,\lambda,\theta}^t\}_{t\geq 0}$, generated by the approximate infinitesimal generator $A_{\pi,\lambda,\theta} = A_{0,\lambda,\theta} + \sum_{j=1}^{d_u} \rg\iota^*A_{j,\lambda, \theta} M_{\pi_j}^*$. 
In other words, suppose that, instead of considering the true optimal control problem \eqref{eq:optimal.0}, we consider a ``Yosida-approximated cost'':
\begin{equation}\label{eq:optimal.problem.Yosida}
    v_{\lambda,\theta}^\star(x) = \inf_{\pi \in [\rg H^s(\mbb{X})]^{d_u}} v_{\pi,\lambda,\theta}(x) = \inf_{\pi \in [\rg H^s(\mbb{X})]^{d_u}}  \int_0^\infty \bra{ \ipp{T_{\pi,\lambda,\theta}^t\rg\phi(x)}{Q}{T_{\pi,\lambda,\theta}^t\rg\phi(x)} + 
    \sum_{j=1}^{d_u} \abs{\ip{\pi_j}{T_{\pi,\lambda,\theta}^t\rg\phi(x)}}^2 } \,\xD{t}.  
\end{equation}
Then, following the derivation process in the previous subsection, we see that for the approximate optimal control problem \eqref{eq:optimal.problem.Yosida}, the optimal cost $P^\star$ verifies $V_{\lambda,\theta}(P^\star)=0$ and that the optimal policy is determined in a similar way to \eqref{eq:optimal.policy.Yosida}: $\pi^\star_j = -\ip{\rg\phi(x)}{P^\star \rg\iota^* A_{j,\lambda,\theta} \phi(x)}$, $j=1,\cdots,d_u$. 

\begin{proposition}[Performance bound on the Yosida-approximated cost]
    Suppose that Assumption \ref{assum:sampling.strategy} holds and \eqref{eq:optimal.Riccati} is satisfied by all $(x_i, x_{i'})\in \mbb{S}\times \mbb{S}$ ($i,i'=1,\cdots,n$). In addition, suppose that $\|T_\pi^t\|\lesssim \mr{e}^{\omega_\pi t}$ for some $\omega_\pi < 0$ and $s\geq2$. Then for all $x\in \mbb{X}$, the Yosida-approximated cost of the policy $\pi$, $v_{\pi,\lambda,\theta}(x)$, satisfies
    \begin{equation} 
        \abs{ v_{\pi,\lambda,\theta}(x) - \ipp{\rg\phi(x)}{P_\pi}{\rg\phi(x)} }
        \lesssim \bra{(\lambda-\omega_\pi)^{-1/3} + r_{\mbb{S}}^{\min(1,\alpha)}} \| V_{\lambda,\theta}(P_\pi) \|_{\mc{B}_1} |x|^2, 
    \end{equation} 
    and consequently,
    \begin{equation} 
        v_{\lambda,\theta}^\star(x) - \ipp{\rg\phi(x)}{P_\pi}{\rg\phi(x)}
        \lesssim \bra{(\lambda-\omega_\pi)^{-1/3} + r_{\mbb{S}}^{\min(1,\alpha)}} \| V_{\lambda,\theta}(P_\pi) \|_{\mc{B}_1} |x|^2. 
    \end{equation} 
\end{proposition}
\begin{proof}
    By the conclusion of the previous proposition, we have
    $$\abs{ \ipp{T_\pi^t \rg\phi(x)}{V_{\lambda,\theta}(P_\pi)}{ T_\pi^t  \rg\phi(x)} } \lesssim r_{\mbb{S}}^{\min(1,\alpha)} \| V_{\lambda,\theta}(P_\pi) \|_{\mc{B}_1} \mr{e}^{2\omega_{\pi} t} |x|^2, \ \forall x\in \mbb{X}, \ t\geq 0, $$
    since $T_\pi^t \rg\phi(x) = \rg\phi\bra{S_{f_\pi}^t (x)}$ holds and $\|T_\pi^t\|\lesssim \mr{e}^{\omega_\pi t}$. 
    
    \par In view of the expression of $V_{\lambda,\theta}(P_\pi)$, we would like to replace $T_\pi^t$ with $T_{\pi,\lambda,\theta}^t$, which, however, is complicated by the fact that $\rg\phi(x) \notin \mc{D}(A_\pi^2)$. Analogous to the proof procedure used in Proposition \ref{prop:error-pre-integral}, we use $\rg\phi(x)_{\pi,\epsilon,\epsilon}$ (the twice short-time average) as an intermediate to build the desired bound. It is thus analogously deduced that
    $$\abs{ \ipp{T_\pi^t \rg\phi(x)_{\pi,\epsilon,\epsilon}}{V_{\lambda,\theta}(P_\pi)}{T_\pi^t \rg\phi(x)_{\pi,\epsilon,\epsilon}} - \ipp{T_\pi^t \rg\phi(x)}{V_{\lambda,\theta}(P_\pi)}{T_\pi^t \rg\phi(x)}} 
    \lesssim \mr{e}^{2\omega_\pi t} \norm{{V_{\lambda,\theta}(P_\pi)}} \bra{2\epsilon + \epsilon^2}|x|^2, $$
    $$ \begin{aligned} 
    & \abs{ \ipp{T_{\pi,\lambda,\theta}^t \rg\phi(x)_{\pi,\epsilon,\epsilon}}{V_{\lambda,\theta}(P_\pi)}{T_{\pi,\lambda,\theta}^t \rg\phi(x)_{\pi,\epsilon,\epsilon}} - \ipp{T_{\pi,\lambda,\theta}^t \rg\phi(x)}{V_{\lambda,\theta}(P_\pi)}{T_{\pi,\lambda,\theta}^t \rg\phi(x)}} \\
    & \qquad \lesssim \mr{e}^{2\omega_{\pi,\lambda} t} \norm{V_{\lambda,\theta}(P_\pi)} \bra{2\epsilon + \epsilon^2}|x|^2, 
    \end{aligned}$$
    and, because $\norm{\rg\phi(x)_{\pi,\epsilon,\epsilon}}_{\mc{G}_2} \lesssim \epsilon^{-2} \norm{\rg\phi(x)}$, by Corollary \ref{cor:Yosida.DD}, 
    $$\begin{aligned}
        & \abs{ \ipp{T_{\pi,\lambda,\theta}^t \rg\phi(x)_{\pi,\epsilon,\epsilon}}{V_{\lambda,\theta}(P_\pi)}{T_{\pi,\lambda,\theta}^t \rg\phi(x)_{\pi,\epsilon,\epsilon}} - \ipp{T_{\pi}^t \rg\phi(x)_{\pi,\epsilon,\epsilon}}{V_{\lambda,\theta}(P_\pi)}{T_{\pi}^t \rg\phi(x)_{\pi,\epsilon,\epsilon}} } \\ 
        &\qquad \lesssim \frac{1}{\lambda-\omega_\pi} \frac{1}{\epsilon^2} t\mr{e}^{2\max(\omega_\pi, \omega_{\pi,\lambda})t}  \norm{{V_{\lambda,\theta}(P_\pi)}} |x|^2. 
    \end{aligned}$$
    Combining the above four bounds, we have
    $$\begin{aligned}
    & \abs{ \ipp{T_{\pi,\lambda,\theta}^t \rg\phi(x)}{V_{\lambda,\theta}(P_\pi)}{T_{\pi,\lambda,\theta}^t \rg\phi(x)} } 
    \lesssim \frac{1}{\lambda-\omega_\pi} \frac{1}{\epsilon^2} t\mr{e}^{2\max(\omega_\pi, \omega_{\pi,\lambda})t}  \norm{V_{\lambda,\theta}(P_\pi)} |x|^2  \\
    & \qquad + \epsilon \mr{e}^{2\max(\omega_\pi, \omega_{\pi,\lambda})t}  \norm{{V_{\lambda,\theta}(P_\pi)}} |x|^2 + r_{\mbb{S}}^{\min(1,\alpha)} \| V_{\lambda,\theta}(P_\pi) \|_{\mc{B}_1} \mr{e}^{2\omega_{\pi} t} |x|^2.
    \end{aligned} $$
    By optimizing the choice of the arbitrary constant $\epsilon$, we obtain
    $$\abs{ \ipp{T_{\pi,\lambda,\theta}^t \rg\phi(x)}{V_{\lambda,\theta}(P_\pi)}{T_{\pi,\lambda,\theta}^t \rg\phi(x)} } 
    \lesssim \frac{t^{1/3} \mr{e}^{2\max(\omega_\pi, \omega_{\pi,\lambda})t}}{(\lambda-\omega_\pi)^{1/3}}\norm{V_{\lambda,\theta}(P_\pi)} |x|^2 + r_{\mbb{S}}^{\min(1,\alpha)} \| V_{\lambda,\theta}(P_\pi) \|_{\mc{B}_1} \mr{e}^{2\omega_{\pi} t} |x|^2. $$

    \par We show that the bound above leads to a bound on the discrepancy between the kernel SOS form specified by $P_\pi$ and the Yosida-approximated cost of the policy $\pi$, $ v_{\pi,\lambda,\theta}(x)$. We note that 
    $$ V_{\lambda,\theta}(P_\pi) = A_{\pi,\lambda,\theta}^* P_\pi + P_\pi A_{\pi,\lambda,\theta} + Q + \sum_{j=1}^{d_u} \pi_j \times \pi_j . $$ 
    Therefore, along the approximate semigroup generated by the bounded operator $A_{\pi,\lambda,\theta}$,
    $$ \xDD{}{t} \ipp{T_{\pi,\lambda,\theta}^t \rg\phi(x)}{P_\pi}{T_{\pi,\lambda,\theta}^t \rg\phi(x)}
    = \ipp{T_{\pi,\lambda,\theta}^t \rg\phi(x)}{V_{\lambda,\theta}(P_\pi) - Q - \sum_{j=1}^{d_u} \pi_j \times \pi_j}{T_{\pi,\lambda,\theta}^t \rg\phi(x)}. $$
    Integrating the above equality over $t\in [0,\infty)$ yields
    $$ v_{\pi,\lambda,\theta}(x) - \ipp{\rg\phi(x)}{P_\pi}{\rg\phi(x)} = \int_0^\infty \ipp{T_{\pi,\lambda,\theta}^t\rg\phi(x)}{V_{\lambda,\theta}(P_\pi)}{T_{\pi,\lambda,\theta}^t\rg\phi(x)} \xD{t}. $$
    Therefore,
    $$ \begin{aligned}
        \abs{ v_{\pi,\lambda,\theta}(x) - \ipp{\rg\phi(x)}{P_\pi}{\rg\phi(x)} }
        & \leq \int_0^\infty \abs{ \ipp{T_{\pi,\lambda,\theta}^t \rg\phi(x)}{V_{\lambda,\theta}(P_\pi)}{T_{\pi,\lambda,\theta}^t \rg\phi(x)} } \xD{t} \\
        &\lesssim \int_0^\infty \Bra{ \frac{t^{1/3} \mr{e}^{2\max(\omega_\pi, \omega_{\pi,\lambda})t}}{(\lambda-\omega_\pi)^{1/3}}\norm{V_{\lambda,\theta}(P_\pi)} |x|^2 + r_{\mbb{S}}^{\min(1,\alpha)} \| V_{\lambda,\theta}(P_\pi) \|_{\mc{B}_1} \mr{e}^{2\omega_{\pi} t} |x|^2 }\xD{t} \\
        & \lesssim \bra{(\lambda-\omega_\pi)^{-1/3} + r_{\mbb{S}}^{\min(1,\alpha)}} \| V_{\lambda,\theta}(P_\pi) \|_{\mc{B}_1} |x|^2.
    \end{aligned}$$
    This proves the first inequality. The second inequality follows, since $\pi \in [\rg H^s(\mbb{X})]^{d_u}$ is admissible in \eqref{eq:optimal.problem.Yosida}, and hence $v_{\lambda,\theta}^\star(x) \leq v_{\pi,\lambda,\theta}(x)$.
\end{proof}

Subsequently, we should compare the cost evaluated under the actual flow, $v_{\pi^\star}(x)$, to the Yosida-approximated cost $v_{\lambda,\theta}^\star(x)$, under the same policy $\pi^\star$. 
\begin{proposition}[Discrepancy between the Yosida-approximated cost and true cost]
    Suppose that $\|T_{\pi^\star}^t\|\lesssim \mr{e}^{\omega_{\pi^\star} t}$ for some $\omega_{\pi^\star}<0$ and $s\geq 2$. Then for all $x\in \mbb{X}$, we have
    \begin{equation}
        \abs{v_{\pi^\star}(x) -v_{\lambda,\theta}^\star(x)} \lesssim 
    \frac{1}{(\lambda-\omega_{\pi^\star})^{1/3}} \bra{\norm{Q} + \sum_{j=1}^{d_u} \norm{\pi^\star_j}^2} |x|^2 . 
    \end{equation}
\end{proposition}
\begin{proof}
    Because $v_{\lambda,\theta}^\star$ is obtained from policy evaluation under $\pi^\star$ for the Yosida-approximated system, 
    $$v_{\lambda,\theta}^\star(x) = \int_0^\infty \ipp{T_{\pi^\star,\lambda,\theta}^t \rg\phi(x)}{Q + \sum_{j=1}^{d_u} \pi^\star_j\times \pi^\star_j}{T_{\pi^\star,\lambda,\theta}^t \rg\phi(x)} \xD{t}, $$
    while the policy evaluation for the true system gives
    $$v_{\pi^\star}(x) = \int_0^\infty \ipp{T_{\pi^\star}^t \rg\phi(x)}{Q + \sum_{j=1}^{d_u} \pi^\star_j\times \pi^\star_j}{T_{\pi^\star}^t  \rg\phi(x)} \xD{t}, $$
    the bound to be proved is for the error resulting from replacing $T_{\pi^\star,\lambda,\theta}^t$ with $T_{\pi^\star}^t$. 
    Again using the twice short-time average $\rg\phi(x)_{\pi^\star, \epsilon, \epsilon} \in \mc{D}(A_{\pi^\star}^2)$, in an analogous way as in Proposition \ref{prop:error-integral}, we obtain 
    $$\begin{aligned}
        &\abs{ \ipp{T_{\pi^\star,\lambda,\theta}^t \rg\phi(x)}{Q + \sum_{j=1}^{d_u} \pi^\star_j\times \pi^\star_j}{T_{\pi^\star,\lambda,\theta}^t \rg\phi(x)} - \ipp{T_{\pi^\star}^t \rg\phi(x)}{Q + \sum_{j=1}^{d_u} \pi^\star_j\times \pi^\star_j}{T_{\pi^\star}^t  \rg\phi(x)} } \\
        &\quad \lesssim \epsilon \norm{Q + \sum_{j=1}^{d_u} \pi^\star_j\times \pi^\star_j} \mr{e}^{2 \max(\omega_{\pi^\star}, \omega_{\pi^\star, \lambda}) t} |x|^2 + \frac{1}{\epsilon^2} \frac{1}{\lambda-\omega_{\pi^\star}} \norm{Q + \sum_{j=1}^{d_u} \pi^\star_j\times \pi^\star_j} t \mr{e}^{2 \max(\omega_{\pi^\star}, \omega_{\pi^\star, \lambda}) t} |x|^2.
    \end{aligned} $$
    By optimizing the choice of $\epsilon$ and integrating over $t\in [0, \infty)$, we arrive at 
    $$ \abs{v_{\pi^\star}(x) -v_{\lambda,\theta}^\star(x)} \lesssim \int_0^\infty \frac{t^{1/3}}{(\lambda-\omega_{\pi^\star})^{1/3}} \mr{e}^{2 \max(\omega_{\pi^\star}, \omega_{\pi^\star, \lambda}) t}  \bra{\norm{Q} + \sum_{j=1}^{d_u} \norm{\pi^\star_j}^2} |x|^2 \xD{t}, $$
    which gives the proposition to be proved. 
\end{proof}

The following conclusion thus follows from the combination of previous propositions in this subsection, which we now state as a theorem. The theorem provides an upper bound on the lowest possible cost, which deviates from the kernel SOS form specified by $P_\pi$ by a constant multiple of $|x|^2$. This constant coefficient is dependent on the density of sample points (captured by the constant $r_\mbb{S}$) and the Yosida regularization constant $\lambda$. 
\begin{theorem}[Near optimality of the performance of Koopman--Riccati controller]
    Suppose Assumption \ref{assum:sampling.strategy} holds. Additionally, suppose that $\|T_\pi^t\|\lesssim \mr{e}^{\omega_\pi t}$ for some $\omega_\pi<0$, $\|T_{\pi^\star}^t\|\lesssim \mr{e}^{\omega_{\pi^\star} t}$ for some $\omega_{\pi^\star}<0$, and $s\geq 2$. Denote $\omega = \max(\omega_\pi, \omega_{\pi^\star})<0$ and $\omega_\lambda = \frac{\lambda\omega}{\lambda - \omega} < 0$. Then 
    \begin{equation}
        v(x) - \ipp{\rg\phi(x)}{P_\pi}{\rg\phi(x)} \lesssim \Bra{
        \bra{(\lambda-\omega_\lambda)^{-1/3} + r_{\mbb{S}}^{\min(1,\alpha)}} \|V_{\lambda,\theta}(P_\pi) \|_{\mc{B}_1} + (\lambda-\omega_\lambda)^{-1/3}  \bra{\norm{Q} + \sum_{j=1}^{d_u} \norm{\pi^\star_j}^2} } |x|^2. 
    \end{equation}
\end{theorem}

\subsection{Data-Driven Solution and Kleinman Iterations}\label{subsec:optimal.solution}
The final issue is an explicit expression of the sample-based version of the operator Riccati equation \eqref{eq:optimal.Riccati} as a matrix Riccati equation. Instead of aiming to find the representation of the $\flat$ symbol on operators, $\flat$ symbol on tensors, and the $\times$ symbol for tensor product, we directly impose the condition \eqref{eq:optimal.quartic} on each $(x_i, x_{i'})$-pair for $i,i'=1,\cdots,n$: 
\begin{equation}\label{eq:optimal.quartic.sample}
    \begin{aligned}
        & \ipp{\rg\phi(x_i)}{P_\pi A_{0,\lambda,\theta} + A_{0,\lambda,\theta}^*P_\pi + Q}{\rg\phi(x_{i'})} \\
        &\quad - \sum_{j=1}^{d_u} \ip{\rg\phi(x_{i'})}{P_\pi \rg\iota^* A_{j,\lambda,\theta}\phi(x_i)} \ip{\rg\phi(x_i)}{P_\pi \rg\iota^* A_{j,\lambda,\theta}\phi(x_i)} \\
        & \quad - \sum_{j=1}^{d_u} \ip{\rg\phi(x_i)}{P_\pi \rg\iota^* A_{j,\lambda,\theta}\phi(x_{i'})} \ip{\rg\phi(x_{i'})}{P_\pi \rg\iota^* A_{j,\lambda,\theta}\phi(x_{i'})} \\
        & \quad + \sum_{j=1}^{d_u} \ip{\rg\phi(x_i)}{P_\pi \rg\iota^* A_{j,\lambda,\theta}\phi(x_i)} \ip{\rg\phi(x_{i'})}{P_\pi \rg\iota^* A_{j,\lambda,\theta}\phi(x_{i'})} = 0, \ \forall i,i'=1,\cdots,n. 
    \end{aligned}
\end{equation}
We tacitly assume that $\ipp{\rg\phi(x_i)}{Q}{\rg\phi(x_{i'})}$ can be explicitly evaluated for any $i,i'=1,\cdots,n$. This is the case whenever $Q$ is given as an explicit finite sum of squares of $\rg H^s(\mbb{X})$-functions. That is, if $q(x) = \sum_{k=1}^{n_q} q_k(x)^2$ with $q_1, \cdots, q_{n_q} \in \rg H^s(\mbb{X})$, then $Q = \sum_{k=1}^{n_q} q_k\times q_k$ and hence $\ipp{\rg\phi(x_i)}{Q}{\rg\phi(x_{i'})} = \sum_{k=1}^{n_q} q_k(x_i)q_k(x_{i'})$ can be computed. 
Hence, let us denote $\ms{Q}$ as the matrix with its $(i,i')$-th entry equal to $\ipp{\rg\phi(x_i)}{Q}{\rg\phi(x_{i'})}$. For example, for $q(x) = |x|^2$, $\ms{Q}_{ii'} = x_i \cdot x_{i'}$ ($i,i'=1,\cdots,n$). 

Following the notations of $\ms{A}^0$ and $\ms{A}^j$ ($j=1,\cdots,d_u$) as given in \eqref{eq:stabilizing.synthesis.kernel.1} and \eqref{eq:stabilizing.synthesis.kernel.2}, we have $$\ip{\rg\phi(x_i)}{P_\pi \rg\iota^* A_{j,\lambda,\theta}\phi(x_{i'})} = (\ms{G}\ms{P}\ms{A}^j)_{ii'}$$ and hence by setting $P_\pi = \sum_{i=1}^n \sum_{i'=1}^n \ms{P}_{ii'} \rg\phi(x_i)\times \rg\phi(x_{i'})$, the condition \eqref{eq:optimal.quartic.sample} becomes equivalent to 
\begin{equation}\label{eq:optimal.sample.Riccati}
\begin{aligned} 
    &\ms{G}\ms{P}\ms{A}^0 + \ms{A}^{0\top}\ms{P}\ms{G} + \ms{Q} \\
    &+ \sum_{j=1}^{d_u} \Bra{ -\bra{\Diag \ms{G}\ms{P} \ms{A}^j} \bra{\ms{A}^{j\top} \ms{P}\ms{G}} - \bra{\ms{G}\ms{P}\ms{A}^j}\bra{\Diag \ms{A}^{j\top} \ms{P}\ms{G}} + \bra{\diag \ms{G}\ms{P} \ms{A}^j} \bra{\diag \ms{A}^{j\top} \ms{P}\ms{G}}^\top }= 0. 
\end{aligned}
\end{equation}
The meaning of $\Diag$ was the same as in the previous section, and the property of $\Diag \ms{M} = \Diag \ms{M}^\top$ is self-evident.

\par The solution procedure of \eqref{eq:optimal.sample.Riccati} as a quadratic equation of $\ms{P}$ is therefore analogous to the usual Riccati equation --- a \emph{Kleinman iteration algorithm} can be implemented. 
Conceptually, the Kleinman iteration algorithm consists of repeated alternations between policy evaluation and policy improvement. In the policy evaluation step, a control policy is fixed, and the resulting $\ms{P}$ is evaluated; in the policy improvement step, the matrix $\ms{P}$ is fixed, and the optimal policy is obtained. 
In the present setting, the optimal policy formula \eqref{eq:optimal.policy.Yosida} says 
$$\pi_j(x_i) = -\ip{\rg\phi(x_i)}{P_\pi \rg\iota^* A_{j,\lambda,\theta} \phi(x_i)} = -(\ms{G}\ms{P}\ms{A}^j)_{ii}.$$
Thus, the diagonal matrices in \eqref{eq:optimal.sample.Riccati} carry the policy information. Hence, we rewrite the equation as 
\begin{subequations}\label{eq:optimal.sample.Kleinman}
\begin{align}
    & \ms{G}\ms{P}\ms{A}^0 + \ms{A}^{0\top}\ms{P}\ms{G} + \ms{Q} + \sum_{j=1}^{d_u} \Bra{ \bra{\Diag \ms{u}_j} \bra{\ms{A}^{j\top} \ms{P}\ms{G}} + \bra{\ms{G}\ms{P}\ms{A}^j}\bra{\Diag \ms{u}_j} + \ms{u}_j\ms{u}_j^\top }= 0, \label{eq:optimal.sample.Kleinman.1} \\ 
    & \ms{u}_j = -\diag \ms{G}\ms{P} \ms{A}^j, \ j=1,\cdots,d_u . \label{eq:optimal.sample.Kleinman.2} 
\end{align}
\end{subequations} 
The Kleinman iteration algorithm is therefore the alternation of solving $\ms{P}$ with fixed $(\ms{u}_1, \cdots, \ms{u}_{d_u})$ from \eqref{eq:optimal.sample.Kleinman.1} and evaluating $(\ms{u}_1, \cdots, \ms{u}_{d_u})$ with given $\ms{P}$ from \eqref{eq:optimal.sample.Kleinman.2}. 
As discussed in Remark \ref{rem:policy.recovery}, the numerical routine only determines the policy values at the sample points. To evaluate the policy at any state $x$, we perform a $\rg C^s$ interpolation, e.g., by a linear--Gaussian kernel or a linear--Sobolev kernel with a higher smoothness index. 
At this point, the obstacle to \emph{Koopman-based optimal control} is cleared.

\section{Numerical Studies}
In this section, we aim to demonstrate that the proposed approach is a computationally efficient one that gives well-performing feedback control laws, with a particular focus on the Koopman--Riccati synthesis method that provides a near-optimal controller by direct synthesis.

The Python codes for all the examples in this section are publicly available in the following GitHub repository: \url{https://github.com/WentaoTang-Pack/KoopmanRiccatiController}. 
All computations were performed on a laptop with an Apple M4 Max processor (14 CPU cores, of which 10 are performance cores) and 36 GB of memory, running macOS 15.7, with Python 3.13, NumPy 2.2, and SciPy 1.15. Semidefinite programs are solved by MOSEK 11.2.4 called through CVXPY 1.8.1, and nonlinear programs (in nonlinear MPC and Koopman MPC) are solved by IPOPT 3.14.11 with the MUMPS linear solver, called through CasADi 3.7.2. The reported computational times are wall-clock times.

Two remarks are given below regarding some implementation details.
\begin{remark}
    In the experimentation, it is empirically found that the Koopman--Lyapunov and Koopman--Riccati problems can usually be solved successfully for open-loop stable systems, but not always for open-loop unstable ones. With open-loop stability, the problem \eqref{eq:optimal.sample.Kleinman} can be conveniently initialized at a zero policy. 
    More importantly, for open-loop stable systems, the Yosida approximations of the unbounded infinitesimal operators $A_0$, $A_j$ ($j=1,\cdots,d_u$) tend not to give a spurious spectrum that would cause the data-driven synthesis problems, formulated with a kernel method, to become infeasible. 
    As pointed out in Remark \ref{rem:spectrum}, it is generally impossible to guarantee that the spectrum after Yosida approximation remains close to the original, and hence unstable modes would result in unnecessary difficulties. 
    \par Hence, whenever the formulations of our Koopman--Lyapunov and Koopman--Riccati syntheses are used on open-loop unstable systems, it is beneficial to have a stabilizing control law $\pi_\mr{prior}$ known \textit{a priori}, e.g., a na{\"i}ve linearized LQR controller. Then by a typical trick $u = \pi_\mr{prior}(x) + u'$, we can consider $u'$ as the new input and consider the optimal control problem as one with cost term $q(x) + |\pi_\mr{prior}(x)+u'|^2$. The cost is still quadratically dependent on $u$, and hence the analysis in the previous section holds with minor changes. 
    To the author's knowledge, imposing such an auxiliary prior controller is a standard practice in controlling open-loop systems. 
\end{remark}
\begin{remark}
    The Koopman--Lyapunov control synthesis method strongly depends on the numerical implementation of the last step, where the policy values at the sample points are recovered from the non-diagonal matrices $\ms{F}^j$ ($j=1,\cdots,d_u$). In the case where the kernel matrix $\ms{G}$ is near-singular (with a large condition number), directly solving \eqref{eq:stabilizing.recovery.2} or finding a solution via \eqref{eq:stabilizing.bilinear} may not be easy, as the obtained solution for $\ms{P}$ may poorly satisfy the actuatability assumption. In addition, the spectrum conditions that are theoretically needed by the error analysis may not be well satisfied \textit{a posteriori}. 
    Therefore, the following remedy, motivated by the Kleinman iteration,
    \begin{equation*}
        \mathsf u_j = -\operatorname{diag}\big(\mathsf{GPA}^j\big), \qquad j = 1, \dots, d_u,
    \end{equation*}
    is adopted to ``guide'' the search of $\ms{u}_j$. 
    At the sample points where \eqref{eq:stabilizing.bilinear} would be violated due to this remedy, we project $(\mathsf u_{1i}, \dots, \mathsf u_{d_u i})$ onto the half-space defined by \eqref{eq:stabilizing.bilinear}.
\end{remark}

\subsection{A Simple Example with Closed-Form Optimal Controller}
We first consider the system of a special form \cite{nevistic1996constrained}:
$$\begin{bmatrix} \dot x_1 \\ \dot x_2 \end{bmatrix} = \begin{bmatrix}
    -x_1 + x_2 \\ -\frac{1}{2} x_1 - \frac{1}{2} x_2\big(1-\gamma(x)^2\big) 
\end{bmatrix} + \begin{bmatrix}
    0 \\ \gamma(x)
\end{bmatrix}u. $$
Let us consider the case $\gamma(x) = 2+ \cos 2\pi x_1$. The state space is assigned to be $\mbb{X} = [-1, 1]\times [-1, 1]$. 
Setting $q(x) = |x|^2 = x_1^2 + x_2^2$, the true solution to the optimal control problem \eqref{eq:optimal.0} is determined by the Hamilton--Jacobi--Bellman equation $\nabla v(x)^\top f_0(x) - \frac{1}{4} \big(f_1(x)^\top\nabla v(x)\big)^2 + q(x)=0$, which always has the classical solution $v(x) = \frac{1}{2}x_1^2 + x_2^2$ and $\pi(x) = -\frac{1}{2}f_1(x)^\top \nabla v(x) = -x_2 \gamma(x)$. 
In this case, the true optimal policy and optimal cost are
$$\pi_\mr{true}(x) = -x_2\bra{2+ \cos 2\pi x_1}, \ v_\mr{true}(x) = \frac{1}{2}x_1^2 + x_2^2.$$ 
The linearized system $\dot{x}_1 = -x_1 + x_2$, $\dot{x}_2 = -\frac{1}{2}x_1 + 4x_2 + 3u$ would result in a controller $\pi_{\mr{LQR}}(x) = -3x_2$, which is expected to differ discernibly from the true optimal policy. 

\paragraph{Hyperparameter tuning.} For the controller synthesis procedures below, the Sobolev smoothness index is chosen as $s=4.5$ and the Yosida constants are chosen to be $\lambda=200$ and $\theta=1/8$. These hyperparameters are selected to ensure that the Yosida approximation has good quality, which matters to the numerical implementation, especially for initializing a stabilizing controller in the Kleinman iteration algorithm. 
The sample size is $n=50$ for the Koopman--Lyapunov synthesis, whose SDP has $O(n^2)$ decision variables, and $n=100$ for the Koopman--Riccati synthesis.
In the implementation, owing to the ill-conditioning of $\ms{G}$, we add a regularization $10^{-5}\ms{I}$ to the kernel matrix $\ms{G}$ whenever it is inverted. 

\paragraph{Stabilizing controller synthesis.} The Koopman--Lyapunov synthesis problem was solved with hyperparameters $\alpha_\ms{X} = 10^{-4}$, $\beta_\ms{X} = 300$, and $\gamma_\ms{X} = 1$ in \eqref{eq:stabilizing.SDP}. The problem is solved in $3.82$ seconds by MOSEK 11.2.4 called through CVXPY 1.8.1.
Compared to the true optimal controller, the cost degrades by $3.42\%$ on average and by $39.11\%$ in the worst case. But since we are not looking for a near-optimal but only a stabilizing controller here, this degradation is not essential. 
\begin{figure}[!t]
    \centering
    \begin{subfigure}[b]{0.65\textwidth}
         \centering
         \includegraphics[width=\textwidth]{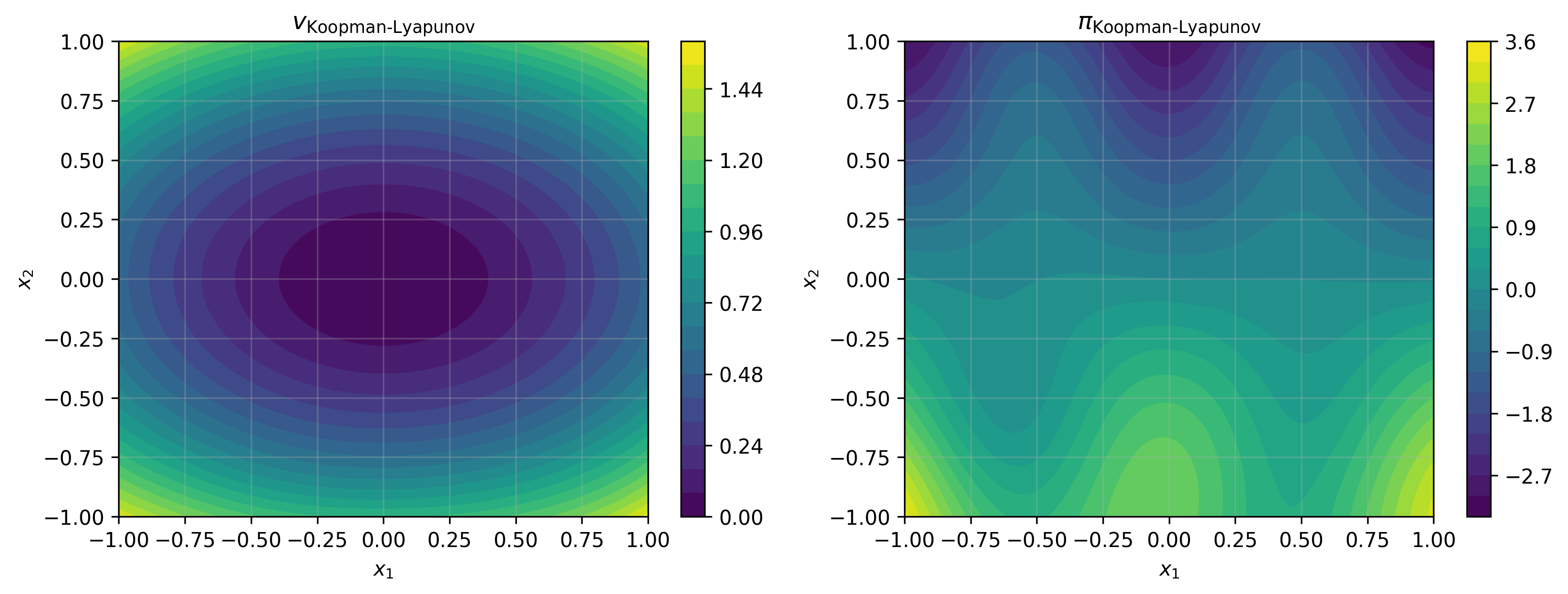}
         \caption{Profiles of cost and policy}
         \label{fig:ex1fig4}
     \end{subfigure} 
     \begin{subfigure}[b]{0.3\textwidth}
         \centering
         \includegraphics[width=\textwidth]{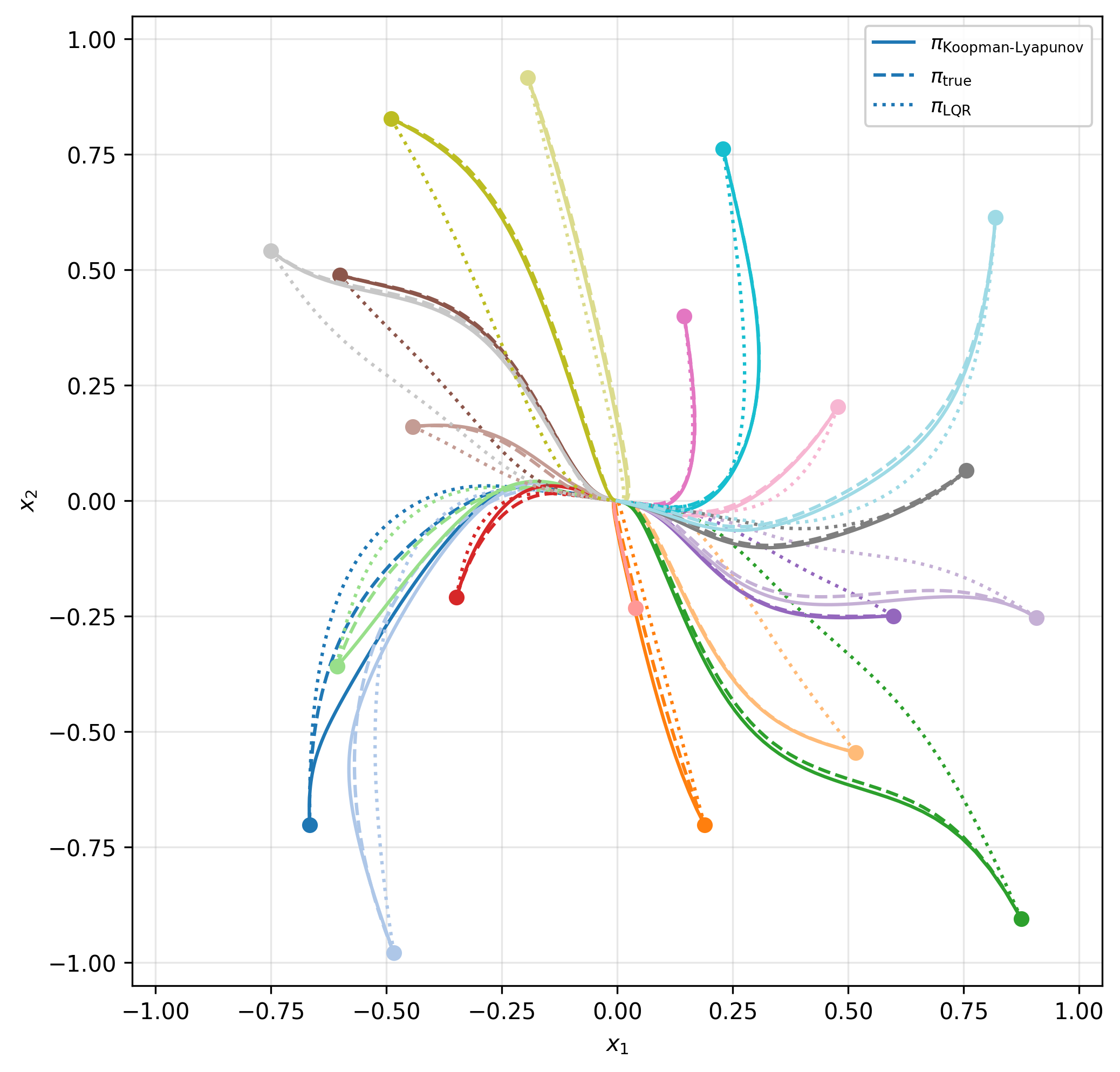}
         \caption{Comparison of closed-loop trajectories}
         \label{fig:ex1fig5}
     \end{subfigure}
    \caption{Koopman--Lyapunov synthesis of a stabilizing controller for Example 1.}
\end{figure}

\paragraph{Optimal controller synthesis.} 
We continue to compare the Koopman--Riccati controller $\pi_\mr{Koopman\text{-}Riccati}$ to $\pi_\mr{true}$ and $\pi_\mr{LQR}$ as two baselines. The convergence of the Kleinman iteration algorithm is shown in Figure \ref{fig:ex1fig1}, with a total computation time of $0.680$ seconds. Iterations are terminated when the newly updated vector $\ms{u}$ does not differ from the previous iterate in $\ell^\infty$-norm above a tolerance of $10^{-6}$. The number of iterations needed for this example is $121$, with an average computational time per iteration of $5.6$ ms.
\begin{figure}[!t]
    \centering
    \begin{subfigure}[b]{0.5\textwidth}
         \centering
         \includegraphics[width=\textwidth]{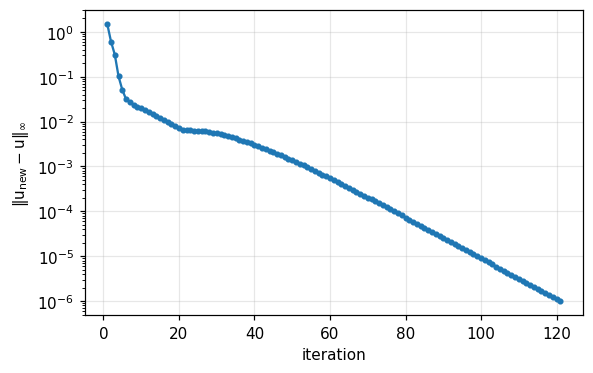}
         \caption{Convergence of Kleinman iterations}
         \label{fig:ex1fig1}
     \end{subfigure} 
     \begin{subfigure}[b]{0.3\textwidth}
         \centering
         \includegraphics[width=\textwidth]{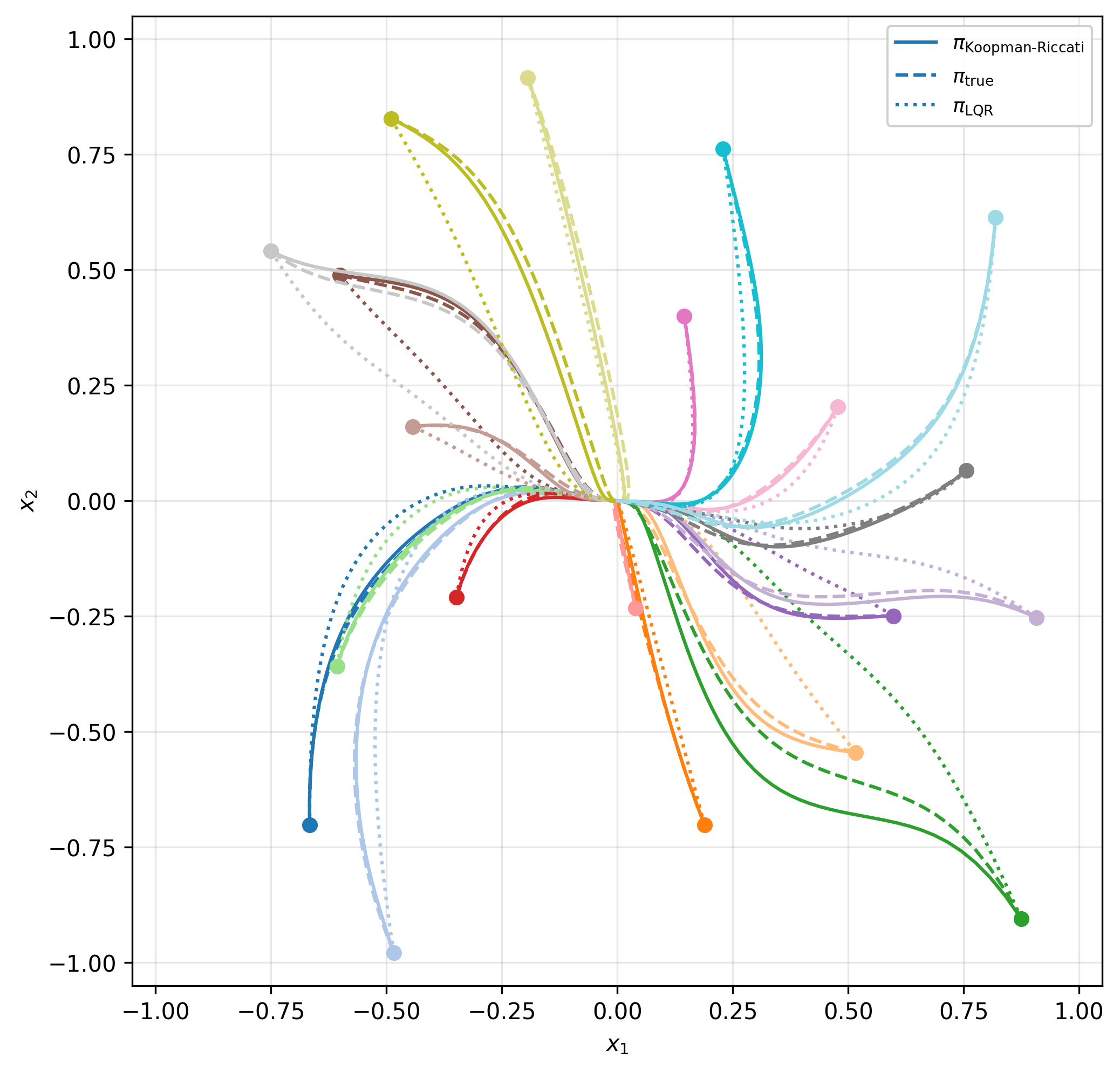}
         \caption{Comparison of closed-loop trajectories}
         \label{fig:ex1fig3}
     \end{subfigure}
    \caption{Koopman--Riccati synthesis of a near-optimal controller for Example 1.}
\end{figure}

Closed-loop trajectories from $20$ initial conditions, drawn by stratified random sampling (one uniformly random point in each cell of a $5\times 4$ partition of $\mbb{X}$), are shown in Figure \ref{fig:ex1fig3}. The trajectories under the Koopman--Riccati controller are close to those under the true optimal policy and differ markedly from those under the linearized LQR policy. 
Also, Table \ref{tab:ex1_costs} lists the accumulated costs under the three policies from these initial conditions, from which we see that $\pi_\mr{LQR}$ causes a performance degradation of $15.75\%$ on average and up to $51.15\%$ in the worst case, while $\pi_\mr{Koopman\text{-}Riccati}$ has only a loss of $0.47\%$ on average and $1.50\%$ in the worst case.
The comparison between the cost function obtained from the Koopman--Riccati method (as a kernel SOS form) and the true cost $v_\mr{true}$ is shown in Figure \ref{fig:ex1fig2}, from which it can be clearly seen that the proposed method synthesizes a policy that approximates $\pi_\mr{true}$ well and captures the periodicity of the $\gamma(\cdot)$ term in the governing equations, while a na{\"{i}}ve LQR loses this important feature. 
\begin{figure}[!t]
    \centering
    \includegraphics[width=\textwidth]{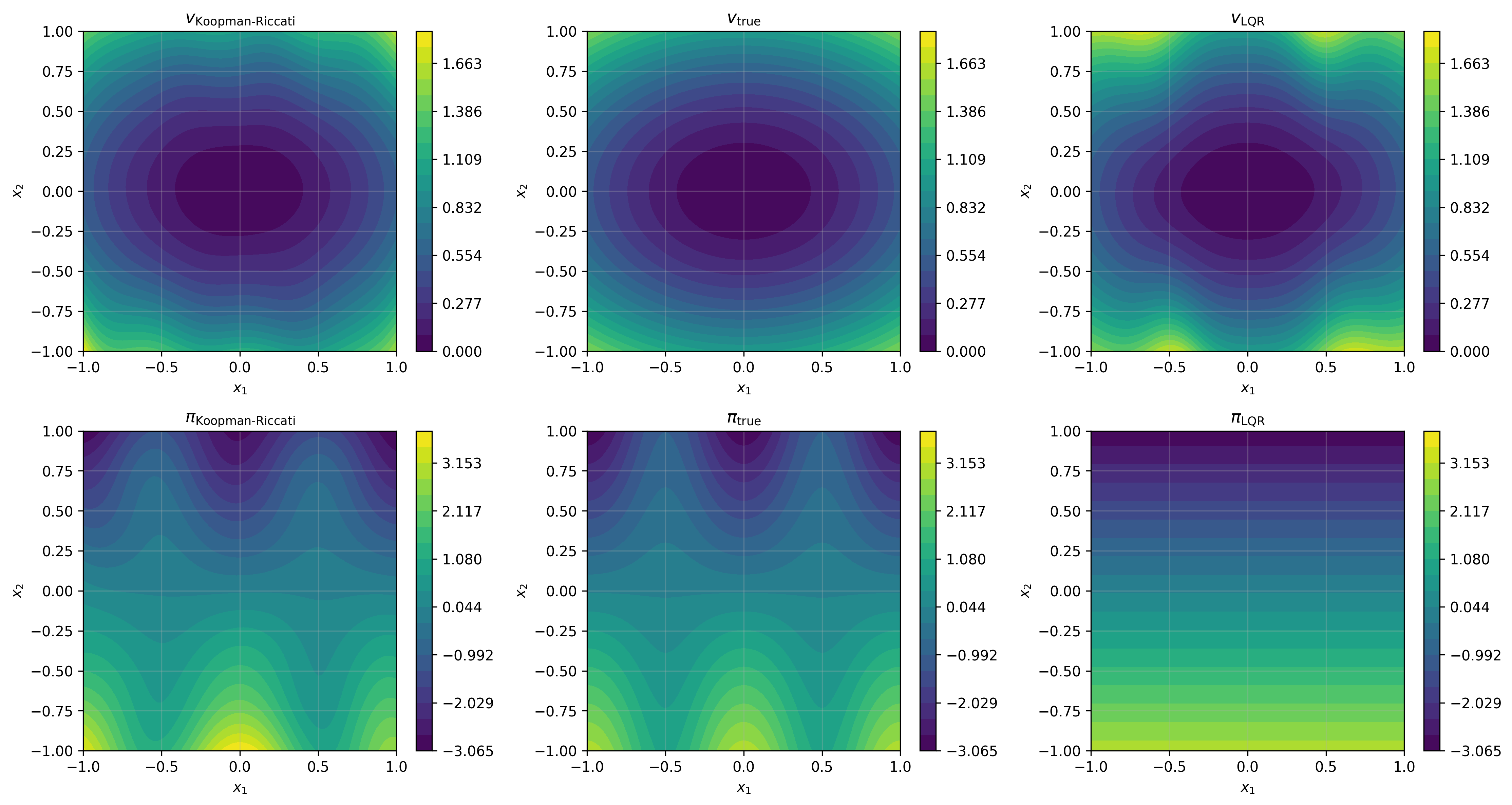}
    \caption{Comparison of cost functions and policies for Example 1.}
    \label{fig:ex1fig2}
\end{figure}
\begin{table}[htbp]
  \centering
  \caption{Costs under different policies for Example 1.}
  \label{tab:ex1_costs}
  \begin{tabular}{cr|rr|rr}
    \toprule
    $x_0$ & $v_\mr{true}(x_0)$ & Koopman--Riccati & error (\%) & linearized LQR & error (\%) \\
    \midrule
    $(-0.666, -0.702)$ & 0.7146 & 0.7162 & 0.22 & 0.8526 & 19.31 \\
    $(-0.484, -0.979)$ & 1.0756 & 1.0805 & 0.46 & 1.6257 & 51.15 \\
    $(0.189, -0.702)$ & 0.5107 & 0.5118 & 0.23 & 0.5195 & 1.72 \\
    $(0.516, -0.545)$ & 0.4302 & 0.4340 & 0.89 & 0.5826 & 35.45 \\
    $(0.875, -0.905)$ & 1.2018 & 1.2199 & 1.50 & 1.4144 & 17.68 \\
    $(-0.607, -0.358)$ & 0.3124 & 0.3135 & 0.36 & 0.3621 & 15.91 \\
    $(-0.348, -0.209)$ & 0.1042 & 0.1045 & 0.29 & 0.1138 & 9.16 \\
    $(0.040, -0.232)$ & 0.0546 & 0.0551 & 0.88 & 0.0548 & 0.33 \\
    $(0.598, -0.249)$ & 0.2408 & 0.2420 & 0.51 & 0.2914 & 21.03 \\
    $(0.908, -0.253)$ & 0.4762 & 0.4779 & 0.34 & 0.5058 & 6.20 \\
    $(-0.601, 0.489)$ & 0.4197 & 0.4212 & 0.36 & 0.5700 & 35.82 \\
    $(-0.443, 0.161)$ & 0.1240 & 0.1243 & 0.24 & 0.1389 & 11.96 \\
    $(0.145, 0.400)$ & 0.1705 & 0.1710 & 0.30 & 0.1740 & 2.07 \\
    $(0.477, 0.204)$ & 0.1554 & 0.1560 & 0.38 & 0.1717 & 10.48 \\
    $(0.756, 0.066)$ & 0.2901 & 0.2909 & 0.27 & 0.2967 & 2.25 \\
    $(-0.750, 0.541)$ & 0.5739 & 0.5781 & 0.72 & 0.7150 & 24.57 \\
    $(-0.490, 0.828)$ & 0.8056 & 0.8100 & 0.54 & 1.0778 & 33.78 \\
    $(-0.194, 0.917)$ & 0.8597 & 0.8626 & 0.33 & 0.8740 & 1.66 \\
    $(0.229, 0.762)$ & 0.6069 & 0.6088 & 0.32 & 0.6745 & 11.15 \\
    $(0.819, 0.614)$ & 0.7124 & 0.7140 & 0.22 & 0.7366 & 3.40 \\ 
    \bottomrule
  \end{tabular}
\end{table}

\subsection{A Chemical Reactor Example with Comparison to EDMD-based Controllers}
Next we consider a non-isothermal CSTR with an irreversible elementary reaction $\mr{A} \xrightarrow{k} \mr{B}$. The states are the reactant concentration and temperature, $(c_\mr{A}, \, T)$, and the input is the coolant temperature $T_\mr{c}$. The model is derived from mass and energy balances, with all parameter values listed in Table~\ref{tab:ex2parameters}:
\begin{equation*}
\begin{aligned}
V \frac{\mathrm{d}c_\mr{A}}{\mathrm{d}t} &= q\,(c_\mr{Ai} - c_\mr{A}) - Vk_0 \mr{e}^{-\frac{E}{RT}} c_\mr{A}, \\
V \rho C_p \frac{\mathrm{d}T}{\mathrm{d}t} &= q \rho C_p\,(T_\mr{i} - T) + (-\Delta H_\mr{R})\, Vk_0 \mr{e}^{-\frac{E}{RT}} c_\mr{A} + UA\,(T_\mr{c} - T).
\end{aligned}
\end{equation*}
The equilibrium point of interest is $(c_\mr{A}^\circ, T^\circ) = (0.7539\text{ mol/L}, 335\text{ K})$ under $T_\mr{c}^\circ = 322.91$ K. Hence, by defining $u = (T_\mr{c} - T_\mr{c}^\circ)/5\text{ K}$ and $x = ((c_\mr{A} - c_\mr{A}^\circ)/0.5\text{ mol/L}, (T-T^\circ)/10\text{ K})$, we have a continuous-time input-affine system \eqref{eq:system}. 
The state space is $\mbb{X} = [-0.4, 0.4]\times[-1, 1]$ after scaling (i.e., $c_\mr{A} \in c_\mr{A}^\circ \pm 0.2$ mol/L and $T\in [325, 345]$ K), and $n=100$ points are sampled uniformly from $\mbb{X}$. 

\begin{table}[!t]
  \centering
  \caption{Model parameters of Example 2.}\label{tab:ex2parameters}
  \begin{tabular}{llr}
    \toprule
    Parameter & Description & Value \\
    \midrule 
    $q$ & volumetric flow rate & $100$ L/min \\
    $V$ & reactor volume & $100$ L \\
    $c_\mr{Ai}$ & feed concentration of A & $1$ mol/L \\
    $T_\mr{i}$ & feed temperature & $350$ K \\
    $\rho$ & density & $1000$ g/L \\
    $C_p$ & heat capacity & $0.239$ J/(g$\cdot$K) \\
    $-\Delta H_\mr{R}$ & heat of reaction & $5 \times 10^{4}$ J/mol \\
    $E/R$ & activation temperature & $8750$ K \\
    $k_0$ & pre-exponential factor & $7.2 \times 10^{10}$ min$^{-1}$ \\
    $UA$ & heat-transfer coefficient $\times$ area & $1.3145 \times 10^{5}$ J/(min$\cdot$K) \\
    \bottomrule
  \end{tabular}
\end{table}

\paragraph{Comparison of Nonlinear MPC and Koopman--Riccati synthesis.} Consider the optimal control problem with $q(x) = |x|^2$, which no longer has a closed-form solution. Therefore, a nonlinear MPC (NMPC) based on the true model (assuming it is available) with a horizon of $4$ minutes ($400$ steps with a sampling time of $0.01$ minute) can be used as an implicit feedback law with approximately optimal performance. 
Comparing the cost and policy profiles of the nonlinear MPC and the proposed Koopman--Riccati synthesis (in Figure \ref{fig:ex2fig2}), the error is found to be very small, which is further confirmed by the comparison of closed-loop trajectories (in Figure \ref{fig:ex2fig3}). 
Computing the accumulated costs from $20$ initial conditions drawn by stratified random sampling (as in Example 1), the loss of performance is only $0.10\%$ on average and $0.36\%$ in the worst case. 
Therefore, Koopman--Riccati synthesis achieves near-optimal performance, while a closed-loop simulation over $6$ minutes takes only $16.4$ milliseconds on average, compared with $2887.6$ milliseconds ($\sim 2.9$ seconds) for nonlinear MPC, which solves a nonlinear program at every sampling instant. This acceleration for 2 orders of magnitude is due to the merit of direct controller synthesis. 

\begin{figure}[!t]
    \centering
    \includegraphics[width=\textwidth]{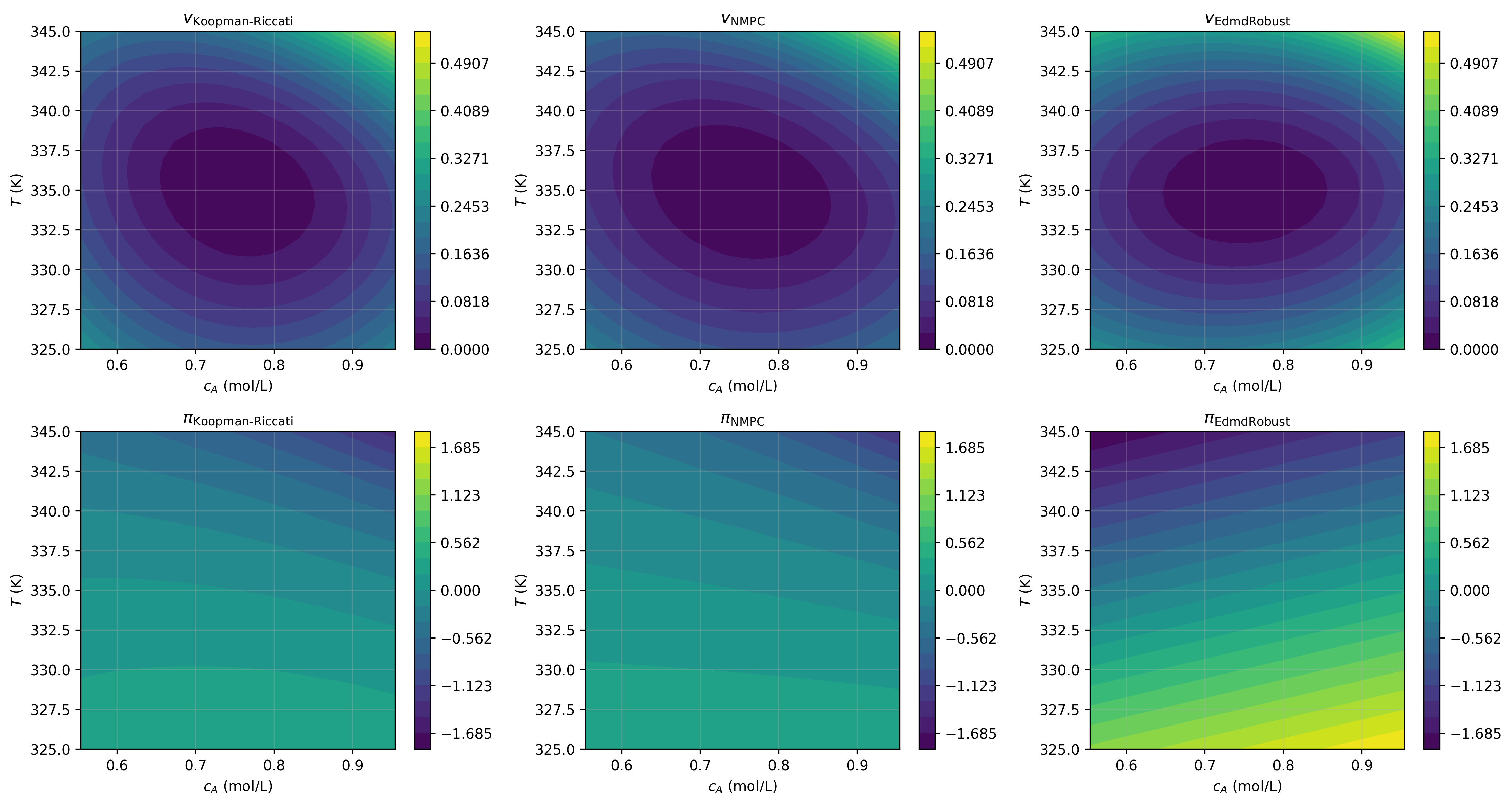}
    \caption{Comparison of cost function and policy profiles for Example 2.}
    \label{fig:ex2fig2}  
    \includegraphics[width=0.35\textwidth]{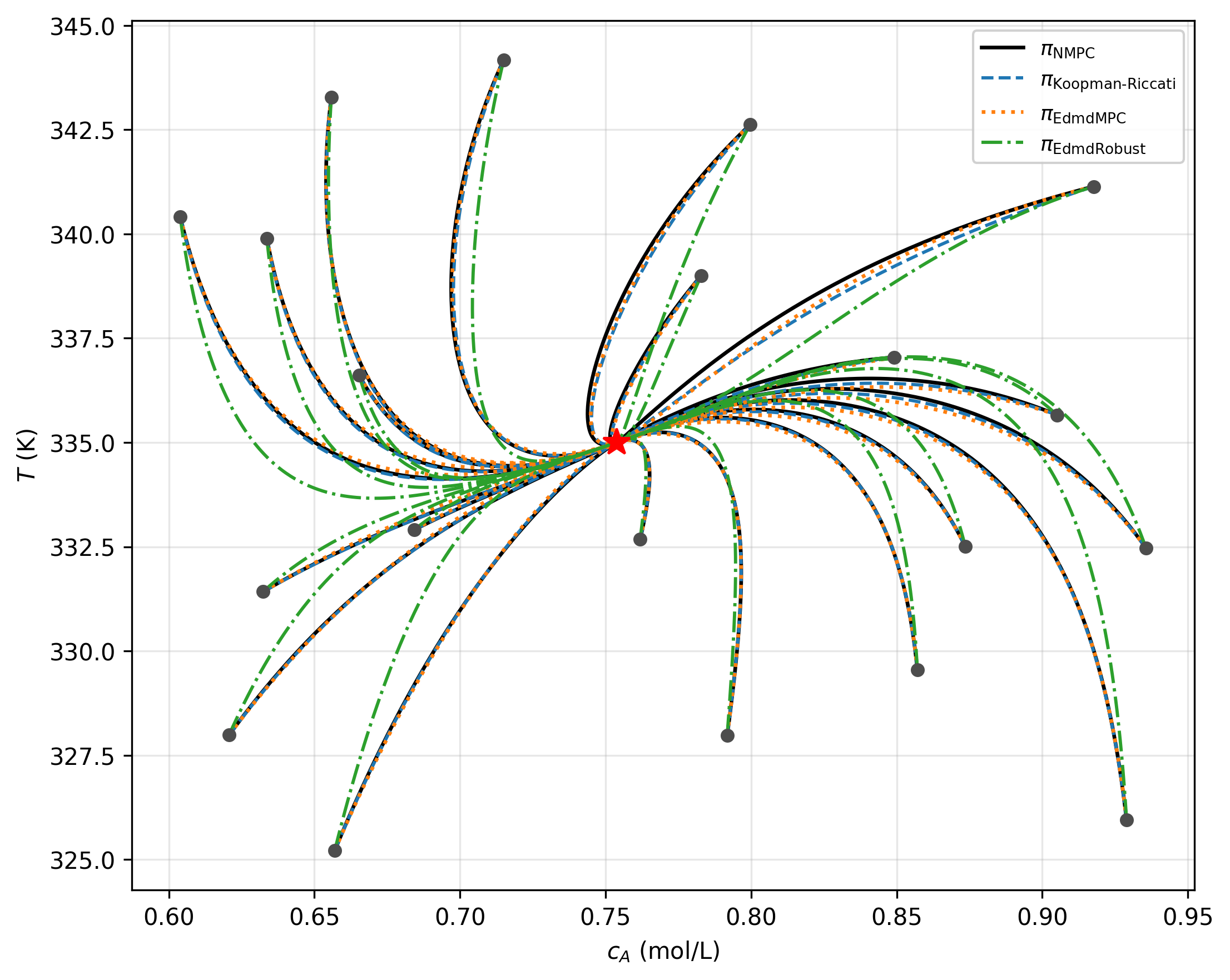}
    \caption{Comparison of closed-loop trajectories for Example 2.}
    \label{fig:ex2fig3}
\end{figure}

\begin{table}[!t]
  \centering
  \caption{Closed-loop cost relative to nonlinear MPC with the true model (mean and worst case over $20$ initial conditions) and mean computational time of one closed-loop simulation in $6$ min for Example 2.}
  \label{tab:ex2comparison}
  \begin{tabular}{lrrr}
    \toprule
    Controller & \multicolumn{2}{c}{Cost} & Time (ms) \\
    \cmidrule(lr){2-3}
     & mean & worst & \\
    \midrule
    Koopman--Riccati & $+0.10\%$ & $+0.36\%$ & $16.4$ \\
    EDMD--Robust ($c_r = 0.3$, degree 1) & $+97.84\%$ & $+179.85\%$ & $21.7$ \\
    EDMD--Robust ($c_r = 0.77$, degree 1) & $+63.34\%$ & $+122.75\%$ & $24.0$ \\
    EDMD--MPC (degree 2) & $+0.50\%$ & $+1.46\%$ & $10374.4$ \\
    EDMD--MPC (degree 3) & $+0.05\%$ & $+0.13\%$ & $30569.4$ \\
    EDMD--LQR (degree 2) & $+0.44\%$ & $+1.11\%$ & $9.6$ \\
    EDMD--LQR (degree 3) & $+0.08\%$ & $+1.10\%$ & $9.9$ \\
    \midrule
    NMPC (reference) & --- & --- & $2887.6$ \\
    \bottomrule
  \end{tabular}
\end{table}

\paragraph{Comparison to EDMD-based approaches.} Using a finite number of dictionary functions (here, monomial functions by convention) to lift the state coordinate $z=\psi(x)$, it is possible to use EDMD to identify a bilinear surrogate model $\dot z = Az + Bu + \sum_{j=1}^{d_u} u_jB_j z$. Hence, we apply EDMD with monomials of degree up to $3$ and first test whether an MPC based on the EDMD model performs well. 
As seen in the closed-loop trajectories (Figure \ref{fig:ex2fig3}) and in Table \ref{tab:ex2comparison}, the EDMD-based bilinear MPC is also near-optimal. With monomials of degree up to $2$, its performance is slightly inferior to that of the kernel-based Koopman--Riccati controller; with degree $3$, it is slightly better, but at a computational cost about three orders of magnitude higher. 
In fact, in this case, the nonlinearity is mild and even an LQR controller based on the linear part of the bilinear model would perform near-optimally. 
\par However, the bilinear EDMD model that is theoretically well justified does not give an effective direct synthesis formulation. In the robust control formulation of Str{\"a}sser et al. \cite{strasser2026overview}, the bilinear model is considered to have a mismatch term from the true dynamics: $|r(x, u)| \leq c_r (|\psi(x)|+|u|)$, and the controller synthesis is to maximize $\log\det P$ where $P\succeq0$ is such that $\BRA{x:\psi(x)^\top P^{-1}\psi(x)\leq 1}$ is certified to be a region of attraction. 
For this example, this EDMD-based robust control approach generally fails; only for a maximum degree of $1$ was it possible to obtain a feasible status from the solver and find a nontrivial controller (i.e., $K\neq 0$), and in this case, the controller is far from optimal.\footnote{The synthesis maximizes $\log\det P$ only, and the gain-related decision variables do not enter the objective. Hence the gain $K$ attaining the optimal $P$ is not unique, and the one returned depends on the numerical path of the SDP solver; the reported gains and closed-loop costs of EDMD--Robust may therefore differ across solver versions and computing platforms.} In addition, there is an essential violation of the theory --- from the sample data of $\psi(x)$ and $u$, it is estimated that $c_r \geq 2.8$, but under such high values of $c_r$, the synthesis problem is infeasible. The two designs in Table \ref{tab:ex2comparison} use $c_r = 0.3$ and $c_r = 0.77$ (the largest value for which the problem is feasible); both are below the estimated $c_r$, and hence neither is certified. 
As a result, among the direct synthesis approaches, the Koopman--Riccati method produces the best performance (by the worst-case cost degradation).

\subsection{A Higher-Dimensional Example}
As the last example we consider a polymerization reactor \cite{daoutidis1990feedforward}. 
The model is written as: 
$$\begin{aligned}
\frac{\mathrm{d}C_{\mathrm{m}}}{\mathrm{d}t} &= -\big[k_{\mathrm{P}}(T) + k_{\mathrm{fm}}(T)\big]\, C_{\mathrm{m}}\, P_0(C_{\mathrm{I}}, T) + \frac{F\,(C_{\mathrm{m,in}} - C_{\mathrm{m}})}{V}, \\
\frac{\mathrm{d}C_{\mathrm{I}}}{\mathrm{d}t} &= -k_{\mathrm{I}}(T)\, C_{\mathrm{I}} + \frac{F_{\mathrm{I}}\, C_{\mathrm{I,in}} - F\, C_{\mathrm{I}}}{V}, \\
\frac{\mathrm{d}T}{\mathrm{d}t} &= k_{\mathrm{P}}(T)\, C_{\mathrm{m}}\, \frac{(-\Delta H_{\mathrm{P}})}{\rho\, c_{\mathrm{p}}}\, P_0(C_{\mathrm{I}}, T) - \frac{U A}{\rho\, c_{\mathrm{p}}\, V}\,(T - T_{\mathrm{j}}) + \frac{F\,(T_{\mathrm{in}} - T)}{V}, \\
\frac{\mathrm{d}D_0}{\mathrm{d}t} &= \big[0.5\, k_{\mathrm{Tc}}(T) + k_{\mathrm{Td}}(T)\big]\, P_0(C_{\mathrm{I}}, T)^2 + k_{\mathrm{fm}}(T)\, C_{\mathrm{m}}\, P_0(C_{\mathrm{I}}, T) - \frac{F\, D_0}{V}, \\
\frac{\mathrm{d}D_1}{\mathrm{d}t} &= M_{\mathrm{m}}\, \big[k_{\mathrm{P}}(T) + k_{\mathrm{fm}}(T)\big]\, C_{\mathrm{m}}\, P_0(C_{\mathrm{I}}, T) - \frac{F\, D_1}{V}, \\
\frac{\mathrm{d}T_{\mathrm{j}}}{\mathrm{d}t} &= \frac{F_{\mathrm{cw}}}{V_0}\,(T_{\mathrm{w0}} - T_{\mathrm{j}}) + \frac{U A}{\rho_{\mathrm{w}}\, c_{\mathrm{w}}\, V_0}\,(T - T_{\mathrm{j}}),
\end{aligned}$$
where the rate constants follow the Arrhenius law, $k_{\ell}(T) = Z_{\ell} \exp(-E_{\ell}/RT)$ for reactions $\ell \in \{\mathrm{Tc}, \mathrm{Td}, \mathrm{I}, \mathrm{P}, \mathrm{fm}\}$, and the concentration of polymer chains is given by the quasi-steady-state approximation:
\begin{equation*}
P_0(C_{\mathrm{I}}, T) = \sqrt \frac{2 f^{\ast} C_{\mathrm{I}}\, k_{\mathrm{I}}(T)}{k_{\mathrm{Td}}(T) + k_{\mathrm{Tc}}(T)}.
\end{equation*}
The states include $(C_\mr{m},\,C_\mr{I},\,T,\,D_0,\,D_1,\,T_\mr{j})$, the monomer concentration, initiator concentration, temperature, polymer molar concentration, polymer mass concentration, and cooling jacket temperature. 
The inputs include the initiator flow rate and cooling water flow rate $(F_\mr{I},\,F_\mr{cw})$. 

\begin{table}[!t]
  \centering
  \caption{Parameters and steady state of Example 3.}
  \label{tab:ex3parameters}
  \begin{tabular}{llr}
    \toprule
    Parameter & Description & Value \\
    \midrule
    \multicolumn{3}{l}{\emph{Kinetics}} \\
    $Z_{\mathrm{Tc}}$, $E_{\mathrm{Tc}}$ & termination by coupling & $3.8223 \times 10^{10}$ m$^3$/(kmol\,h), $2.9442 \times 10^{3}$ kJ/kmol \\
    $Z_{\mathrm{Td}}$, $E_{\mathrm{Td}}$ & termination by disproportionation & $3.1457 \times 10^{11}$ m$^3$/(kmol\,h), $2.9442 \times 10^{3}$ kJ/kmol \\
    $Z_{\mathrm{I}}$, $E_{\mathrm{I}}$ & initiation & $3.7920 \times 10^{18}$ h$^{-1}$, $1.2877 \times 10^{5}$ kJ/kmol \\
    $Z_{\mathrm{P}}$, $E_{\mathrm{P}}$ & propagation & $1.7700 \times 10^{9}$ m$^3$/(kmol\,h), $1.8283 \times 10^{4}$ kJ/kmol \\
    $Z_{\mathrm{fm}}$, $E_{\mathrm{fm}}$ & chain transfer to monomer & $1.0067 \times 10^{15}$ m$^3$/(kmol\,h), $7.4478 \times 10^{4}$ kJ/kmol \\
    $f^{\ast}$ & initiator efficiency & $0.58$ \\
    $-\Delta H_{\mathrm{P}}$ & heat of propagation & $57\,800$ kJ/kmol \\
    $R$ & ideal-gas constant & $8.314$ kJ/(kmol\,K) \\ 
    \midrule
    \multicolumn{3}{l}{\emph{Reactor and cooling system}} \\
    $F$ & monomer volumetric flow rate & $1.00$ m$^3$/h \\
    $V$ & reactor volume & $0.1$ m$^3$ \\
    $\rho$ & density of the reacting mixture & $866$ kg/m$^3$ \\
    $c_{\mathrm{p}}$ & heat capacity of the reacting mixture & $2.0$ kJ/(kg\,K) \\
    $C_{\mathrm{I,in}}$ & initiator concentration in the initiator feed & $8.0$ kmol/m$^3$ \\
    $M_{\mathrm{m}}$ & molecular weight of the monomer & $100.12$ kg/kmol \\
    $U$ & overall heat-transfer coefficient & $720$ kJ/(h\,K\,m$^2$) \\
    $A$ & heat-transfer area & $2.0$ m$^2$ \\
    $\rho_{\mathrm{w}}$ & density of the cooling water & $1000$ kg/m$^3$ \\
    $c_{\mathrm{w}}$ & heat capacity of the cooling water & $4.2$ kJ/(kg\,K) \\
    $V_0$ & effective volume of the cooling jacket & $0.02$ m$^3$ \\
    $T_{\mathrm{w0}}$ & inlet temperature of the cooling water & $293.2$ K \\
    $C_{\mathrm{m,in}}$ & monomer concentration in the feed & $6.0$ kmol/m$^3$ \\
    $T_{\mathrm{in}}$ & feed temperature & $350$ K \\
    \midrule
    \multicolumn{3}{l}{\emph{Nominal inputs and setpoints}} \\
    $F_{\mathrm{I}}$ & initiator flow rate, nominal & $0.05392$ m$^3$/h \\
    $F_{\mathrm{cw}}$ & cooling-water flow rate, nominal & $3.2582$ m$^3$/h \\
    $y_1^{\mathrm{sp}}$ & setpoint of $D_1/D_0$ & $25\,000$ kg/kmol \\
    $y_2^{\mathrm{sp}}$ & setpoint of $T$ & $335$ K \\
    \bottomrule
  \end{tabular}
\end{table}

\paragraph{Data collection and processing.} 
The state space is no longer a very low-dimensional one. For such cases, it is not effective to use random uniform sampling in a hypercube as we did for the previous two examples. Instead, a random binary sequence is applied to the inputs, at the scale of $30\%$ of their steady-state values, with a dwell time of $0.3$ minutes ($1/20$ of the residence time of the reactor) to excite the system around the equilibrium. 
Due to the open-loop stability, this excitation results in a stationary process of the state; from its stationary distribution, a sample of size $n=300$ is collected (see Figure \ref{fig:ex3fig1}). 
We note that the sixth state is in fact redundant due to $D_1 = M_\mr{m}(C_\mr{m,in}-C_\mr{m})$, and hence we consider the system as one with $d_x = 5$ and discard $D_1$. Hence, we use $s=4.5$ which is greater than $d_x/2=2.5$. 
The five state components are rescaled by $C_\mr{m}$: 0.1 kmol/m$^3$, $C_\mr{I}$: 0.08 kmol/m$^3$, $T$: 2 K, $D_0$: $9 \times 10^{-4}$ kmol/m$^3$, $T_\mr{j}$: 2 K. 
\begin{figure}[!t]
    \centering
    \includegraphics[width=0.75\linewidth]{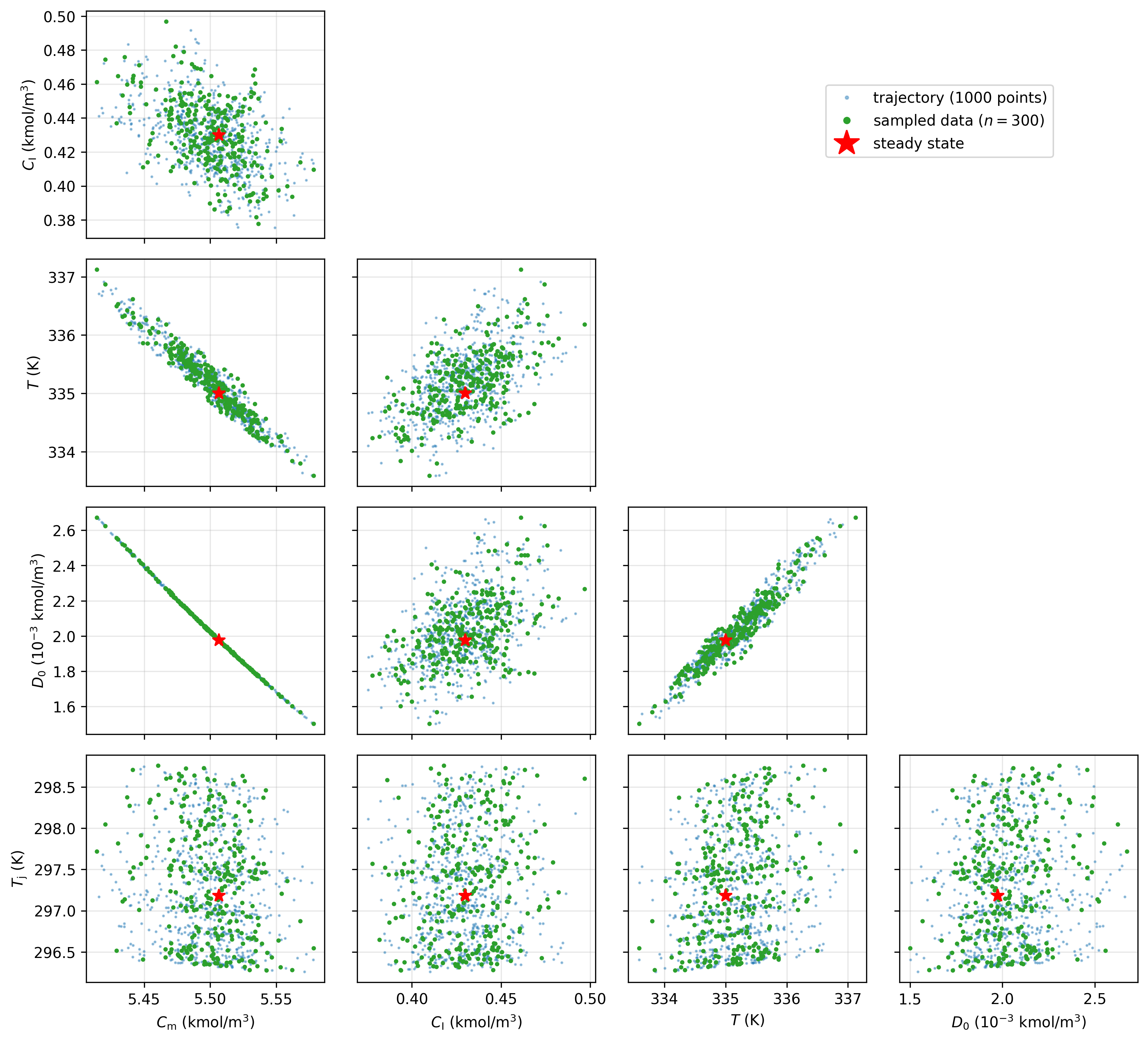}
    \caption{Stationary distribution of data points under excitation, used for sampling in Example 3.}
    \label{fig:ex3fig1}
\end{figure}

\paragraph{Optimal controller synthesis.} The parameter values are given in Table \ref{tab:ex3parameters}. For control, we consider $y_1=D_1/D_0$ (number-average molecular weight of the polymer) and $y_2=T$ (temperature) as two controlled variables to be regulated at their setpoints, and hence the optimal control problem defines $q(x) = (y_1 - y_1^\mr{sp})^2 / (y_1^\mr{scale})^2 + (y_2 - y_2^\mr{sp})^2 / (y_2^\mr{scale})^2$. Here the scales $y_1^\mr{scale} = 2500$ kg/kmol (which is 10\% of the setpoint) and $y_2^\mr{scale} = 2$ K. The input scales for the two input components are chosen to be $20\%$ of their steady-state values.  
With these settings, the Koopman--Riccati controller is synthesized and compared to nonlinear MPC in closed-loop simulations for $1.5$ hours under 20 initial conditions. The results are reported in Table \ref{tab:ex3comparison}, which again shows that Koopman--Riccati is near-optimal. (For this specific example, linearized LQR performs well too, since the state region covered by data is not wide enough to exhibit severe nonlinearity. On the other hand, linearized LQR still requires the full model, and it can be difficult to tell \textit{a priori} whether nonlinearity is strong.)

\paragraph{Comparison to kernel-based Koopman MPC.} We also compare against Koopman MPC with a kernel-based surrogate model. Following the kernel EDMD approach of \cite{bold2025kernel}, a discrete-time control-affine surrogate $x^+ = \hat{f}_0(x) + \sum_{j=1}^{d_u} u_j \hat{f}_j(x)$ over the sampling interval of $0.3$ minutes is learned from one-step simulations of the reactor. The centers are the $300$ data points used above together with the origin, and at each center the system is simulated under the $3$ inputs $u = (0,0)$, $u = (\delta,0)$, and $u = (0,\delta)$ with $\delta = 0.25$, so that $\hat{f}_0$ and $\hat{f}_j$ are identified exactly at the centers and then interpolated. 
In the \emph{Sobolev} variant, both $\hat{f}_0$ and $\hat{f}_j$ are interpolated with a Wendland kernel ($k = 1$, $\ell=8$). In the \emph{linear--Sobolev} variant, $\hat{f}_0$ is instead interpolated with the corresponding linear--Sobolev kernel, in line with the requirement $f_0 \in \rg C^s$, while $\hat{f}_j$ is interpolated with the Wendland kernel as before. 
On $500$ held-out states, the mean one-step prediction errors are $6.2 \times 10^{-3}$ and $5.7 \times 10^{-3}$ (relative to a mean step of $0.03$) for the two surrogates, respectively. Each surrogate is used as the predictor in MPC with the same horizon ($2$ hours) and stage cost as the nonlinear MPC and without terminal cost, as in \cite{schimperna2026data}.

As shown in Table \ref{tab:ex3comparison}, both Koopman MPC controllers are near-optimal, with the linear--Sobolev variant (mean $+0.02\%$, worst $+0.05\%$) more accurate than the Sobolev variant (mean $+0.14\%$, worst $+0.43\%$) due to the better approximation of the drift near the equilibrium. However, they require solving a large-scale nonlinear program at every sampling instant, and each closed-loop simulation takes about $130$--$150$ seconds, about three orders of magnitude slower than Koopman--Riccati and even $30$--$40$ times longer than with the nonlinear MPC using the true model (whose dynamics are integrated much more cheaply than the kernel surrogate with $n=300$ comprising terms).
\begin{table}[!t]
  \centering
  \caption{Closed-loop cost relative to nonlinear MPC with the true model (mean and worst case over $20$ initial conditions) and mean computational time of one closed-loop simulation on $[0, 1.5]$ h for Example 3.}
  \label{tab:ex3comparison}
  \begin{tabular}{lrrr}
    \toprule
    Controller & \multicolumn{2}{c}{Cost} & Time (ms) \\
    \cmidrule(lr){2-3}
     & mean & worst & \\
    \midrule
    Koopman--Riccati & $0.00\%$ & $+0.08\%$ & $170.6$ \\
    Koopman MPC (Sobolev) & $+0.14\%$ & $+0.43\%$ & $147934.0$ \\
    Koopman MPC (linear--Sobolev) & $+0.02\%$ & $+0.05\%$ & $127995.5$ \\
    Linearized LQR & $0.00\%$ & $+0.09\%$ & $34.9$ \\
    Open loop & $+82.93\%$ & $+163.46\%$ & $28.8$ \\
    \midrule
    Nonlinear MPC (reference) & --- & --- & $3722.8$ \\
    \bottomrule
  \end{tabular}
\end{table}

The closed-loop trajectories starting from a few initial conditions are shown in Figure \ref{fig:ex3fig2}. The differences between the Koopman--Riccati control policy and nonlinear MPC are minor, while the computational time is highly different. 
The near-optimal performance of Koopman--Riccati is further confirmed by a disturbance rejection test, where $T_\mr{in}$ is randomly perturbed by a disturbance uniformly distributed on $\pm 5$ K with a dwell time of $0.3$ minutes. As seen in Figure \ref{fig:ex3fig3}, the closed-loop trajectory nearly coincides with that under nonlinear MPC. Over a $5$-hour simulation, the time-averaged cost under Koopman--Riccati is $1.3\%$ lower than that under nonlinear MPC (since the continuous-time feedback reacts to each change of $T_\mr{in}$ immediately, whereas MPC holds its input over each sampling interval), and those under the Sobolev and linear--Sobolev Koopman MPC are $0.5\%$ and $0.1\%$ higher, respectively, while the computational time is $2.8$ seconds for Koopman--Riccati, $36$ seconds for nonlinear MPC, and about $870$ seconds for each Koopman MPC.  
\begin{figure}[!t]
    \centering
    \includegraphics[width=0.5\linewidth]{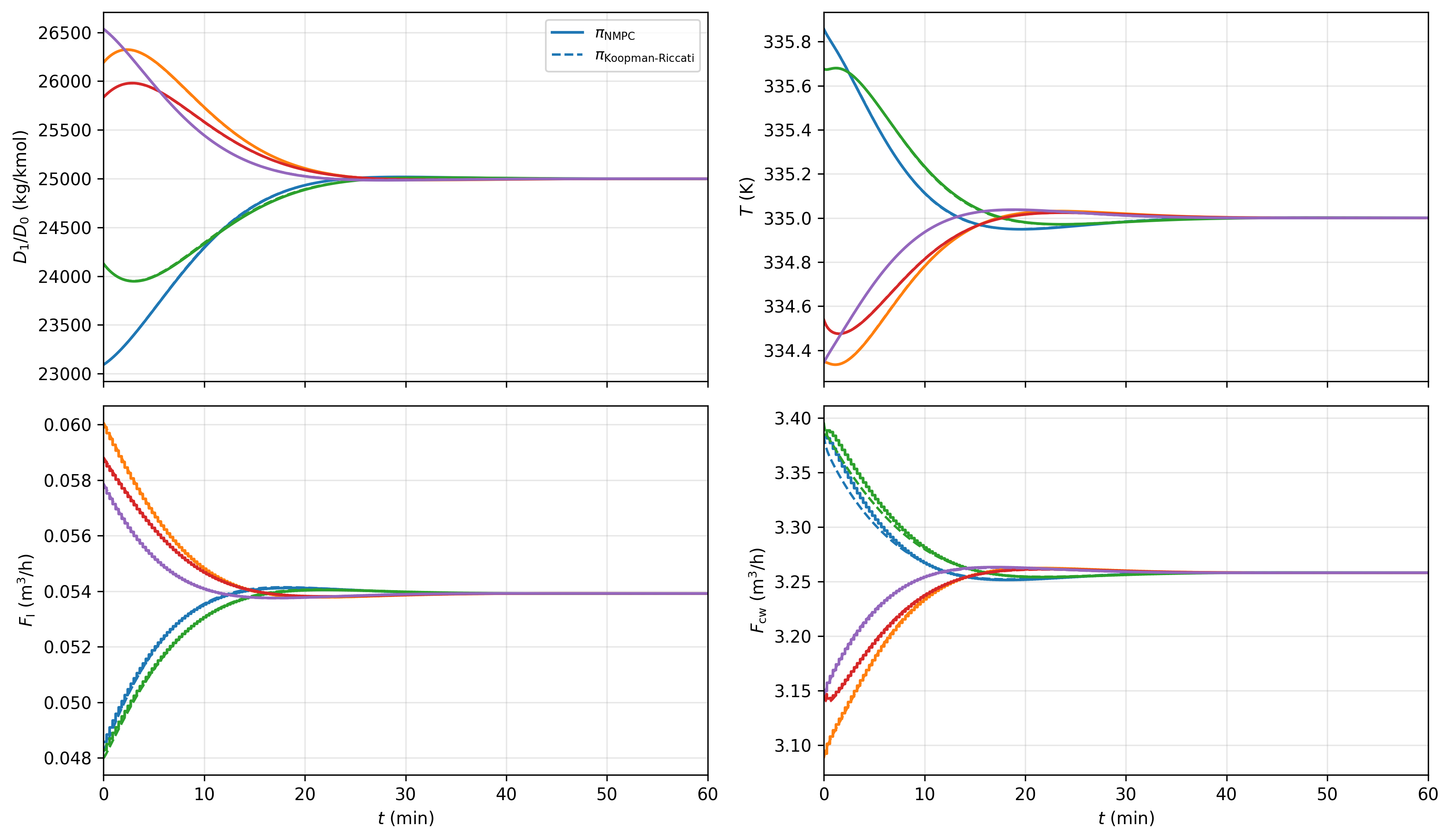}
    \caption{Closed-loop trajectories under different initial conditions for Example 3.}
    \label{fig:ex3fig2} 
    \includegraphics[width=0.5\linewidth]{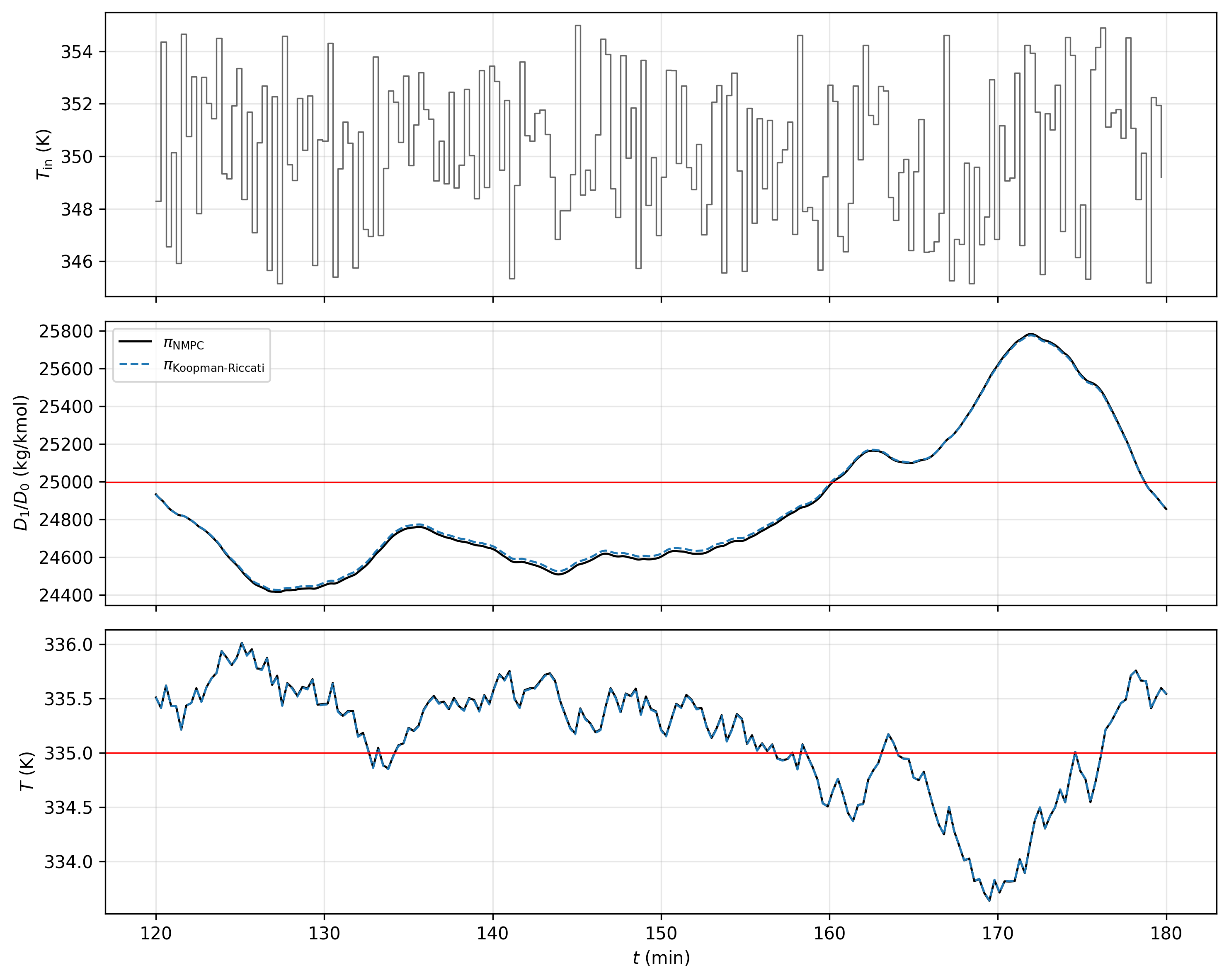}
    \caption{Closed-loop trajectories under inlet temperature disturbances for Example 3.}
    \label{fig:ex3fig3}
\end{figure}

\section{Conclusions and Discussions}
\paragraph{Conclusion of this paper.} We conclude that for continuous-time input-affine systems, provided that their nonlinearity belongs to a class of sufficient smoothness, the problems of stabilizing controller synthesis and optimal controller synthesis (unconstrained, with a quadratic cost on input) can be solved with kernel-based methods that result in formulations analogous to finite-dimensional linear systems. 
The kernel adopted here is the linear--Sobolev kernel that defines an RKHS comprising state-dependent functions that are locally at least linear and globally Sobolev-regular, and the kernel quadratic/SOS forms on this RKHS can be used for capturing the control Lyapunov function and its decay rate, or the control costs. 
For computational tractability, Yosida approximation is used to approximate unbounded infinitesimal operators of semigroups as bounded operators, and the stabilization constraints or optimality conditions are imposed on the finite sample. 
Practically, compared to existing methods for Koopman-based control and Koopman-based controller synthesis, the proposed Koopman--Lyapunov and Koopman--Riccati synthesis methods are computationally advantageous, avoid repetitive online computation and large-scale nonlinear programming as in nonlinear MPC or kernel-based MPC, and do not suffer from unnecessary complications of bilinearity. 
This work, therefore, provides a relatively complete answer to the program envisioned by the pioneer researchers on Koopman operators for control --- \emph{using ``linear predictors'' for nonlinear control}. 

\paragraph{Future directions.} The present work purposefully restricts the scope to continuous-time input-affine systems (for which the operator models do not involve a ``lifting'' of inputs), and for the optimal control problem, requires a cost that involves a quadratic cost on the inputs. In such a problem, the relation of the optimal cost and the optimal policy, under the Yosida approximation of the Perron--Frobenius operators, becomes explicit, and hence the classical Riccati structure is recovered in the form of a kernel method. 
This structure is no longer observed in two more complex settings. (i) When the input is \emph{constrained}, the optimality condition involves an inner layer optimization on inputs, whose solution is nonsmooth even in the simplest case of box-constrained inputs. With a nonsmooth policy, the synthesis must rely on either a different approach from the RKHS or an approximation scheme by a sequence of smooth policies. (ii) When the state is also \emph{constrained}, e.g., in safety-critical systems where the state is forbidden to enter an assigned ``unsafe region'', even the cost function is nonsmooth. 
It is therefore of interest to see if related problems, such as estimating the domain of null controllability or synthesizing a control barrier function, can be formulated and solved efficiently in the same spirit as this paper. 

\paragraph{Open problems.} In the nonlinear control literature, the controllability concepts were typically characterized through differential geometry \cite{jurdjevic1997geometric}. Revealing the connections between geometric controllability or stabilizability and the existence of a Koopman semigroup-based controller would be a meaningful addition to the current Koopman operators research. 
Moreover, observer synthesis for autonomous nonlinear systems and controller synthesis for state-measurable systems can be addressed separately; although a generic separation principle is too difficult to establish, it would be of interest to clarify when and whether a separately designed observer and controller can be combined. 
If the state vector is indeed unmeasurable, the Koopman operator model itself can only be identified based on input and output data, instead of on the state data. The author believes that this problem is of particular importance to control engineers. 

\section*{Acknowledgments}
The work is supported by NSF CBET Award \#2414369 and CAREER Award \#2543029. The author acknowledges the use of Claude Code (Anthropic), powered by Claude Opus 5.5, for assistance with code development, hyperparameter tuning, and model validation in this work. All AI-generated code was reviewed, tested, and validated by the author, who takes full responsibility for its correctness. 

\bibliographystyle{ieeetr}
\bibliography{mybib.bib}
\end{document}